\documentclass[a4paper,reqno,9pt]{amsart}

\usepackage{amsmath}
\usepackage{amsthm}
\usepackage{amssymb}
\usepackage{amscd} 
\usepackage{palatino}

\usepackage{bbm} 
\usepackage{mathtools,mathrsfs} 

\usepackage{verbatim} 

\usepackage{enumitem}

\newtheoremstyle{newremark}
  {5pt}
  {5pt}
  {\rmfamily}
  {}
  {\rmfamily\bf}
  {.}
  {.5em}
  {}

\newtheorem{theorem}{Theorem}
\newtheorem{lemma}[theorem]{Lemma}
\newtheorem{corollary}[theorem]{Corollary}
\newtheorem{proposition}[theorem]{Proposition}

\theoremstyle{newremark}
\newtheorem{remark}[theorem]{Remark}
\newtheorem{definition}[theorem]{Definition}

\newtheorem*{definition*}{Definition} 
\newtheorem*{notations*}{Notations}

\numberwithin{theorem}{section}
\numberwithin{equation}{section}

\newcommand{\N}{\mathbb{N}} 
\newcommand{\R}{\mathbb{R}} 

\newcommand{\Z}{\mathbb{Z}} 
\newcommand{\C}{\mathbb{C}} 
\newcommand{\D}{\mathbb{D}} 

\def\XXint#1#2#3{{%
\setbox0=\hbox{$#1{#2#3}{\int}$}
\vcenter{\hbox{$#2#3$}}\kern-.5\wd0}}

\renewcommand{\leq}{\leqslant}
\renewcommand{\geq}{\geqslant}
\renewcommand{\subset}{\subseteq}

\newcommand{\LL}{\mathop{\hbox{\vrule height 6pt width .5pt depth 0pt
\vrule height .5pt width 3pt depth 0pt}}\nolimits}

\newcommand{\Om}{\Omega}
\newcommand{\eps}{\varepsilon}
\newcommand{\tom}{\widetilde\Omega}
\newcommand{\e}{{\rm e}}
\newcommand{\V}{\mathbf{V}^*_{n-1}(\tom)}

\newcommand{\de}{{\rm d}}

\begin{document}


\title[On a fractional Ginzburg-Landau equation]{ON A FRACTIONAL GINZBURG-LANDAU~EQUATION\\ 
AND 1/2-HARMONIC MAPS INTO SPHERES}

\author{Vincent Millot}
\address{Vincent Millot\\ Universit\'e Paris Diderot - Paris 7\\
 Laboratoire J.-L. Lions (CNRS UMR 7598)\\ Paris, France}
\email{millot@ljll.univ-paris-diderot.fr}

\author{Yannick Sire}
\address{Yannick Sire\\ Universit\'e Aix-Marseille \\  
Laboratoire d'Analyse, Topologie, Probabilit\'es (CNRS UMR 7353)\\ Marseille, France}
\email{sire@cmi.univ-mrs.fr}


\begin{abstract}
This paper is devoted to the asymptotic analysis of a fractional version of the Ginzburg-Landau equation in bounded domains, where the Laplacian is replaced by an integro-differential operator related to the square root Laplacian as defined in Fourier space. In the singular limit $\eps\to0$, we show that solutions with uniformly bounded energy converge weakly  to sphere valued 1/2-harmonic maps, {\it i.e.}, the fractional analogues of the usual harmonic maps. 
In addition, the convergence holds in smooth functions spaces away  from a countably $\mathscr{H}^{n-1}$-rectifiable closed  set of finite $(n-1)$-Hausdorff measure. The proof relies on the representation of the square root Laplacian as a Dirichlet-to-Neumann operator in one more dimension, and on the analysis of a boundary version of the Ginzburg-Landau equation.  Besides the analysis of the fractional Ginzburg-Landau equation, we also give a general partial regularity result for stationary 1/2-harmonic maps in arbitrary dimension.
\vskip10pt

\end{abstract}

\maketitle


\section{Introduction}

Let $n\geq 1$ and $m\geq 2$ be given integers, and let $\omega\subset \R^n$ be a smooth bounded open set. In this paper we investigate the asymptotic behavior, as $\varepsilon \downarrow 0$, of weak solutions $v_\varepsilon:\R^n\to\R^m$ of the fractional Ginzburg-Landau equation 
\begin{equation}\label{GLIntro}
 (-\Delta)^{\frac{1}{2}} v_\varepsilon= \frac{1}{\varepsilon}(1-|v_\varepsilon|^2)v_\varepsilon \qquad \text{in $\omega$}\,,
 \end{equation}
possibly subject to an exterior Dirichlet condition
\begin{equation}\label{Dircondintro}
v_\eps= g \qquad\text{on $\R^n\setminus\omega$}\,,
\end{equation}
where $g:\R^n\to\R^m$ is a smooth function satisfying $|g|=1$ in $\R^n\setminus\omega$. The action of the integro-differential operator $ (-\Delta)^{\frac{1}{2}} $ on a smooth bounded function $v:\R^n\to\R^m$ is defined   by 
$$ (-\Delta)^{\frac{1}{2}}  v(x):={\rm p.v.} \left(\gamma_n\int_{\mathbb{R}^n}\frac{v(x)-v(y)}{|x-y|^{n+1}}\,\de y\right) \quad \text{with } \gamma_n:=\frac{\Gamma\big(\frac{n+1}{2}\big)}{\pi^{\frac{n+1}{2}}} \,,$$
and the notation ${\rm p.v.}$ means that the integral  is taken in the {\it Cauchy principal value} sense. 
In Fourier space\footnote{we consider here the ordinary frequency Fourier transform $v\mapsto \hat v$ given by
$\hat v(\xi):=\int_{\R^n}v(x) e^{-2 i \pi x\cdot\xi}\,\de\xi$.

}, this operator has symbol $2\pi|\xi|$, while $4\pi^2|\xi|^2$ is the symbol of $-\Delta$. In particular, $ (-\Delta)^{\frac{1}{2}} $ is the square root of the standard Laplacian when acting on the Schwartz class $\mathscr{S}(\R^n)$.

The weak sense for equation \eqref{GLIntro} will be understood  variationaly using the distributional formulation of the fractional Laplacian in the open set $\omega$.  More precisely, we define the action of $ (-\Delta)^{\frac{1}{2}}  v$ on an element $\varphi\in\mathscr{D}(\omega;\R^m)$ by setting
\begin{multline*}
\big\langle   (-\Delta)^{\frac{1}{2}}  v,\varphi\big\rangle_\omega:= \frac{\gamma_n}{2} \iint_{\omega\times\omega}\frac{\big(v(x)-v(y)\big)\cdot\big(\varphi(x)-\varphi(y)\big)}{|x-y|^{n+1}}\,\de x\de y\\
+\gamma_n \iint_{\omega\times(\R^n\setminus\omega)}\frac{\big(v(x)-v(y)\big)\cdot\big(\varphi(x)-\varphi(y)\big)}{|x-y|^{n+1}}\,\de x\de y\,.
 \end{multline*}
This formula turns out to define a distribution on $\omega$ whenever $v\in L^2_{\rm loc}(\R^n;\R^m)$ satisfies
\begin{equation}\label{defenergE}
\mathcal{E}(v,\omega):=\frac{\gamma_n}{4}  \iint_{\omega\times\omega}\frac{|v(x)-v(y)|^2}{|x-y|^{n+1}}\,\de x\de y
+\frac{\gamma_n}{2} \iint_{\omega\times(\R^n\setminus\omega)} \frac{|v(x)-v(y)|^2}{|x-y|^{n+1}}\,\de x\de y<\infty\,.
\end{equation}
In such a case we say that $v$ is admissible, and $ (-\Delta)^{\frac{1}{2}} v$ belongs to $H^{-1/2}(\omega;\R^m)$, {\it i.e.}, the topological dual space of $H^{1/2}_{00}(\omega;\R^m)$ (functions in $H^{1/2}(\R^n;\R^m)$ vanishing outside~$\omega$). When prescribing the exterior Dirichlet condition \eqref{Dircondintro}, the class of admissible functions  is given by the affine space $H^{1/2}_g(\R^n;\R^m):=g+H^{1/2}_{00}(\omega;\R^m)$. Note that weak and strong definitions of $ (-\Delta)^{\frac{1}{2}} $ are consistent, {\it i.e.}, if $v$ is smooth and bounded, then $v$ is admissible, and the pointwise and distributional definitions of $ (-\Delta)^{\frac{1}{2}} v$ agree, see {\it e.g.} \cite{SerV}.  
Then observe that the derivative at $v$ of $\mathcal{E}(\cdot,\omega)$ in the direction $\varphi\in \mathscr{D}(\omega;\R^m)$ is precisely given by  $\big\langle   (-\Delta)^{\frac{1}{2}}  v,\varphi\big\rangle_\omega$. Hence the energy $\mathcal{E}(\cdot,\omega)$, which is built on the $H^{1/2}$-seminorm, can be  thought as {\it Dirichlet 1/2-energy} in $\omega$ associated to the operator $ (-\Delta)^{\frac{1}{2}} $. Integrating  the potential in the right hand side of~\eqref{GLIntro} we obtain the {\it Ginzburg-Landau 1/2-energy} associated to the equation in $\omega$, 
\begin{equation}\label{defFGLenerg}
\mathcal{E}_\varepsilon(v,\omega):=\frac{1}{2}\int_{\omega} {\rm e}(v,\omega) + \frac{1}{2\varepsilon}(1-|v|^2)^2\,\de x\,,
\end{equation}
where we have set ${\rm e}(v,\omega)$ to be the {\it nonlocal energy density in $\omega$}  given by 
\begin{equation}\label{nonlocendens}
{\rm e}\big(v,\omega\big):=\frac{\gamma_n}{2}\int_{\omega}\frac{|v(x)-v(y)|^2}{|x-y|^{n+1}}\,\de y +\gamma_n\int_{\R^n\setminus \omega}\frac{|v(x)-v(y)|^2}{|x-y|^{n+1}}\,\de y \,.
\end{equation}
We define weak solutions of \eqref{GLIntro} as critical points  of $\mathcal{E}_\eps$ with respect to perturbations supported in $\omega$. 
\vskip5pt

If $-\Delta$ is used instead of its square root and $m=1$, equation \eqref{GLIntro} reduces to the well known Allen-Cahn equation. This scalar equation appears in the van der Waals-Cahn-Hilliard theory of phase transitions in fluids. From the geometrical point of view, the Allen-Cahn equation provides  an approximation, as $\varepsilon\to 0$, of codimension one minimal surfaces, {\it i.e.}, limits of critical points are functions with values in $\{\pm1\}$, and the limiting interfaces between the two phases~$\pm 1$ are  {\it generalized minimal hypersurfaces}, see {\it e.g.} \cite{HT,Mod}. Since the work of {\sc Alberti, Bouchitt\'e, \& Seppecher} \cite{ABScr,ABS}, fractional generalizations of the Allen-Cahn equation involving $(-\Delta)^s$ with  $0<s<1$ have been investigated by many authors, see {\it e.g.} \cite{CC,CC2,CS1,CS2,CSM,Gonz,PSV,SV}, specially in the case $s=1/2$  which is related to models for dislocations in crystals \cite{GM,GonzMon} (see also \cite{DFV,DPV} for $s\not=1/2$), and to thin-film micromagnetics \cite{K1,K2}. The convergence as $\eps\to0$ of {\it minimizers} of the fractional Ginzburg-Landau {\it $s$-energy} has been  treated very recently in \cite{SV1,SV2}. In these papers, Dirichlet and Ginzburg-Landau $s$-energies are built on the $H^s$-seminorm and are defined as in \eqref{defenergE}-\eqref{defFGLenerg}-\eqref{nonlocendens} with the weight $|x-y|^{-(n+2s)}$ in place of $|x-y|^{-(n+1)}$. It is proved in \cite{SV1,SV2} that limits of minimizers  still  take values in $\{\pm1\}$, and the geometrical characterization of the associated interfaces depends on wether $s\geq 1/2$ or $s<1/2$. For $s\geq 1/2$, limiting interfaces turn out to be (usual generalized) area-minimizing minimal surfaces, while the picture completely changes for $s<1/2$, where limiting interfaces are the so-called {\it nonlocal $s$-minimal surfaces} introduced by {\sc Caffarelli, Roquejoffre, \& Savin}~\cite{CRS}. The interface induced by a map $v$ taking values in $\{\pm1\}$ is said to be a minimizing $s$-minimal surface if $v$ itself is minimizing  the Dirichlet $s$-energy in the class of maps with values  in $\{\pm1\}$. In a sense, nonlocal $s$-minimal surfaces can be  thought as  {\it  fractional $s$-harmonic maps} with values in $\mathbb{S}^0\simeq\{\pm1\}$. The dichotomy in the parameter $s$ essentially comes from the fact that characteristic functions of (smooth) sets do not belong to $H^s$ for $s\geq 1/2$, so that the $s$-energy density concentrates  near a limiting interface very much like in the local case. For $s<1/2$, the $H^s$-regularity allows for (some) characteristic functions, and the energy of minimizers remains uniformly bounded as $\eps$ tends to $0$. In this paper, we focus on the case $s=1/2$ for some specific    geometrical features that we shall explain shortly.

In the vectorial case $m=2$, the local Ginzburg-Landau equation ({\it i.e.}, with $-\Delta$ instead of its square root) has been widely studied because it shares many of the relevant features of more elaborated models of superconductivity or superfluidity, see {\it e.g.} \cite{BBH,Riv2,SandSer}. In the spirit of the classical Landau's theory of phase transtions, fractional Ginzburg-Landau equations have been recently suggested in the physics literature in order to incorporate a long-range dependence posed by a nonlocal ordering, as it might appear in certain high temperature superconducting compounds, see \cite{MR,TZ,WZ} and the references therein. 
In arbitrary dimensions $n\geq 1$ and $m\geq 2$, the local Ginzburg-Landau equation has  a geometrical flavor like in the scalar case, and the limiting system in the asymptotic $\eps\to 0$ is given by the harmonic map equation into the sphere $\mathbb{S}^{m-1}$ (or into more general manifolds according to the potential well). Harmonic maps can be seen as higher dimensional generalizations of geodesics and are defined as critical points of the Dirichlet energy with respect to perturbations in the target. If the target is a sphere, it corresponds to variations of the Dirichlet energy of the form $t\mapsto\frac{v_*+t\varphi}{|v_*+t\varphi|}$ for bounded test functions $\varphi$ and $|t|\ll1$.  It leads to 
the Euler-Lagrange equation (in the sense of distributions)
\begin{equation}\label{locharmmap}
 -\Delta v_* = |\nabla v_*|^2 v_* \quad\Longleftrightarrow \quad
-\Delta v_*\perp T_{v_*}\mathbb{S}^{m-1} \,.
\end{equation}
The equation in the right hand side of \eqref{locharmmap} arises when using test functions satisfying $\varphi(x)\in T_{v_*(x)}\mathbb{S}^{m-1}$ a.e., while the equation in the left hand side appears when $\varphi$ is chosen with no constraint.  The implication from right to left can be easily deduced (at least formally) writing $\Delta v_*=(\Delta v_*\cdot v_*) v_*$ and using the identity $0=\Delta(|v_*|^2)=2\Delta v_*\cdot v_* +2|\nabla v_*|^2$. 
A main feature of  the harmonic map system is its critical structure concerning regularity. Indeed, the source term in the left hand side equation of  \eqref{locharmmap} has {\it a priori} no better integrability than $L^1$. This is precisely the borderline case where linear elliptic regularity does not apply. In dimension $n=2$, harmonic maps are smooth by the famous result of {\sc H\'elein} \cite{Hel}. However, 
equation \eqref{locharmmap} admits highly discontinuous solutions in dimensions $n\geq 3$~\cite{Riv}, and specific assumptions are required to obtain at least partial regularity ({\it i.e.}, smoothness outside some "small set", see Subsection~\ref{regtheo} for a more detailed discussion and references).  In addition to this regularity issue, equation \eqref{locharmmap} turns out to be invariant  under composition with conformal maps in dimension $n=2$. This conformal invariance implies that sequences of harmonic maps are usually non-compact in the energy space~\cite{Lin}, and this lack of compactness is formally transferred to the local Ginzburg-Landau equation in the asymptotic $\eps\to0$.   
The weak convergence of critical points of the local Ginzburg-Landau energy towards (weakly) harmonic maps 
has been proved by {\sc Lin \&Wang} \cite{LW1,LW2,LW3} (see also \cite{CS}). In the spirit of the blow-up analysis for harmonic maps developped in \cite{Lin}, the authors have determined in~\cite{LW1,LW2,LW3} the geometrical nature of the possible defect  measure arising in the weak convergence process, as well as  the rectifiability of the energy concentration set. 
\vskip5pt

The main object of the present article is to provide a careful analysis of critical points of the Ginzburg-Landau 1/2-energy $\mathcal{E}_\eps$ as $\varepsilon\to0$ in the spirit of \cite{LW1,LW2,LW3}, and as a byproduct to extend the previous works on the scalar case to the vectorial setting.  By analogy with the local case, we expect solutions of the fractional Ginzburg-Landau equation \eqref{GLIntro} with uniformly bounded    
energy in $\eps$ to  converge weakly as $\eps\to0$ to critical points of the Dirichlet 1/2-energy under the constraint to be $\mathbb{S}^{m-1}$-valued. In other words, a weak limit $v_*$ of critical points of the Ginzburg-Landau 1/2-energy should satisfy
\begin{equation}\label{defcritsphere}
\left[\frac{\de}{\de t} \,\mathcal{E}\left(\frac{v_*+t\varphi}{|v_*+t\varphi|}\,,\omega\right)\right]_{t=0} =0\qquad\text{for all $\varphi\in\mathscr{D}(\omega;\R^m)$}\,.
\end{equation}
Computing the associated Euler-Lagrange equation, we obtain 
 \begin{equation}\label{eqhhalfHarm}
 (-\Delta)^{\frac{1}{2}} v_*\perp T_{v_*}\mathbb{S}^{m-1} \quad\text{in $\omega$}\,,
 \end{equation}
in the distributional sense. In the case where $|v_*|=1$ in $\R^n$, the Lagrange multiplier related to the constraint takes a quite simple form and leads to the equation
\begin{equation}\label{eqhhalfHarm2}
 (-\Delta)^{\frac{1}{2}} v_*(x)=\left(\frac{\gamma_n}{2}\int_{\R^n}\frac{|v_*(x)-v_*(y)|^2}{|x-y|^{n+1}}\,\de y\right)v_*(x)\quad\text{in $\omega$}\,,
  \end{equation}
which is in clear analogy with \eqref{locharmmap}. We shall refer to as {\it weak 1/2-harmonic map into $\mathbb{S}^{m-1}$ in $\omega$}
 any map taking values into $\mathbb{S}^{m-1}$ in $\omega$ and satisfying condition  \eqref{defcritsphere}. 
 
The notion of 1/2-harmonic map into a manifold has been introduced in the case $\omega=\R$ by {\sc Da Lio \& Rivi\`ere}~\cite{DaLRi, DaLRi2}. More precisely, weak 1/2-harmonic maps from $\R$ into $\mathbb{S}^{m-1}$ are defined in \cite{DaLRi} as critical points of the {\it line energy} 
$$\mathcal{L}(v):= \frac{1}{2}\int_\R \big|(-\Delta)^{\frac{1}{4}} v\big|^2\,\de x\,,$$
under the $\mathbb{S}^{m-1}$-constraint as in \eqref{defcritsphere}. Here $(-\Delta)^{\frac{1}{4}}$ is the operator with symbol $\sqrt{2\pi|\xi|}$. A classical computation in Fourier space gives $\mathcal{L}(v)=\mathcal{E}(v,\R)$, so that the two notions of 1/2-harmonic maps indeed coincide. Recently, 1/2-harmonic maps into a manifold have found applications in geometrical problems such as {\it minimal surfaces with free boundary} (see \cite{DaLRi3,FS,struw}, and Remark~\ref{minimsurf}), and they are intimately related to {\it harmonic maps with partially free boundary}  (see Remark \ref{boundharmVsFreebound}, and \cite{GJ,Ham,Mos}).   
While the main purpose of~\cite{DaLRi,DaLRi2} was to show the regularity of 1/2-harmonic maps for $n=1$, it was suggested there that they arise as limits of the fractional Ginzburg-Landau equation. Our primary objective here is  to prove this assertion in arbitrary dimensions and for arbitrary domains, prescribing eventually an exterior Dirichlet condition like in \eqref{Dircondintro}. Working on such a program, it is natural to face the regularity problem for 1/2-harmonic maps in arbitrary dimension, and this is our secondary objective. Looking at the 1/2-harmonic map equation~\eqref{eqhhalfHarm}-\eqref{eqhhalfHarm2}, one immediately realizes that it has precisely the same  critical structure than the one of the local case, for both regularity and compactness. Here, equation \eqref{eqhhalfHarm} turns out to be invariant in dimension $n=1$ under composition with traces on~$\R$ of conformal maps from the half plane $\R^2_+$ into itself. By analogy with the local case again, it suggests that solutions of the fractional Ginzburg-Landau equation \eqref{GLIntro}
are not strongly compact as $\eps\to 0$ in general. Describing the possible defect of strong convergence shall be one of the main concern in our asymptotic analysis of the fractional Ginzburg-Landau equation. 
\vskip5pt

Our first main result, stated in Theorem \ref{mainthm} below, deals with  the fractional Ginzburg-Landau equation supplemented with the Dirichlet exterior condition \eqref{Dircondintro}, and provides an answer to the questions above. We assume in this theorem the uniform energy bound with respect to $\eps$ which easily holds in most cases whenever the exterior condition $g$ admits an $\mathbb{S}^{m-1}$-valued extension in $\omega$ of finite energy. Note that, if $g$ does not admits a smooth $\mathbb{S}^{m-1}$-valued extension in $\omega$, singularities have to appear in the limit $\eps\to 0$ as in usual Ginzburg-Landau problems. If $n=m$, this is the case when  $\partial\omega\simeq \mathbb{S}^{n-1}$ and $g:\partial\omega\to\mathbb{S}^{n-1}$ has a non vanishing topological degree.

\begin{theorem}\label{mainthm}
Let $\omega\subset\R^{n}$ be a smooth bounded open set, and let $g:\R^n\to \R^m$ be a smooth map satisfying $|g|=1$ in $\R^n\setminus\omega$. Let $\varepsilon_k\downarrow0$ be an arbitrary sequence,  and let  
$\{v_k\}_{k\in\mathbb{N}}\subset H^{1/2}_g(\omega;\mathbb{R}^m)\cap L^4(\omega)$ be such that for each $k\in\N$, $v_k$ weakly solves   
$$\begin{cases}
\displaystyle  (-\Delta)^{\frac{1}{2}}  v_k= \frac{1}{\varepsilon_k}(1-|v_k|^2)v_k & \text{in $\omega$}\,,\\[8pt]
v_k =g & \text{in $\R^n\setminus\omega$}\,.
\end{cases}$$
If $\sup_k \mathcal{E}_{\eps_k}(v_k,\omega)<\infty$, then there exist a (not relabeled) subsequence  and $v_*\in H^{1/2}_g(\omega;\R^m)$ a weak 1/2-harmonic map into $\mathbb{S}^{m-1}$ in $\omega$ such that $v_k-v_*\rightharpoonup 0$ weakly in $H^{1/2}_{00}(\omega)$. In addition, there exist 
a finite nonnegative Radon measure $\mu_{\rm sing}$ on $\omega$,  
a countably $\mathscr{H}^{n-1}$-rectifiable relatively closed set $\Sigma\subset \omega$ 
of  finite $(n-1)$-dimensional Hausdorff measure, and a Borel function $\theta:\Sigma\to(0,\infty)$ such that  
\vskip5pt
\begin{itemize}[leftmargin=22pt]
\item[\rm (i)] $\displaystyle {\rm e}(v_k,\omega)\,\mathscr{L}^n\LL\omega\, \mathop{\rightharpoonup}\limits^* \,{\rm e}(v_*,\omega)\,\mathscr{L}^n\LL\omega +\mu_{\rm sing}$ weakly* as Radon measures on $\omega$; 
\vskip5pt
\item[\rm (ii)] $\displaystyle \frac{(1-|v_k|^2)^2}{\eps_k}\to 0\;$ in $L^1_{\rm loc}(\omega)$;
\vskip5pt
\item[\rm (iii)] $\displaystyle \frac{1-|v_k(x)|^2}{\eps_k}\rightharpoonup \frac{\gamma_n}{2}\int_{\R^n}\frac{|v_*(x)-v_*(y)|^2}{|x-y|^{n+1}}\,\de y +\mu_{\rm sing}\;$ in $\mathscr{D}'(\omega)$; 
\vskip5pt
\item[\rm (iv)] $\mu_{\rm sing}= \theta \mathscr{H}^{n-1}\LL\Sigma$; 
\vskip5pt
\item[\rm (v)] $v_*\in C^\infty(\omega\setminus\Sigma)$ and $v_n\to v_*$ in $C^{\ell}_{\rm loc}(\omega\setminus\Sigma)$ for every $\ell\in\N$; 
\vskip5pt
\item[\rm (vi)] if $n\geq 2$, the limiting 1/2-harmonic map $v_*$ and the defect measure $\mu_{\rm sing}$ satisfy the stationarity relation
\begin{equation}\label{statintro}
\left[\frac{\de}{\de t} \,\mathcal{E}\big(v_*\circ\phi_t,\omega\big)\right]_{t=0} =\frac{1}{2} \int_\Sigma {\rm div}_{\Sigma}X\,\de \mu_{\rm sing}
\end{equation}
for all vector fields $X\in C^1(\R^{n};\R^{n})$ compactly supported in $\omega$, where $\{\phi_t\}_{t\in\R}$ denotes the flow on $\R^n$ generated by $X$; 
\vskip5pt
 \item[\rm (vii)] if $n=1$, the set $\Sigma$ is finite and $v_*\in C^\infty(\omega)$. 
\end{itemize}
\end{theorem}

The proof of this theorem rests on the representation of $(-\Delta)^{\frac1 2}$ as {\it Dirichlet-to-Neumann operator} associated to the harmonic extension to the open half space $\R^{n+1}_+:=\R^n\times(0,\infty)$ given by the convolution product with the Poisson kernel, see \eqref{poisson}. More precisely, denoting by $v\mapsto v^{\e}$ this harmonic extension,  we shall prove that it is well defined on the space of functions of finite Dirichlet 1/2-energy, and that 
$\partial_\nu v^\e=(-\Delta)^{\frac1 2}v$ as distributions  on $\omega$. Here, $\partial_\nu$ denotes the exterior normal differentiation on $\partial\R^{n+1}_+\simeq\R^n$. When applying the extension procedure to a solution 
$v_\eps$ of the fractional Ginzburg-Landau equation~\eqref{GLIntro}, we end up with the following {\it system of Ginzburg-Landau boundary reactions}
\begin{equation}\label{BoundGLIntro}
\begin{cases}
\Delta v_\eps^\e=0 & \text{in $\R^{n+1}_+$}\,,\\[8pt]
\displaystyle \frac{\partial v^\e_\eps}{\partial\nu}= \frac{1}{\varepsilon}(1-|v^\e_\eps|^2)v^\e_\eps & \text{on $\omega$}\,.
\end{cases}
\end{equation}
The asymptotic analysis of this system as $\eps\to0$  will lead us to the main conclusions stated in Theorem~\ref{mainthm}. To this purpose, we shall first establish an {\it epsilon-regularity} type of estimate for \eqref{BoundGLIntro} in the spirit of the regularity theory for harmonic maps \cite{Bet,SU} or usual Ginzburg-Landau equations \cite{CS}. This estimate is the key to derive the convergence result and the rectifiability of the defect measure. Note that Theorem~\ref{mainthm} actually shows that $\mu_{\rm sing}$ is a $(n-1)$-rectifiable varifold in the sense of {\sc Almgren}, see {\it e.g.} \cite{Sim2}. 
We emphasize that identity \eqref{statintro} is  precisely the coupling equation between the limiting 1/2-harmonic $v_*$ and the defect measure $\mu_{\rm sing}$. It states that the first inner variation of the Dirichlet 1/2-energy of $v_*$ is equal to $-\frac1 2$ times the first inner variation of the varifold $\mu_{\rm sing}$, see \cite[formulas 15.7 and~16.2]{Sim2}.  We shall achieve \eqref{statintro} in two independent steps. The first step consists in proving an analogous identity when passing to the limit $\eps\to0$ in system \eqref{BoundGLIntro}. In the spirit of \cite{LW3}, the convenient way to let $\eps\to 0$ in the first inner variation of the Dirichlet energy of $v^\e_\eps$ is to use the notion of {\it generalized varifold} of {\sc Ambrosio \& Soner} \cite{AS}, once adapted to the boundary setting. In turn, the second step allows us to return to the original formulation on $\R^n$. It shows that the first inner variation of the Dirichlet 1/2-energy of an arbitrary map $v$ is equal to the first inner variation up to $\omega$ of the Dirichlet energy  of its harmonic extension~$v^\e$, see Lemma~\ref{compstatfrac}. Note that this crucial observation also implies that an harmonic extension $v^\e$ is stationary up to $\omega$ for the Dirichlet energy whenever the original function $v$ is stationary for the Dirichlet 1/2-energy. Here, stationarity means stationarity with respect to inner variations ({\it i.e.}, variations of the domain). 
 \vskip5pt
 
 Concerning the regularity issue for 1/2-harmonic maps, it is also fruitful to rephrase the problem in terms of harmonic extensions. Applying the extension procedure to a weak $\mathbb{S}^{m-1}$-valued 1/2-harmonic map $v_*$ in $\omega$ leads to the system
\begin{equation}\label{BoundHarmIntro}
\begin{cases}
\Delta v_*^\e=0 & \text{in $\R^{n+1}_+$}\,,\\[8pt]
\displaystyle \frac{\partial v^\e_*}{\partial\nu}\perp T_{v^\e_*}\mathbb{S}^{m-1} & \text{on $\omega$}\,.
\end{cases}
\end{equation}
This system turns out to be (almost) included in the class of {\it harmonic maps with free boundary} (see Remark~\ref{boundharmVsFreebound}) for which a regularity theory do exist \cite{DS,DS2,HL,Sch}. This theory provides partial regularity results, and it requires {\it minimality}, or at least {\it stationarity} up to the free boundary.  As already mentioned in the discussion above, if $v_*$ is assumed to be stationary, then the extension $v^\e_*$ is stationary up to $\omega$. Similarly, we shall see that harmonic extensions also inherit minimality up to $\omega$. This is then enough to derive a general regularity result for 1/2-harmonic maps under a minimality or stationarity assumption. 
We finally point out that the relation between fractional harmonic maps and harmonic maps with a free boundary has been previously noticed by {\sc Moser} \cite{Mos} for a (non-explicit) operator  slightly different from the square root Laplacian $(-\Delta)^{\frac1 2}$, and leading to slightly weaker results. 
\vskip5pt

Our regularity results for 1/2-harmonic maps can be summarized in the following theorem. For  simplicity we only state it for a sphere target, but we would like to stress that it actually holds for more general target manifolds, see Remark~\ref{generaltarget}. Notice that it is an {\it interior regularity} result. The regularity at the boundary when prescribing an exterior Dirichlet condition  \eqref{Dircondintro} remains an open question. 

 \begin{theorem}\label{reghalfharmintro}
 Let $\omega\subset\R^{n}$ be a smooth bounded open set, and let $v_*\in \widehat H^{1/2}(\omega;\mathbb{R}^{m})\cap L^\infty(\R^n)$ be a weak 1/2-harmonic map into $\mathbb{S}^{m-1}$ in~$\omega$. Then $v_*\in C^\infty\big(\omega\setminus{\rm sing}(v_*)\big)$, where ${\rm sing}(v_*)$ denotes the complement of the largest open set on which $v_*$ is continuous. Moreover, 
 \begin{itemize}[leftmargin=22pt]
\item[\rm (i)] if $n=1$, then ${\rm sing}(v_*)\cap\omega=\emptyset$;  
\vskip3pt

\item[\rm (ii)] if $n\geq 2$ and $v_*$ is stationary, then  $\mathscr{H}^{n-1}\big({\rm sing}(v_*)\cap\omega\big)=0$;
\vskip3pt

\item[\rm (iii)] if $v_*$ is minimizing, then ${\rm dim}_{\mathscr{H}}\big({\rm sing}(v_*)\cap\omega\big)\leq n-2$ for $n\geq 3$, and ${\rm sing}(v_*)\cap\omega$ is discrete for $n=2$. 
 \end{itemize}
 \end{theorem}

\vskip5pt

Before concluding this introduction, let us briefly comment on some possible extensions of the present results and further open questions. First of all, in the energy \eqref{defFGLenerg} one could replace  the Ginzburg-Landau potential $(1-|u|^2)^2$ by a more general nonnegative potential $W(u)$ having a  zero set $\{W=0\}$ given by a smooth compact submanifold $\mathcal{N}$ of $\R^m$ without boundary, and then consider the corresponding fractional Ginzbug-Landau equation. In this context, the singular limit $\eps\to 0$ leads to the 1/2-harmonic map system into $\mathcal{N}$ (see Remark~\ref{generaltarget}). If the codimension of $\mathcal{N}$ is equal to $1$ (plus some non degeneracy assumptions on $W$), the proof of Theorem \ref{mainthm} can certainly be reproduced with minor modifications. However, the higher codimension case seems to require additional analysis since our {\it espilon-regularity estimate} (Theorem \ref{epsreg}) strongly uses the codimension 1 structure. It could be interesting to have a proof handling both cases. In another direction, one could consider the fractional Ginzburg-Landau equation in some open subset $\omega$ of a more general (smooth) complete Riemannian manifold $(\mathcal{M}, \mathbf{g})$ with a fractional $1/2$-Laplacian $(-\Delta_{\mathbf{g}})^{\frac{1}{2}}$ given by the square root of the Laplace-Beltrami operator on $\mathcal{M}$ as defined through spectral theory and functional calculus. In the most difficult case of a noncompact manifold, the construction of $(-\Delta_{\mathbf{g}})^{\frac{1}{2}}$ (or any fractional power $s\in(0,1)$) can be found in \cite{BGS} (and the references therein).  Under some geometrical assumption on $\mathcal{M}$ ({\it e.g.} $\mathcal{M}$ has bounded geometry), \cite{BGS} follows  the approach of \cite{ST} to construct the Poisson kernel for the harmonic extension in $\mathcal{M}\times\R_+$ (endowed with the product metric $\mathbf{g}+|\de t|^2$) of functions defined on $\mathcal{M}$, and it is shown that the fractional Laplacian  $(-\Delta_{\mathbf{g}})^{\frac{1}{2}}$ coincides with the Dirichlet-to-Neumann operator associated to this harmonic extension. This is then enough to rephrase the fractional Ginzburg-Landau equation on $\mathcal{M}$ (and $1/2$-harmonic maps) as a  system of Ginzburg-Landau boundary reactions on $\partial(\mathcal{M}\times\R_+)\simeq\mathcal{M}$ (respectively, harmonic maps with partially free boundary), and we believe that most of our results should extend to this setting.  Note that this approach was successfully used in \cite{GSS} in the scalar case $m=1$ and for any fractional power $s\in(0,1)$ to study the so-called {\it layer solutions} in the hyperbolic space $\mathbb{H}^n$. As already mentioned, we have considered here only the case $s=1/2$ for the specific relations with the geometrical problems of harmonic maps with partially free boundary and minimal surfaces with free boundary. However, it could be interesting to investigate in the case $s\not=1/2$ the analogues of this geometric objects, and then  perform the asymptotic analysis as $\eps\to 0$ of the corresponding Ginzburg-Landau equation. For this question, some aspects of the present analysis  can actually be used for an arbitrary $s\in(0,1)$ by means of the $2s$-harmonic extension procedure discovered in \cite{CaffSil}. In a forthcoming article \cite{MilSir}, we shall consider the scalar Ginzburg-Landau (Allen-Cahn) equation in the case $s\in(0,1/2)$ and the related nonlocal $s$-minimal surfaces in the spirit of \cite{CRS,SV1}. Interpreting nonlocal $s$-minimal surfaces as fractional $s$-harmonic maps with values in $\mathbb{S}^0\simeq\{\pm1\}$, and taking advantage of the $2s$-harmonic extension, we shall prove in \cite{MilSir} the convergence, as $\eps\to0$,  
of arbitrary solutions of the fractional Allen-Cahn equation with equibounded energy towards nonlocal $s$-minimal surfaces (possibly non-minimizing), thus extending a result in  \cite{SV1} to the non-minimizing case. 
\vskip5pt

The paper is organized as follows. In Section~\ref{prelim}, we present all the needed material on the functional aspects of the {\sl square root Laplacian/Dirichlet-to-Neumann operator}, as well as some key estimates related to harmonic extensions. Actually, we believe that this section can be of independent interest. Section \ref{FractGL} is devoted to the qualitative analysis of the fractional Ginzburg-Landau equation for $\eps$ fixed, and provides regularity results for this equation. Section~\ref{1/2Harm} is entirely devoted to the analysis of 1/2-harmonic maps. In a first part, we give the regularity theory  and prove Theorem~\ref{reghalfharmintro}. The second part deals with explicit examples of 1/2-harmonic maps underlying their geometrical nature, and stressing the analogies with standard harmonic maps.   In Section~\ref{epsregtitle}, we prove the aforementioned {\it epsilon-regularity estimate} for the system of Ginzburg-Landau boundary reactions, and its asymptotic analysis as $\eps\to0$ follows in Section~\ref{asymptGLBR}.  Then we return to the fractional Ginzburg-Landau equation in Section~\ref{FGLasymp} where we perform its asymptotic analysis with and without exterior Dirichlet condition. The special case of energy minimizers under exterior Dirichlet condition is also treated. Finally, we collect in Appendix A the proofs of several statements from Section~\ref{prelim}, and Appendix B provides some elliptic estimates coming into play in the proof of the epsilon-regularity result. 

\vskip5pt

\begin{notations*}
Throughout the paper,  $\R^n$ is identified with $\partial  \mathbb{R}^{n+1}_+=\R^n\times\{0\}$. 
We denote by $B_r(x)$ the open ball in $\mathbb{R}^{n+1}$ of radius $r$ centered at $x$, while $D_r(x):= B_r(x)\cap\mathbb{R}^{n}$ is the open disc in $\R^n$ centered at $x\in\R^n$.  For a set $A\subset  \mathbb{R}^{n+1}$, we write $A^+:=A\cap \mathbb{R}^{n+1}_+$, $\partial^+ A:=\partial A\cap \mathbb{R}^{n+1}_+$, and $\overline A^+$ the closure of $A^+$ in $\R^{n+1}$.  If $A\subset\R^n$ and no confusion arises, the complement of $A$ in $\R^n$ is simply denoted by~$A^c$. The Froebenius inner product between two matrices $M=(M_{ij})$ and $N=(N_{ij})$ is denoted by  $M:N:=\sum_{i,j} M_{ij}N_{ij}$. 

Concerning  bounded open sets $\Om\subset\R^{n+1}_+$, we shall say that $\Omega$ is an {\bf admissible open set} whenever 

\hskip5pt $\bullet$ $\partial\Omega$ is Lipschitz regular;  
\vskip2pt

\hskip5pt $\bullet$ the open set $\partial^0\Omega\subset\R^n$ defined by 
$$\partial^0\Omega:=\left\{x\in\partial\Omega\cap\R^{n} : B^+_{r}(x)\subset\Om \text{ for some $r>0$}\right \}\,,$$
is non empty and has Lipschitz boundary; 
\vskip2pt

\hskip5pt $\bullet$ $\partial \Omega=\partial^+\Omega\cup\overline{\partial^0\Omega}\,$.
\vskip5pt

Finally, we shall always denote by $C$ a generic positive constant which only depends on the dimensions $n$ and $m$, and possibly changing from line to line. If a constant depends on additional given parameters, we shall write those parameters using the subscript notation. 
\end{notations*}

 \vskip10pt

\section{Energy spaces and local representation of the 1/2-Laplacian} \label{prelim}  

\subsection{Basics on $H^{1/2}$-spaces}
 
Given an open set $\omega\subset \mathbb{R}^n$,  the fractional Sobolev space $H^{1/2}(\omega;\mathbb{R}^m)$ is defined by  
 $$H^{1/2}(\omega;\mathbb{R}^m):=\big\{v\in L^2(\omega;\mathbb{R}^m): [v]_{H^{1/2}(\omega)}<\infty\big\}\,,$$
where
$$[v]^2_{H^{1/2}(\omega)}:=\frac{\gamma_n}{4}\iint_{\omega\times \omega} \frac{|v(x)-v(y)|^2}{|x-y|^{n+1}}\,\de x\de y\,,\quad \gamma_n:=\frac{\Gamma\big((n+1)/2\big)}{\pi^{(n+1)/2}}\,,$$ 
and we recall that $H^{1/2}(\omega;\mathbb{R}^m)$ is an Hilbert space normed by  
$$\|v\|_{H^{1/2}(\omega)}:= \|v\|_{L^2(\omega)}+[v]_{H^{1/2}(\omega)}\,.$$
Then $ H^{1/2}_{\rm loc}(\mathbb{R}^n;\mathbb{R}^m)$ denotes the class of functions which belongs to $\in H^{1/2}(\omega;\mathbb{R}^m)$ for every bounded open set $\omega\subset\mathbb{R}^n$. 
The linear space $H^{1/2}_{00}(\omega;\mathbb{R}^m)$ is defined by 
$$H^{1/2}_{00}(\omega;\mathbb{R}^m):=\left\{v\in H^{1/2}(\mathbb{R}^n;\mathbb{R}^m) :  v=0 \text{ a.e. in } \R^n\setminus\omega\right\} \subset H^{1/2}(\mathbb{R}^n;\mathbb{R}^m)\,. $$
Endowed with the  norm induced by $H^{1/2}(\mathbb{R}^n;\mathbb{R}^m)$,   
the space  $H^{1/2}_{00}(\omega;\mathbb{R}^m)$ is also an Hilbert space, and for $v\in H^{1/2}_{00}(\omega;\mathbb{R}^m)$, 
\begin{equation}\label{defErond}
[v]^2_{H^{1/2}(\mathbb{R}^n)}=\frac{\gamma_n}{4}\iint_{\omega\times\omega}  \frac{|v(x)-v(y)|^2}{|x-y|^{n+1}}\,\de x\de y 
+ \frac{\gamma_n}{2} \iint_{\omega\times \omega^{c}}  \frac{|v(x)-v(y)|^2}{|x-y|^{n+1}}\,\de x\de y \,. 
\end{equation}
The topological dual space of $H^{1/2}_{00}(\omega;\mathbb{R}^m)$ will be denoted by $H^{-1/2}(\omega;\mathbb{R}^m)$. We recall that if the boundary of $\omega$ is smooth enough ({\it e.g.} if  $\partial \omega$ is Lipschitz regular), then 
\begin{equation}\label{densitysmoothH1/200}
 H^{1/2}_{00}(\omega;\mathbb{R}^m)= \overline{\mathscr{D}(\omega;\R^m)}^{\,\|\cdot\|_{H^{1/2}(\mathbb{R}^n)}}\,,
 \end{equation}
 see \cite[Theorem~1.4.2.2]{G}.
\vskip5pt

Throughout the paper we shall be interested in functions for which the right hand side of \eqref{defErond} is finite. We denote this class of funtions by 
$$\widehat{H}^{1/2}(\omega;\mathbb{R}^m):=\left\{v\in L^{2}_{\rm loc}(\mathbb{R}^n;\mathbb{R}^m) : \mathcal{E}(v,\omega)<\infty\right\} \,,$$
where $\mathcal{E}(\cdot,\omega)$ is  the {\it 1/2-Dirichlet energy} defined in \eqref{defenergE}. 
The following properties hold for any bounded open subsets $\omega$ and $\omega'$ of $\R^n$: 
\begin{itemize}
\vskip5pt
\item $ \widehat{H}^{1/2}(\omega;\mathbb{R}^m)$ is a linear space; 
\vskip5pt
\item $ \widehat{H}^{1/2}(\omega;\mathbb{R}^m) \subset  \widehat{H}^{1/2}(\omega';\mathbb{R}^m)$ whenever $\omega'\subset\omega$, and    
$ \mathcal{E}(v,\omega')\leq \mathcal{E}(v,\omega)\,$;
\vskip5pt
\item $ \widehat{H}^{1/2}(\omega;\mathbb{R}^m)\cap H^{1/2}_{\rm loc}(\mathbb{R}^n;\R^m) \subset  \widehat{H}^{1/2}(\omega';\mathbb{R}^m)\,$;
\vskip5pt
\item $H^{1/2}_{\rm loc}(\mathbb{R}^n;\mathbb{R}^m)\cap L^\infty(\mathbb{R}^n) \subset  \widehat{H}^{1/2}(\omega;\mathbb{R}^m)\,$.
\end{itemize}
\vskip3pt

An elementary property concerning $\widehat{H}^{1/2}(\omega;\mathbb{R}^m)$  is given in Lemma \ref{adminHchap} below whose proof can be found in Appendix A.
Using this lemma, it is then elementary to see that  $ \widehat{H}^{1/2}(\omega;\mathbb{R}^m)$ is a Hilbert space for the scalar product associated to the norm 
$v\mapsto \|v\|_{L^2(\omega)}+( \mathcal{E}(v,\omega))^{1/2}$.

\begin{lemma}\label{adminHchap}
Let $x_0\in\omega$ and $\rho>0$ be  such that $D_{\rho}(x_0)\subset\omega$.  There exists a constant $C_\rho>0$, independent of~$x_0$, such that 
$$\int_{\R^n}\frac{|v(x)|^2}{(|x-x_0|+1)^{n+1}}\,\de x\leq C_{\rho}\big(\mathcal{E}\big(v,D_{\rho}(x_0)\big)+\|v\|^2_{L^2(D_{\rho}(x_0))}\big)$$
for all $v\in  \widehat{H}^{1/2}(\omega;\mathbb{R}^m)$. 
\end{lemma}

\begin{remark}[\bf Exterior Dirichlet condition] We observe that for any $v\in \widehat{H}^{1/2}(\omega;\mathbb{R}^m)$, 
\begin{equation}\label{v+H1/200}
v+ H^{1/2}_{00}(\omega;\R^m)\subset  \widehat{H}^{1/2}(\omega;\mathbb{R}^m)\,.  
\end{equation}
Moreover,  if 
$v=g$ a.e. in $\R^n\setminus\omega$ for some $g\in \widehat{H}^{1/2}(\omega;\mathbb{R}^m)$, then $v-g\in H^{1/2}_{00}(\omega;\R^m)$. 
As a consequence, the affine subspace
$$H^{1/2}_g(\omega;\mathbb{R}^m):= \left\{v\in \widehat{H}^{1/2}(\omega;\mathbb{R}^m) : v=g \text{ a.e. in $\R^n\setminus\omega$}\right\}$$
is characterized by 
\begin{equation}\label{caracH1/2g}
H^{1/2}_g(\omega;\mathbb{R}^m)=g+H^{1/2}_{00}(\omega;\R^m)\,.
\end{equation}
Finally, we shall keep in mind  that 
\begin{equation}\label{H1/2gH1/2loc}
H^{1/2}_g(\omega;\mathbb{R}^m)\subset  H^{1/2}_{\rm loc}(\mathbb{R}^n;\mathbb{R}^m)
\end{equation}
whenever $g\in \widehat{H}^{1/2}(\omega;\mathbb{R}^m)\cap H^{1/2}_{\rm loc}(\mathbb{R}^n;\R^m)$. 
\end{remark}


 \subsection{The fractional Laplacian}

Let $\omega\subset \mathbb{R}^n$ be a bounded open set. We define the fractional $1/2$-Laplacian $(-\Delta)^{\frac{1}{2}}: \widehat{H}^{1/2}(\omega;\mathbb{R}^m) \to \big(\widehat{H}^{1/2}(\omega;\mathbb{R}^m)\big)^\prime$ as the continuous linear operator induced by the quadratic form $\mathcal E(\cdot,\omega)$. More precisely,  
for a given map $v\in \widehat{H}^{1/2}(\omega;\mathbb{R}^m)$, we define the distribution $ (-\Delta)^{\frac{1}{2}} v$ through its action on elements of $ \widehat{H}^{1/2}(\omega;\mathbb{R}^m)$ by setting 
\begin{multline}\label{deffraclap}
\big\langle  (-\Delta)^{\frac{1}{2}} v, \varphi\big\rangle_\omega:=\frac{\gamma_n}{2}\iint_{\omega\times\omega}  \frac{(v(x)-v(y))\cdot(\varphi(x)-\varphi(y))}{|x-y|^{n+1}}\,\de x\de y\\ 
+ \gamma_n \iint_{\omega\times \omega^c}  \frac{(v(x)-v(y))\cdot(\varphi(x)-\varphi(y))}{|x-y|^{n+1}}\,\de x\de y \,. 
\end{multline}
Note that for every $v\in \widehat{H}^{1/2}(\omega;\mathbb{R}^m)$, the restriction of $ (-\Delta)^{\frac{1}{2}} v$ to $H^{1/2}_{00}(\omega;\mathbb{R}^m)$ belongs to $H^{-1/2}(\omega;\mathbb{R}^m)$ with the estimate 
\begin{equation}\label{estinormH-1/2fraclap}
\| (-\Delta)^{\frac{1}{2}} v\|^2_{H^{-1/2}(\omega)}\leq 2\mathcal E(v,\omega)\,,
\end{equation}
which obviously follows from \eqref{defErond} and Cauchy-Schwarz Inequality. In addition, formula \eqref{deffraclap} being precisely twice the {\it symmetric 
bilinear form} associated to $\mathcal E(\cdot,\omega)$, we observe that the restriction of $ (-\Delta)^{\frac{1}{2}} v $ to $H^{1/2}_{00}(\omega;\mathbb{R}^m)$ corresponds to the first variation of the energy $\mathcal{E}(v,\omega)$ with respect to pertubations supported in $\omega$, {\it i.e.},  
\begin{equation}\label{firstvarcalE}
\big\langle  (-\Delta)^{\frac{1}{2}} v, \varphi\big\rangle_\omega=\left[\frac{\de}{\de t} \mathcal E (v+t \varphi , \omega)\right]_{t=0} 
\end{equation}
for all $\varphi \in H^{1/2}_{00}(\omega;\mathbb{R}^m)$.

\begin{remark}
Consider  an open set $\omega'\subset \omega$.  Since $\widehat{H}^{1/2}(\omega;\mathbb{R}^m)\subset \widehat{H}^{1/2}(\omega';\mathbb{R}^m)$ 
and $H^{1/2}_{00}(\omega';\mathbb{R}^m)\subset H^{1/2}_{00}(\omega;\mathbb{R}^m)$, direct computations from \eqref{deffraclap} yield  
$$\big\langle  (-\Delta)^{\frac{1}{2}} v, \varphi\big\rangle_\omega= \big\langle  (-\Delta)^{\frac{1}{2}} v, \varphi\big\rangle_{\omega'}$$
for all $\varphi\in H^{1/2}_{00}(\omega';\mathbb{R}^m)$.
\end{remark}
\vskip5pt

\begin{remark} If $v$ is a smooth bounded function, the distribution $ (-\Delta)^{\frac{1}{2}} v$ can be rewritten from \eqref{deffraclap} as a pointwise defined function through the formula
$$ (-\Delta)^{\frac{1}{2}} v(x) = {\rm p.v.}\left(\gamma_n\int_{\R^n}\frac{v(x)-v(y)}{|x-y|^{n+1}}\,\de y\right):=\lim_{\delta\downarrow 0}
\gamma_n\int_{\R^n\setminus D_\delta(x)}\frac{v(x)-v(y)}{|x-y|^{n+1}}\,\de y\,.$$
In turn, this last formula can be written in  a non-singular form, that is 
$$(-\Delta)^{\frac{1}{2}} v(x) = \frac{\gamma_n}{2}\int_{\R^n}\frac{v(x)-v(x+h)-\nabla v(x)\cdot h}{|h|^{n+1}}\,\de h\,.$$
or
$$(-\Delta)^{\frac{1}{2}} v(x) = \frac{\gamma_n}{2}\int_{\R^n}\frac{-v(x+h)+2v(x)-v(x-h)}{|h|^{n+1}}\,\de h\,.$$
Finally, as a motivation to the following subsection, we recall that  
$ (-\Delta)^{\frac{1}{2}} $ is the infinitesimal generator of the Poisson Kernel, {\it i.e.}, 
$$e^{-t (-\Delta)^{\frac{1}{2}} }(x,y)=\frac{\gamma_nt}{(|x-y|^2+t^2)^{\frac{n+1}{2}}}\;,\quad t>0 \,,$$
see {\it e.g.} \cite{LL}.
\end{remark}


 \subsection{Harmonic extension and the Dirichlet-to-Neumann operator}

Throughout the paper, for a measurable function $v$ defined over $\mathbb{R}^n$, we shall denote by 
$v^\e$ its extension to the half-space $\mathbb{R}^{n+1}_+$ given by the convolution of $v$ with the 
Poisson kernel, {\it i.e.}, 
\begin{equation}\label{poisson}
v^\e(x):= \gamma_n\int_{\mathbb{R}^n}\frac{x_{n+1} v(z)}{(|x'-z|^2+x_{n+1}^2)^{\frac{n+1}{2}}}\,\de z \,, 
\end{equation}
where we write $x=(x',x_{n+1})\in\R^{n+1}_+=\R^n\times(0,\infty)$.  
Notice that $v^\e$ is well defined if $v$ belongs to the Lebesgue space $L^p$ over $\R^n$ with respect to the finite measure 
\begin{equation}\label{defmeasm}
\mathfrak{m}:=(1+|z|)^{-(n+1)}\,\de z
\end{equation}
for some $1\leq p \leq\infty$. In particular,  $v^\e$ can be defined  whenever 
$v\in \widehat{H}^{1/2}(\omega;\mathbb{R}^m)$ for some open bounded set $\omega\subset\R^n$ by Lemma~\ref{adminHchap}. 
Moreover, if  $v\in L^\infty(\R^n)$, then $v^\e\in L^\infty(\R_+^{n+1})$ and  
\begin{equation}\label{bdlinftyext}
\|v^\e\|_{L^\infty(\R_+^{n+1})} \leq \|v\|_{L^\infty(\R^n)}\,.
\end{equation}
For an admissible function $v$, the extension  $v^\e$ is harmonic in $\R^{n+1}_+$. In addition, it has a pointwise trace on $\partial\R^{n+1}_+=\R^n$ which is equal to $v$ at every Lebesgue point. In other words,  
$v^\e$ solves
\begin{equation}\label{eqextharm}
\begin{cases} 
\Delta v^\e = 0 & \text{in $\mathbb{R}_+^{n+1}$}\,,\\
v^\e = v  & \text{on $\mathbb{R}^{n}$}\,. 
\end{cases}
\end{equation}
\vskip3pt

The operator $v\mapsto v^\e$ enjoys some useful continuity properties. Among them, we shall use  the following lemma which is proven in Appendix A. 
\begin{lemma}\label{context}
For every $R>0$, the linear operator $\mathfrak{P}_R:L^2(\R^n,\mathfrak{m})\to L^2(B_R^+)$ defined by 
\begin{equation}\label{Pfrak}
\mathfrak{P}_R(v):={v^\e}_{|B_R^+} \,,
\end{equation}
is continuous. 
\end{lemma}

If $v\in \dot H^{1/2}(\mathbb{R}^n;\mathbb{R}^m)$, it is  well known that  the harmonic extension $v^\e$  belongs to $\dot H^1(\mathbb{R}_+^{n+1};\mathbb{R}^m)$, 
and  the $H^{1/2}$-seminorm of $v$ can be computed from the Dirichlet energy of~$v^\e$ (here $\dot H^{1/2}$ and $\dot H^1$ denote the homogeneous Sobolev spaces). 
Moreover, $v^\e$ turns out to be the extension of $v$ of minimal energy. 
This result, summarized in Lemma \ref{normexth1/2} below,  is a classical application of Fourier Transform (see {\it e.g.} \cite{Hitch}). 

\begin{lemma}\label{normexth1/2}
Let $v\in \dot H^{1/2}(\mathbb{R}^n;\mathbb{R}^m)$, and let $v^\e$ be its harmonic extension to~$\mathbb{R}^{n+1}_+$ given by \eqref{poisson}. Then 
$v^\e$ belongs to $\dot H^1(\mathbb{R}_+^{n+1};\mathbb{R}^m)$ and
$$[v]^2_{H^{1/2}(\mathbb{R}^n)}  =\frac{1}{2}\int_{\mathbb{R}^{n+1}_+}|\nabla v^\e|^2\,\de x=\inf\left\{ \frac{1}{2}\int_{\mathbb{R}^{n+1}_+}|\nabla u|^2\,\de x  : u\in \dot H^1(\mathbb{R}^{n+1}_+;\mathbb{R}^m)\,, \ u=v \text{ on $\mathbb{R}^n$} \right\}\,.$$
\end{lemma}

For a bounded open set $\omega\subset\R^n$,  
we have the following analogous result whose proof is postponed to Appendix~A. In the following statement and hereafter, we denote by $H^1_{\rm loc}(\R^{n+1}_+\cup\omega;\R^m)$ the family of maps which belongs to $H^1(\Omega;\R^m)$ for every bounded open set $\Omega\subset \R^{n+1}_+$ such that $\overline\Omega\subset \R^{n+1}_+\cup\omega$. 

\begin{lemma}\label{hatH1/2toH1}
Let $\omega\subset \mathbb{R}^n$ be a bounded open set. For every $v\in \widehat{H}^{1/2}(\omega;\mathbb{R}^m)$, the  harmonic extension $v^\e$ 
given by \eqref{poisson} belongs to $H^1_{\rm loc}(\R^{n+1}_+\cup\omega;\R^m)\cap L^2_{\rm loc}(\overline{\R}^{n+1}_+)$. In addition, for every $x_0\in\omega$, $R>0$ and  $\rho>0$ 
such that $D_{3\rho}(x_0)\subset\omega$, there exist constant $C_{R,\rho}>0$  and $C_R>0$, independent of $v$ and $x_0$, such that  
$$\big\|v^\e\big\|^2_{L^2(B_R^+(x_0))}\leq C_{R,\rho} \left(\mathcal{E}\big(v,D_{2\rho}(x_0)\big)+\|v\|^2_{L^2(D_{2\rho}(x_0))}\right)\,, $$
and 
$$ \big\|\nabla v^\e\big\|^2_{L^2(B_\rho^+(x_0))}\leq C_{\rho} \left(\mathcal{E}\big(v,D_{2\rho}(x_0)\big)+\|v\|^2_{L^2(D_{2\rho}(x_0))}\right)\,. $$
\end{lemma}

\begin{remark}\label{H1/2loctoH1loc}
By the previous lemma, for any $v\in \widehat H^{1/2}(\omega;\mathbb{R}^m)\cap H^{1/2}_{\rm loc}(\mathbb{R}^n;\R^m)$, the harmonic extension $v^\e$ 
belongs to $H^1_{\rm loc}(\overline{\R}^{n+1}_+;\R^m)$, and for any $R>0$, 
$$\big\| v^\e\big\|^2_{H^1(B_R^+)}\leq C_{R} \left(\mathcal{E}\big(v,D_{2R}\big)+\|v\|^2_{L^2(D_{2R})}\right)\,. $$
\end{remark}

If $v\in \widehat H^{1/2}(\omega;\mathbb{R}^m)$ for some bounded open set $\omega\subset\R^n$ with Lipschitz boundary, we now observe that $v^\e$ admits a distributional (exterior) normal derivative $\partial_\nu v^\e$ on $\omega$ by its harmonicity 
in $\mathbb{R}_+^{n+1}$. More precisely, we define $\partial_\nu v^\e$  through its action on $\varphi\in \mathscr{D}(\omega;\mathbb{R}^m)$ by
\begin{equation}\label{defNeumOp}
\left\langle \frac{\partial v^\e}{\partial\nu}, \varphi\right\rangle := \int_{\mathbb{R}^{n+1}_+}\nabla v^\e\cdot\nabla\Phi\,\de x\,,
\end{equation}
where $\Phi$ is any smooth extension of $\varphi$ compactly supported in  $\mathbb{R}_+^{n+1}\cup\omega$. Note that the right hand side of \eqref{defNeumOp} 
is well defined by Lemma~\ref{hatH1/2toH1}. Using the harmonicity of $v^\e$ and the divergence theorem, it is routine to check that 
the integral in \eqref{defNeumOp} does not depend on the choice of the extension $\Phi$, and that $\partial_\nu v^\e$  coincides with the classical exterior normal derivative of $v^\e$ on $\omega\subset \partial\R^{n+1}_+$  
whenever $v$ is smooth. It turns out that $\partial_\nu v^\e$ coincides with the distribution $ (-\Delta)^{\frac{1}{2}} v$ defined in the previous subsection. Here again the proof is left to Appendix A.

\begin{lemma}\label{repnormderfraclap}
Let $\omega\subset \mathbb{R}^n$ be a bounded open set with  Lipschitz boundary. For every $v\in \widehat H^{1/2}(\omega;\mathbb{R}^m)$ we have 
$$\big\langle  (-\Delta)^{\frac{1}{2}} v, \varphi\big\rangle_\omega =\left\langle \frac{\partial v^\e}{\partial\nu}, \varphi\right\rangle  \quad \text{for all $\varphi\in H^{1/2}_{00}(\omega;\mathbb{R}^m)$}\,,$$
where $v^\e$ is the harmonic extension of $v$ to $\mathbb{R}^{n+1}_+$ given by \eqref{poisson}.
\end{lemma}

Up to now, we have not said anything about the local counterpart of Lemma \ref{normexth1/2} concerning the minimality of $v^\e$ for a function $v$ in $\widehat H^{1/2}(\omega;\mathbb{R}^m)$. This is the purpose of the following lemma inspired from \cite[Lemma 7.2]{CRS}. 

\begin{corollary}\label{minenergdirchfrac}
Let $\omega\subset \mathbb{R}^n$ be a bounded open set, and
let $\Omega\subset \R^{n+1}_+$ be an admissible bounded  open set  such that 
$\overline{\partial^0 \Omega}\subset \omega$. 
Let $v\in \widehat H^{1/2}(\omega;\mathbb{R}^m)$, and let $v^\e$ be its harmonic extension to $\mathbb{R}^{n+1}_+$ given by \eqref{poisson}. Then, 
\begin{equation}\label{ineqenergDfrac}
\frac{1}{2}\int_\Omega|\nabla u|^2\,\de x - \frac{1}{2}\int_\Omega|\nabla v^\e|^2\,\de x
\geq  \mathcal{E}(u,\omega) -\mathcal{E}(v,\omega)
\end{equation}
for all $u\in H^1(\Omega;\R^m)$ such that $u-v^\e$ is compactly supported in $\Omega\cup\partial^0\Om$. In the right hand side of \eqref{ineqenergDfrac}, 
the trace of $u$ on $\partial^0\Om$ is extended by $v$ outside $\partial^0\Om$. 
\end{corollary}

\begin{proof}
Let $u\in H^1(\Omega;\R^m)$ such that $u-v^\e$ is compactly supported in $\Omega\cup\partial^0\Om$. We extend $u$ by 
$v^\e$ outside~$\Omega$. Then $h:=u-v^\e\in H^1(\mathbb{R}_+^{n+1};\mathbb{R}^m)$ and $h$ is compactly supported in $\Omega\cup\partial^0\Om$. Hence  
$h_{|\R^n} \in H^{1/2}_{00}(\partial^0 \Omega;\R^m)$. Since $v\in \widehat H^{1/2}(\partial^0 \Omega;\R^m)$ we deduce from \eqref{v+H1/200}
that  $u$ admits a trace on $\mathbb{R}^n$ which belongs to $ \widehat{H}^{1/2}(\partial^0 \Omega;\mathbb{R}^m)$. 

Using Lemma \ref{normexth1/2} and Lemma \ref{repnormderfraclap}, we now estimate
\begin{align}
\nonumber \frac{1}{2}\int_\Omega |\nabla u|^2\,\de x-\frac{1}{2}\int_\Omega |\nabla v^\e|^2\,\de x & = \frac{1}{2}\int_{\R^{n+1}_+} |\nabla h|^2\,\de x+\int_{\mathbb{R}^{n+1}_{+}} \nabla v^\e \cdot \nabla h\,\de x \\
\nonumber& = \frac{1}{2}\int_{\R^{n+1}_+} |\nabla h|^2\,\de x +  \big\langle  (-\Delta)^{\frac{1}{2}} v,h_{|\mathbb{R}^n}\big\rangle_{\partial^0 \Omega} \\
\nonumber& \geq  \big[h_{|\mathbb{R}^n}\big]^2_{H^{1/2}(\mathbb{R}^n)}+  \big\langle  (-\Delta)^{\frac{1}{2}} v,h_{|\mathbb{R}^n}\big\rangle_{\partial^0 \Omega} \\
& = \mathcal{E}(h_{|\mathbb{R}^n},\partial^0 \Omega)+  \big\langle  (-\Delta)^{\frac{1}{2}} v,h_{|\mathbb{R}^n}\big\rangle_{\partial^0 \Omega} \,. \label{estiequiv1}
\end{align}
Using the fact that $u_{|\mathbb{R}^n},v \in \widehat{H}^{1/2}(\partial^0 \Omega;\mathbb{R}^m)$, we derive  
\begin{equation}\label{estiequiv2}
 \mathcal{E}(h_{|\R^n},\partial^0 \Omega)=  \mathcal{E}(u_{|\mathbb{R}^n},\partial^0 \Omega)+ \mathcal{E}(v,\partial^0 \Omega)- \big\langle  (-\Delta)^{\frac{1}{2}} v,u_{|\mathbb{R}^n}\big\rangle_{\partial^0 \Omega}\,,
 \end{equation}
and 
\begin{equation}\label{estiequiv3}
\big\langle  (-\Delta)^{\frac{1}{2}} v,h_{|\mathbb{R}^n}\big\rangle_{\partial^0 \Omega}= \big\langle  (-\Delta)^{\frac{1}{2}} v,u_{|\mathbb{R}^n}\big\rangle_{\partial^0 \Omega}-  2\mathcal{E}(v,\partial^0 \Omega)\,.
\end{equation}
Gathering \eqref{estiequiv1}-\eqref{estiequiv2}-\eqref{estiequiv3} yields 
$$\frac{1}{2}\int_\Omega|\nabla u|^2\,\de x - \frac{1}{2}\int_\Omega|\nabla v^\e|^2\,\de x
\geq  \mathcal{E}(u_{|\R^n},\partial^0 \Omega) -\mathcal{E}(v,\partial^0 \Omega)\,.$$
Since $u_{|\R^n}=v$ outside $\partial^0 \Omega$, we infer that 
$$\mathcal{E}(u_{|\R^n},\partial^0 \Omega)- \mathcal{E}(v,\partial^0 \Omega)
=\mathcal{E}(u_{|\R^n},\omega)- \mathcal{E}(v,\omega)\,,$$
and the conclusion follows.  
\end{proof}

 \vskip10pt

\section{The fractional Ginzburg-Landau equation}\label{FractGL}                      

We consider for the entire section a bounded open set $\omega\subset \mathbb{R}^n$ with Lipschitz boundary. We are interested  
in weak solutions  $v_\varepsilon\in \widehat H^{1/2}(\omega;\R^m)\cap L^4(\omega)$ of the fractional Ginzburg-Landau  equation
\begin{equation}\label{eqfractGL}
 (-\Delta)^{\frac{1}{2}} v_\varepsilon=\frac{1}{\varepsilon}(1-|v_\varepsilon|^2)v_\varepsilon \quad\text{in $\omega$}\,.
\end{equation}
Here the notion of weak solution is understood in the duality sense according to the formulation \eqref{deffraclap} of the fractional Laplacian, {\it i.e.}, 
$$\big\langle  (-\Delta)^{\frac{1}{2}} v_\varepsilon, \varphi\big\rangle_{\omega} = \frac{1}{\varepsilon}\int_{\omega}(1-|v_\varepsilon|^2)v\cdot\varphi\,\de x\quad 
\text{for all $\varphi\in H^{1/2}_{00}(\omega;\R^m)\cap L^4(\omega)$}$$
(or equivalently for all $\varphi\in \mathscr{D}(\omega;\R^m)$).  
By \eqref{firstvarcalE}, such solutions correspond  to  critical points in $\omega$ of the {\it Ginzburg-Landau $1/2$-energy}  
$$\mathcal{E}_\varepsilon(v,\omega):=\mathcal{E}(v,\omega) + \frac{1}{4\varepsilon}\int_\omega (1-|v|^2)^2\,\de x\,.$$
In other words, we are interested in maps $v_\varepsilon\in \widehat H^{1/2}(\omega;\R^m)\cap L^4(\omega)$ satisfying  
$$\left[\frac{\de}{\de t} \mathcal{E}_\varepsilon(v_\varepsilon+t\varphi,\omega) \right]_{t=0} =0\quad 
\text{for all $\varphi\in H^{1/2}_{00}(\omega;\R^m)\cap L^4(\omega)$}\,.$$
Among all kinds of critical points are the {\it minimizers}. We say that $v_\varepsilon\in \widehat H^{1/2}(\omega;\R^m)\cap L^4(\omega)$ 
is a minimizer of $\mathcal{E}_\varepsilon$ in $\omega$ if 
$$\mathcal{E}_\varepsilon(v_\varepsilon,\omega)\leq \mathcal{E}_\varepsilon(v_\varepsilon+\varphi,\omega) 
\quad \text{for all $\varphi\in H^{1/2}_{00}(\omega;\R^m)\cap L^4(\omega)$}$$
(or equivalently for all $\varphi\in \mathscr{D}(\omega;\R^m)$). 
The most standard way to construct minimizing solutions (and in particular solutions of \eqref{eqfractGL}) is to minimize $\mathcal{E}_\varepsilon(\cdot,\omega)$ under an exterior Dirichlet condition. 
More precisely, given a map $g\in \widehat H^{1/2}(\omega;\R^m)\cap L^4(\omega)$, one consider the minimization problem
\begin{equation}\label{GLminProb}
\min\left\{ \mathcal{E}_\varepsilon(v,\omega) : v\in H^{1/2}_g (\omega;\R^m)\cap L^4(\omega) \right\}\,, 
\end{equation}
whose resolution follows directly from the Direct Method of Calculus of Variations. 
\vskip5pt

To investigate the qualitative properties of solutions of \eqref{eqfractGL}, we  rely on the harmonic extension to the half space $\R^{n+1}_+$ 
introduced in the previous section.  
According to Lemma \ref{repnormderfraclap}, if $v_\varepsilon\in \widehat H^{1/2}(\omega;\R^m)\cap L^4(\omega)$ is a weak solution of \eqref{eqfractGL}, then its 
harmonic extension $v^\e_\varepsilon$ given by \eqref{poisson}  weakly solves
\begin{equation}\label{eqvepsext}
\begin{cases}
\Delta v^\e_\varepsilon= 0 & \text{in $\R^{n+1}_+$}\,,\\[10pt]
\displaystyle \frac{\partial v^\e_\varepsilon}{\partial \nu}=\frac{1}{\varepsilon}\big(1-|v^\e_\varepsilon|^2\big)v^\e_\varepsilon & \text{on $\omega$}\,.
\end{cases}
\end{equation} 
In view of Lemma \ref{hatH1/2toH1} and \eqref{defNeumOp}, the weak sense for this system corresponds to 
\begin{equation}\label{varformbdgleq}
\int_{\R^{n+1}_+} \nabla v^\e_\varepsilon \cdot\nabla\Phi\,\de x= \frac{1}{\varepsilon}\int_\omega \big(1-|v^\e_\varepsilon|^2\big)v^\e_\varepsilon \cdot\Phi\,\de \mathscr{H}^n
\end{equation}
for all $\Phi \in H^1(\R^{n+1}_+;\R^m)\cap L^4(\omega)$ compactly supported in $\R^{n+1}_+\cup\omega$. 
System \eqref{eqvepsext} also has  a variational structure. Indeed, 
considering an admissible  bounded  open set $\Om\subset \R^{n+1}_+$ such that $\overline{\partial^0\Omega}\subset \omega$, and setting for $u\in H^1(\Om;\R^m)\cap L^4(\partial^0\Omega)$, 
\begin{equation}\label{defEnergGLB}
E_\varepsilon(u,\Om):=\frac{1}{2}\int_\Om |\nabla u|^2\,\de x + \frac{1}{4\varepsilon}\int_{\partial^0\Omega} (1-|u|^2)^2\,\de \mathscr{H}^n\,,
\end{equation}
the weak formulation \eqref{varformbdgleq} can be rephrased as 
$$\left[\frac{\de}{\de t} E_\varepsilon\left(v^{\e}_\varepsilon+t\Phi,\Omega\right) \right]_{t=0} =0$$ 
for all $\Phi\in H^{1}(\Omega;\R^m)\cap L^4(\partial^0\Omega)$ compactly supported in $\Omega\cup\partial^0\Omega$. Hence $v^\e$ is a critical point of $E_\eps$ in $\Om$ for all admissible bounded  open set $\Om\subset \R^{n+1}_+$ such that $\overline{\partial^0\Omega}\subset \omega$.  The energy $E_\eps$ is what we may refer  to as {\it Ginzburg-Landau boundary energy}. 
\vskip5pt

If $v_\varepsilon$ turns out to be minimizing, 
we can transfer the minimality of $v_\varepsilon$ to $v^\e_\varepsilon$ with the help of Corollary~\ref{minenergdirchfrac}. In the following proposition, 
we say that $u\in H^1(\Om;\R^m)\cap L^4(\partial^0\Omega)$ is minimizer of $E_\varepsilon$ in $\Omega$ if 
$$E_\varepsilon(u,\Omega)\leq E_\varepsilon(u+\Phi,\Omega) $$ 
for all $\Phi\in H^1(\Om;\R^m)\cap L^4(\partial^0\Omega)$ compactly supported in $\Omega\cup\partial^0\Omega$. 

\begin{proposition}[\bf Minimality transfer]
Let $v_\varepsilon\in \widehat H^{1/2}(\omega;\R^m)\cap L^4(\omega)$ be a minimizer of $\mathcal{E}_\varepsilon$ in $\omega$, and let $v^\e_\varepsilon$  
be its harmonic extension to $\mathbb{R}^{n+1}_+$ given by \eqref{poisson}. Then $v^\e_\varepsilon$ is a minimizer of $E_\varepsilon$ in every admissible bounded  open set $\Omega\subset \R^{n+1}_+$  such that 
$\overline{\partial^0\Omega}\subset \omega$.  
\end{proposition}

Concerning regularity for system \eqref{eqfractGL}, it suffices to investigate the boundary regularity of the harmonic extension which satisfies \eqref{eqvepsext}. For this issue we can rely on the well known regularity theory for linear Neumann boundary value problems, see {\it e.g.} \cite{GT}.  The following theorem deals with regularity in the interior of the free boundary, and a simple proof can be found in~\cite[Lemma~2.2]{CSM} (which actually deals with scalar valued equations, but a quick inpection of the proof shows 
that the result still holds in the vectorial case).  

\begin{theorem}\label{regint}
Let $R>0$, and let  $u_\varepsilon\in H^1(B_{2R}^+;\mathbb{R}^m)\cap L^\infty(B_{2R}^+)$ be a weak solution of 
$$\begin{cases} 
\Delta u_\varepsilon=0 & \text{in $B_{2R}^+$}\,,\\[8pt]
\displaystyle \frac{\partial u_\varepsilon}{\partial\nu}= \frac{1}{\varepsilon}(1-|u_\varepsilon|^2)u_\varepsilon & \text{on $D_{2R}$}\,.
\end{cases}$$ 
Then $u\in C^\infty\big(B_R^+\cup D_R\big)$. 
\end{theorem}
 
 Applying Theorem \ref{regint} to system \eqref{eqvepsext}  yields interior regularity for 
 {\it bounded} variational solutions of the fractional Ginzburg-Landau system \eqref{eqfractGL}. 
 
 \begin{corollary}
 Let $v_\varepsilon \in  \widehat H^{1/2}(\omega;\R^m)\cap L^\infty(\R^n)$ be a weak solution of \eqref{eqfractGL}. Then $v_\varepsilon \in C^\infty(\omega)$. 
 \end{corollary}
  
Dealing with arbitrary critical points of  $\mathcal{E}_\varepsilon(\cdot, \omega)$ in $H^{1/2}_g(\omega;\R^m)$ for a 
smooth domain $\omega$ and a smooth (exterior) Dirichlet data $g$, one may wonder if regularity holds accross the boundary $\partial\omega$. 
If $v_\varepsilon$ is such a critical point and $v_\varepsilon^\e$ is its harmonic extension, we end up with the mixed boundary value problem  
\begin{equation}\label{equatsystmixbdval}
 \begin{cases}
 \Delta v^\e_\varepsilon= 0 & \text{in $\R^{n+1}_+$}\,,\\[10pt]
\displaystyle \frac{\partial v^\e_\varepsilon}{\partial \nu}=\frac{1}{\varepsilon}\big(1-|v^\e_\varepsilon|^2\big)v^\e_\varepsilon & \text{on $\omega$}\,,\\[10pt]
v^\e_\varepsilon = g & \text{on $\R^n\setminus \omega$}\,.
 \end{cases}
 \end{equation}
 Even if we obviously expect regularity in the interior of $\omega$, it might not hold across the edge~$\partial\omega$, 
since solutions of linear mixed boundary value problems are usually not better than H\"older continous (and 1/2 is in general the best possible H\"older exponent, see {\it e.g.} \cite{Sav} and the references therein). H\"older continuity for variational solutions of (non homogeneous) linear mixed boundary value problems follows from the general results of {\sc Stampacchia} \cite{Stamp1, Stamp2} (see also \cite{Chicc}), and 
 an estimate on the order of regularity is given by the classical result of {\sc Shamir}~\cite{Sham}. In our case, it is enough to prove an {\it a priori} global $L^\infty$-bound on $v_\varepsilon$ to derive from \cite{Sham} the regularity stated in 
 the following theorem. 
  
\begin{theorem}\label{regDirich}
Assume that $\partial\omega$ is smooth. Let $g\in C^1(\R^n;\R^m)\cap L^\infty(\R^n)$,  
and let $v_\varepsilon \in H_{g}^{1/2}(\omega;\R^m)\cap L^4(\omega)$ 
be a weak solution of \eqref{eqfractGL}. Then 
$v_\epsilon \in C^\infty(\omega)$ and $v_\varepsilon$ is $\alpha$-H\"older continuous  near $\partial\omega$ for every $\alpha\in(0,1/2)$.  
\end{theorem}  

 As mentioned in the few lines above, the proof of Theorem \ref{regDirich} rely on  the boundedness of $v_\eps$. Indeed, once the $L^\infty$-bound on $v_\varepsilon$ is obtained, 
 the right hand side in the Neumann boundary equation of \eqref{equatsystmixbdval}  remains bounded. This is enough to apply the regularity result of \cite{Sham}, and then deduce the H\"older 
 continuity of $v^\e_\varepsilon$ (and thus of $v_\epsilon$). The higher order regularity away from $\partial\omega$ follows from Theorem~\ref{regint}. 
 
 \begin{lemma}
 Let $\omega$ and $g$ be as in Theorem \ref{regDirich}. Let $v_\varepsilon \in H_{g}^{1/2}(\omega;\R^m)\cap L^4(\omega)$ 
be a weak solution of\eqref{eqfractGL}. Then $v_\varepsilon \in L^\infty(\R^n)$. 
 \end{lemma}
 
 \begin{proof}
 Let us fix for the whole proof an admissible bounded open set $\Omega\subset \R^{n+1}_+$  with smooth boundary such that  $\overline\omega\subset \partial^0\Omega$. We shall use an argument in the spirit of the proof by {\sc Brezis \& Kato} for the standard Laplacian. 
 
\noindent{\it Step 1.}  By \eqref{H1/2gH1/2loc} and Remark~\ref{H1/2loctoH1loc}, $v^\e_\varepsilon \in H^1_{\rm loc}(\overline{\R}^{n+1}_+;\R^m)$ 
and $v^\e_\varepsilon$ weakly solves \eqref{equatsystmixbdval}.  By standard elliptic regularity for the Dirichlet problem, 
 we have $v^\e_\varepsilon\in C^1(\overline{\R}^{n+1}_+\setminus \overline \omega)$. 
 Since ${\rm dist}(\partial^+\Omega,\overline\omega)>0$, we have 
 $$M:=\| v^\e_\varepsilon\|_{L^\infty(\partial^+\Omega)}+\|\nabla v^\e_\varepsilon\|_{L^\infty(\partial^+\Omega)}<\infty\,. $$
 
 Let us now consider the scalar function
 $$\eta:= \sqrt{|v^\e_\varepsilon|^2+\lambda} \in H^1(\Omega)\cap L^4(\omega)$$
with $\lambda:=\max(1,\|g\|^2_{L^\infty(\R^n\setminus\omega)})$, 
and  fix an arbitrary nonnegative function $\Phi \in C^\infty(\overline\Omega)$ with compact support in 
$\Omega\cup\omega\cup\partial^+\Omega$.  
 By the chain-rule formula for Sobolev functions we have 
$$
 \int_\Omega \nabla\eta \cdot \nabla \Phi\,\de x= \sum_{j=1}^{n+1} \int_\Omega \frac{(v^\e_\varepsilon\cdot\partial_j v^\e_\varepsilon)\partial_j\Phi}{\sqrt{|v^\e_\varepsilon|^2+\lambda}} \,\de x\,.$$
Since 
 $$ \frac{v^\e_\varepsilon }{\sqrt{|v^\e_\varepsilon|^2+\lambda}} \in H^1(\Omega;\R^m)\,,$$
 we deduce that
$$ \int_\Omega \nabla\eta \cdot \nabla \Phi\,\de x= 
 \int_\Omega \nabla v^\e_\varepsilon\cdot \nabla \left(\frac{ \Phi v^\e_\varepsilon }{\sqrt{|v^\e_\varepsilon|^2+\lambda}}\right) \,\de x 
 - \int_\Omega  \left( \frac{|\nabla v^\e_\varepsilon|^2 }{\sqrt{|v^\e_\varepsilon|^2+\lambda}}-  \sum_{j=1}^{n+1} \frac{|v^\e_\varepsilon \cdot\partial_j v^\e_\varepsilon|^2 }{(|v^\e_\varepsilon|^2+\lambda)^{3/2}}  
 \right) \Phi \,\de x\,.$$
On the other hand $\Phi\geq 0$, so that 
 $$ \int_\Omega \nabla\eta \cdot \nabla \Phi\,\de x \leq   \int_\Omega \nabla v^\e_\varepsilon\cdot \nabla \left(\frac{ \Phi v^\e_\varepsilon }{\sqrt{|v^\e_\varepsilon|^2+\lambda}}\right) \,\de x \,.$$
 Using equation \eqref{equatsystmixbdval}, we  infer that 
 $$\int_\Omega \nabla\eta \cdot \nabla \Phi\,\de x \leq 
 \int_{\partial^+\Omega}  \left(v_\varepsilon\cdot\frac{\partial v_\varepsilon^\e}{\partial\nu}\right)\frac{\Phi}{\sqrt{|v^\e_\varepsilon|^2+\lambda}}\,\de\mathscr{H}^n
 +\frac{1}{\varepsilon}\int_\omega(1-|v_\varepsilon^\e|^2)\frac{|v_\varepsilon^\e|^2\Phi}{\sqrt{|v^\e_\varepsilon|^2+\lambda}}\,\de \mathscr{H}^n\,,$$ 
 whence 
 \begin{equation}\label{inegKato}
 \int_\Omega \nabla\eta \cdot \nabla \Phi\,\de x \leq 
 \int_{\partial^+\Omega}  \left(v_\varepsilon\cdot\frac{\partial v_\varepsilon^\e}{\partial\nu}\right)\frac{\Phi}{\eta}\,\de \mathscr{H}^n 
 -\frac{1}{\varepsilon}\int_\omega\frac{|v_\varepsilon^\e|^2(\eta+\sqrt{1+\lambda})}{\eta}\,(\eta-\sqrt{1+\lambda})\Phi\,\de \mathscr{H}^n\,. 
 \end{equation}
 Then we conclude by approximation that \eqref{inegKato} actually holds for any nonnegative $\Phi \in H^1(\Omega)\cap L^4(\omega)$ satisfying 
 $\Phi=0$ $\mathscr{H}^n$-a.e. on $\partial^0\Omega\setminus\omega$.  
 \vskip5pt
 
 \noindent{\it Step 2.} Given $T>0$ and $\beta>0$, we  define  
 $$\rho:=\max\{\eta-\sqrt{2\lambda},0\} \,,\quad \rho_T:=\min(\rho,T)\,,\quad \Psi_{T,\beta}:= \rho^\beta_T\rho\,,\quad \Phi_{T,\beta}:= \rho^{2\beta}_T\rho\,.$$
 Those functions belong to $H^1(\Omega)\cap L^4(\omega)$ and vanish on $\partial^0\Omega\setminus\omega$. 
 Setting $\Omega_T:=\{0<\rho<T\}\cap \Omega$,  straightforward computations yield 
 $$\int_\Omega |\nabla \Psi_{T,\beta}|^2\,\de x=\int_\Omega \rho_T^{2\beta}|\nabla\eta|^2\,\de x+ (\beta^2+2\beta)\int_{\Omega_T} \rho^{2\beta}|\nabla\eta|^2\,\de x\,,$$
 and 
 $$\int_\Omega \nabla\eta\cdot\nabla\Phi_{T,\beta}\,\de x= \int_\Omega \rho_T^{2\beta}|\nabla \eta|^2\,\de x + 2\beta\int_{\Omega_T}\rho^{2\beta}|\nabla \eta|^2\,\de x\,.$$
 From this two equalities we now  infer that 
 $$ \int_\Omega |\nabla \Psi_{T,\beta}|^2\,\de x\leq (\beta+1)\int_\Omega \nabla\eta\cdot\nabla\Phi_{T,\beta}\,\de x\,.$$
 Using $\Phi_{T,\beta}$ as a test function in \eqref{inegKato},  we  deduce that
$$ \int_\Omega |\nabla \Psi_{T,\beta}|^2\,\de x\leq (\beta+1) 
 \int_{\partial^+\Omega}  \left(v_\varepsilon\cdot\frac{\partial v_\varepsilon^\e}{\partial\nu}\right)\frac{\rho_T^{2\beta}\rho}{\eta}\,\de \mathscr{H}^n 
 -\frac{\beta+1}{\varepsilon}\int_\omega\frac{|v_\varepsilon^\e|^2(\eta+\sqrt{1+\lambda})}{\eta}\,\rho_T^{2\beta}\rho^2\,\de \mathscr{H}^n\,,$$
 which leads to 
 $$\int_{\Omega}\left|\nabla\left( \rho_T^{\beta}\rho\right)  \right|^2\,\de x \leq (\beta+1)\mathscr{H}^n(\partial^+\Omega) M^{2\beta+2} \,.$$
 Applying the Poincar\'e Inequality in \cite[Corollary 4.5.2]{Zi} to $ \rho_T^{\beta}\rho$ yields
 $$\int_{\partial \Omega}\big|\rho_T^\beta\rho\big|^{2}\,\de x\leq C_{\Omega,\omega}(\beta+1)M^{2\beta+2} \,,$$
 for a constant $C_{\Omega,\omega}>0$ which only depends on $\Omega$ and $\omega$.  Next we let $T\to \infty$ in this last inequality to obtain 
 $$\|\rho\|_{L^{2(\beta+1)}(\partial\Omega)} \leq C_{\Omega,\omega}^{1/(2\beta+2)}(\beta+1)^{1/(2\beta+2)}M\,.$$
 Letting now $\beta\to\infty$ leads to $\|\rho\|_{L^\infty(\partial\Omega)}\leq M$, which in turn implies $v_\varepsilon\in L^\infty(\omega)$. 
 Since $v_\varepsilon=g$ outside~$\omega$, we have thus proved that $v_\varepsilon\in L^\infty(\R^n)$. 
 \end{proof}
 
We now deduce from the maximum principle the following  upper bound on the modulus of a critical point 
$v_\varepsilon$ satisfying an exterior Dirichlet condition.

\begin{corollary}\label{modless1}
Let $\omega$ and $g$ be as in Theorem \ref{regDirich}. Let $v_\varepsilon \in H_{g}^{1/2}(\omega;\R^m)\cap L^4(\omega)$ 
be a weak solution of~\eqref{eqfractGL}. Then $\|v_\varepsilon\|_{L^\infty(\R^n)}\leq \max(1,\|g\|_{L^\infty(\R^n\setminus\omega)})$. 
\end{corollary}
 
\begin{proof}
By Theorem \ref{regDirich}, $v^\e_\eps$ is smooth in $\R^{n+1}_+\cup\omega$ and continous in $\overline{\R}^{n+1}_+$. 
 We consider the function $m_\varepsilon:=\lambda^2-|v^\e_\varepsilon|^2$ with $\lambda:=\max(1,\|g\|_{L^\infty(\R^n\setminus\omega)})$. Then $m_\varepsilon$ satisfies
$$\begin{cases} 
-\Delta m_\varepsilon=2|\nabla v^\e_\varepsilon|^2\geq 0 & \text{in $\R^{n+1}_+$}\,,\\[8pt]
\displaystyle \frac{\partial m_\varepsilon}{\partial\nu}= -\frac{2}{\varepsilon}|v_\varepsilon^\e|^2(m_\varepsilon+1-\lambda^2) & \text{on $\omega$}\,,\\[8pt]
m_\varepsilon \geq 0 & \text{on $\R^n\setminus\omega$}\,.
\end{cases}
$$
Assume that $m_\varepsilon$ achieves its minimum over $\R^n$ at a point $x_0\in\omega$. Then $x_0$ is a point of maximum of $|v_\varepsilon|$, and hence $x_0$ is an absolute minima of 
$m_\varepsilon$ over the whole half space $\overline\R^{n+1}_+$ by \eqref{bdlinftyext}. If $m(x_0)< 0$ we obtain 
$\partial_\nu m_\varepsilon(x_0)> 0$. On the other hand, by the strong maximum maximum principle and the 
Hopf boundary lemma, we have $\partial_\nu m_\varepsilon(x_0)< 0$ which leads to contradiction.  
\end{proof}

 \begin{remark}
In the case where the exterior condition $g$ satisfies $\|g\|_{L^\infty(\R^n\setminus \omega)}\leq 1$, Corollary~\ref{modless1} provides the estimate 
$|v_\varepsilon|\leq 1$ which is very standard for the usual (local) Ginzburg-Landau equation. We emphasize that, in this case, we have  
$|v^\e_\varepsilon|\leq 1$. 
 \end{remark}

 \vskip10pt

 \section{$1/2$-harmonic maps: definitions, regularity, and examples}\label{1/2Harm}		       

\subsection{Definitions and regularity theory}\label{regtheo}
In this subsection we assume that  $\omega\subset \mathbb{R}^n$ is a bounded  open set with Lipschitz boundary. We start introducing the concept 
of {\it 1/2-harmonic map in $\omega$ with values into the sphere $\mathbb{S}^{m-1}$} as well as the related notion of {\it boundary harmonic map}. 
We shall then discuss their regularity in view of the existing theory for classical harmonic maps. 
To simplify slightly the presentation, we made the choice to focus on the case where the target manifold is a sphere. However definitions and results 
extend to  more general compact manifolds  without boundary, see Remark~\ref{generaltarget} at the end of this subsection.  

 \begin{definition}\label{defhalfharm}
 Let $v\in \widehat{H}^{1/2}(\omega;\mathbb{R}^{m})$ be such that $|v|=1$ a.e. in $\omega$. 
 We say that $v$ is {\it weakly 1/2-harmonic into $\mathbb{S}^{m-1}$ in $\omega$} if  
 \begin{equation}\label{criticconstrS}
 \left[\frac{\de}{\de t} \,\mathcal{E}\left(\frac{v+t\varphi}{|v+t\varphi|}\,,\omega\right)\right]_{t=0} =0
 \end{equation}
 for all $\varphi\in H^{1/2}_{00}(\omega;\mathbb{R}^m)\cap L^\infty(\omega)$ compactly supported in $\omega$. 
 \end{definition}

Writing explicitly  the criticality condition \eqref{criticconstrS}, we find the variational formulation of the Euler-Lagrange equation for  1/2-harmonic into 
$\mathbb{S}^{m-1}$ as stated in the following proposition. 

\begin{proposition}\label{caracthalfharm}
Let $v\in \widehat{H}^{1/2}(\omega;\mathbb{R}^{m})$ be such that $|v|=1$ a.e. in $\omega$. 
Then $v$ is a weak 1/2-harmonic into $\mathbb{S}^{m-1}$ in $\omega$ if and only if 
\begin{equation}\label{varformwkfracharm}
\big\langle  (-\Delta)^{\frac{1}{2}} v , \varphi\big\rangle_\omega =0 
\end{equation}
for all $\varphi \in H^{1/2}_{00}(\omega;\mathbb{R}^m)$  satisfying 
$$\varphi(x)\in T_{v(x)}\mathbb{S}^{m-1}\quad\text{a.e. in $\omega$}\,. $$
 \end{proposition}
 
 \begin{proof}
{\it Step 1.} We start proving that  \eqref{varformwkfracharm} implies \eqref{criticconstrS}. By the density of compactly supported smooth functions stated in \eqref{densitysmoothH1/200}, it is enough to consider the case where $\varphi\in \mathscr{D}(\omega;\mathbb{R}^m)$. Then we observe that $(\varphi\cdot v)v\in H^{1/2}_{00}(\omega;\R^m)$. Hence 
 $\psi := \varphi - (\varphi\cdot v)v$ belongs to  $H^{1/2}_{00}(\omega;\R^m)$, and it satisfies $\psi\cdot v= 0$ a.e. in~$\omega$.  
 Noticing that 
 $$\frac{v(x)+t\varphi(x)}{|v(x)+t\varphi(x)| }= v(x)+ t\psi(x) +O(t^2)\quad\text{as $t\to 0$}\,,$$
 we obtain by dominated convergence and assuming  \eqref{varformwkfracharm} that
$$  \left[\frac{\de}{\de t} \,\mathcal{E}\left(\frac{v+t\varphi}{|v+t\varphi|}\,,\omega\right)\right]_{t=0} =\big\langle  (-\Delta)^{\frac{1}{2}} v , \psi\big\rangle_\omega =0 \,,$$ 
and this first  step is complete. 
\vskip3pt

\noindent{\it Step 2.} We now prove the reverse implication. Let  $\varphi \in H^{1/2}_{00}(\omega;\mathbb{R}^m)$ be such that $\varphi\cdot v =0$ a.e. in~$\omega$. By standard 
truncations and cut-off arguments, we may assume without loss of generality that $\varphi$ is compactly supported in $\omega$ and that $\varphi\in L^\infty(\omega)$. Then, 
$$\frac{v(x)+t\varphi(x)}{|v(x)+t\varphi(x)| }=\frac{v(x)+t\varphi(x)}{\sqrt{1+t^2|\varphi(x)|^2} }= v(x)+ t\varphi(x) +O(t^2)\quad\text{as $t\to 0$}\,,$$
and we deduce again by dominated convergence and  assumption \eqref{criticconstrS} that 
$$ 0=\left[\frac{\de}{\de t} \,\mathcal{E}\left(\frac{v+t\varphi}{|v+t\varphi|}\,,\omega\right)\right]_{t=0} =\big\langle  (-\Delta)^{\frac{1}{2}} v , \varphi\big\rangle_\omega\,,$$
which ends the proof. 
 \end{proof}

 \begin{remark}\label{eq1/2harm}
 Let $v\in \widehat{H}^{1/2}(\omega;\mathbb{R}^{m})$  be a weak 1/2-harmonic map into $\mathbb{S}^{m-1}$ in $\omega$. In view of Proposition~\ref{caracthalfharm}, the Euler-Lagrange equation can be written in the form  
 \begin{equation}\label{EL1/2Harm}
  (-\Delta)^{\frac{1}{2}}  v  \ \bot  \ T_v\,\mathbb{S}^{m-1} \quad \text{ in $H^{-1/2}(\omega)$}\,,
 \end{equation}
 which is the clear analogue of the classical (weak) harmonic map system. Of course this equation is not completely explicit, but one can  rewrite it  computing the Lagrange multiplier associated to the constraint. In the case where $|v|=1$ a.e. in all of $\R^n$, it takes a quite simple form. 
Indeed, for $\varphi\in\mathscr{D}(\omega;\R^m)$, Proposition~\ref{caracthalfharm} yields 
 $$\big\langle  (-\Delta)^{\frac{1}{2}} v , \varphi\big\rangle_\omega =\big\langle  (-\Delta)^{\frac{1}{2}} v , (\varphi\cdot v)v \big\rangle_\omega\,.$$ 
Using the representation \eqref{deffraclap} and the pointwise identity (based on the fact that $|v|=1$), 
$$\big(v(x)-v(y)\big)\cdot\big((\varphi(x)\cdot v(x))v(x) - (\varphi(y)\cdot v(y))v(y)\big)= \frac{1}{2}\,|v(x)-v(y)|^2\big(\varphi(x)\cdot v(x)+\varphi(y)\cdot v(y)\big) \,, $$
we obtain
 $$\big\langle  (-\Delta)^{\frac{1}{2}} v , (\varphi\cdot v)v \big\rangle_\omega = \frac{\gamma_n}{2}\iint_{\omega\times\R^n} \frac{|v(x)-v(y)|^2}{|x-y|^{n+1}}\,v(x)\cdot\varphi(x)\,\de x\de y\,.$$
The identity above then leads to  
 $$  (-\Delta)^{\frac{1}{2}} v(x) = \left(\frac{\gamma_n}{2}\int_{\R^n} \frac{|v(x)-v(y)|^2}{|x-y|^{n+1}}\,\de y\right)v(x)\quad\text{in $\mathscr{D}'(\omega)$}\,,$$
 which is again in clear analogy with the standard harmonic map system into a sphere. 
 \end{remark}

We now introduce what we call {\it $\mathbb{S}^{m-1}$-boundary harmonic maps}. Those maps correspond to critical points of the Dirichlet energy under the constraint to be $\mathbb{S}^{m-1}$-valued on a prescribed portion of the boundary.

 \begin{definition}\label{defFBHarm}
Let  $\Omega\subset \R^{n+1}_+$ be an admissible bounded  open set, and let $\;u\in H^1(\Omega;\mathbb{R}^m)$ be such that $|u|=1$ $\mathscr{H}^n$-a.e. on $\partial^0 \Omega$. 
We say that $u$ is {\it weakly harmonic  in $\Omega$ with respect to the (partially) free boundary condition $u(\partial^0\Omega)\subset \mathbb{S}^{m-1}$} if  
\begin{equation}\label{weakformharmmap}
\int_\Omega \nabla u \cdot \nabla \Phi\,\de x=0
\end{equation}
for all $\Phi \in H^1(\Omega;\mathbb{R}^m)\cap L^\infty(\Omega)$ compactly supported  $\Omega\cup\partial^0 \Omega$ and satisfying 
\begin{equation}\label{condtestharmon}
\Phi(x)\in T_{u(x)}\mathbb{S}^{m-1} \quad\text{$\mathscr{H}^n$-a.e. on $\partial^0 \Omega$}\,.
\end{equation}
In short, we shall say that $u$ is a {\it weak $(\mathbb{S}^{m-1}, \partial^0\Om)$-boundary  harmonic map in $\Omega$}.  
\end{definition}

Defintion \ref{defFBHarm} can be motivated by the fact that 
$$\left[\frac{\de}{\de t} \left(\frac{1}{2} \int_\Omega |\nabla u_t|^2\,\de x \right)\right]_{t=0} = \int_\Omega \nabla u \cdot \nabla \Phi\,\de x$$ 
for variations $u_t$ of the form
$$u_t:=\frac{u+t\Phi}{\sqrt{1+t^2|\Phi|^2}}\,,$$
with $\Phi$ satisfying \eqref{condtestharmon}. Such variations $u_t$ belong to $H^1(\Omega;\mathbb{R}^m)$, and they satisfy the constraint $|u_t|=1$ $\mathscr{H}^n$-a.e. on 
$\partial^0 \Omega$ by \eqref{condtestharmon}.  Actually, if  $u \in H^1(\Omega;\mathbb{R}^m)$ is a  weak $(\mathbb{S}^{m-1}, \partial^0\Om)$-boundary 
harmonic map in~$\Omega$, a standard  truncation argument shows that \eqref{weakformharmmap} 
holds for any  $\Phi \in H^1(\Omega;\mathbb{R}^m)$ with compact support support in  $\Omega\cup\partial^0 \Omega$ and satisfying \eqref{condtestharmon}.  Moreover, choosing $\Phi$ with compact support in $\Omega$ in formula \eqref{weakformharmmap} shows that $\Delta u=0$ in $\Omega$.  
Integrating by parts in \eqref{weakformharmmap} then yields 
$$\int_{\partial^0 \Omega} \frac{\partial u}{\partial\nu}\cdot \Phi \,\de \mathscr{H}^n =0$$
for all smooth functions $\Phi$ compactly supported in  $\Omega\cup\partial^0 \Omega$ and satisfying \eqref{condtestharmon}. In general the integral above has of course to be 
understood in the $H^{-1/2}-H^{1/2}_{00}$ duality sense. Therefore, the weak $(\mathbb{S}^{m-1}, \partial^0\Om)$-boundary harmonic map system can be reformulated as 
\begin{equation}\label{eqweakFBCharm}
\begin{cases}
\Delta u = 0 & \text{in $\Omega$}\,,\\[8pt]
\displaystyle \frac{\partial u}{\partial\nu} \ \bot  \ T_u\,\mathbb{S}^{m-1} & \text{in $H^{-1/2}(\partial^0 \Omega)$}\,.
\end{cases}
\end{equation}
\vskip3pt

\begin{remark}[\bf Harmonic maps with partially free boundary]\label{boundharmVsFreebound}
In view of the discussion above, let us mention that $(\mathbb{S}^{m-1}, \partial^0\Om)$-boundary harmonic maps belong to a larger class of harmonic maps known in the literature as {\it harmonic maps with partially free boundary}, see \cite{Bal,DG,DS,DS2,GJ,Ham,HL,Sch} and references therein. In most studies, one considers a smooth compact Riemannian manifold $\mathcal{M}$ without boundary (that we can assume to be isometrically embedded in some Euclidean space by the Nash embedding Theorem), and $\mathcal{N}$ a smooth closed submanifold of $\mathcal{M}$. The boundary portion $\partial^0\Omega$ is called {\it the partially free boundary}, and $\mathcal{N}$ is {\it the supporting manifold}. Then $\mathcal{M}$-valued (weak) harmonic maps in $\Om$ with the partially free boundary condition $u(\partial^0\Om)\subset\mathcal{N}$ are defined as critical points of the Dirichlet energy  under the constraints  $u(x)\in\mathcal{M}$ for a.e. $x\in\Om$ and $u(x)\in\mathcal{N}$ for $\mathscr{H}^n$-a.e. $x\in\partial^0\Om$. For $(\mathbb{S}^{m-1}, \partial^0\Om)$-boundary harmonic maps, we may consider the submanifold 
$\mathcal{N}=\mathbb{S}^{m-1}$ of the target $\mathcal{M}=\R^m$. However, to apply known results on harmonic maps, the compactness of $\mathcal{M}$ is usually required. Similarly to \cite[Sec.~3]{Mos}, a way to avoid this problem is to consider {\it bounded} $(\mathbb{S}^{m-1}, \partial^0\Om)$-boundary harmonic maps, noticing that $\mathbb{S}^{m-1}$ can be viewed as a submanifold of a flat torus $\mathcal{M}=\mathbb{R}^m/r\mathbb{Z}^m$ with factor $r>2$.  
Indeed, if $u$ is a $(\mathbb{S}^{m-1}, \partial^0\Om)$-boundary harmonic map in $\Omega$ satisfying $\|u\|_{L^\infty(\Om)}<r/2$, then it can be considered as a $\mathcal{M}$-valued harmonic map. We finally point out that such $L^\infty$-bound is not too restrictive in many applications.  For instance, if we consider an additional 
Dirichlet boundary condition $u=g$ on $\partial^+\Om$ with $g:\overline{\partial^+\Om}\to\mathbb{S}^{m-1}$ smooth, then we usually obtain 
$\|u\|_{L^\infty(\Om)}\leq 1$ through some maximum principle. 
\end{remark}

 As a consequence of \eqref{defNeumOp}, Lemma \ref{repnormderfraclap}, and Proposition \ref{caracthalfharm}, we can relate 1/2-harmonic maps into $\mathbb{S}^{m-1}$  to $(\mathbb{S}^{m-1}, \partial^0\Om)$-boundary harmonic maps, as we already did for the fractional Ginzburg-Landau equation. 
  
 \begin{proposition}[\bf Criticality transfer]\label{fracweakharm}
 Let $v\in \widehat H^{1/2}(\omega;\mathbb{R}^{m})$  be a weak 1/2-harmonic map into $\mathbb{S}^{m-1}$ in~$\omega$,  
 and let $v^\e$ be its harmonic extension to $\mathbb{R}^{n+1}_+$ given by \eqref{poisson}.
 Then $v^\e$ is  a weak $(\mathbb{S}^{m-1}, \partial^0\Om)$-boundary  harmonic map in $\Omega$ for every admissible bounded  open set $\Omega\subset \R^{n+1}_+$  such that $\overline{\partial^0\Omega}\subset \omega$. 
  \end{proposition}

For what concerns regularity, let us emphasize that a full regularity result  cannot hold in general neither for weakly $1/2$-harmonic maps into $\mathbb{S}^{m-1}$ nor for weak $(\mathbb{S}^{m-1}, \partial^0\Om)$-boundary  harmonic maps due to topological constraints. In the case of boundary harmonic maps, the regularity issue is of course the regularity at the free boundary $\partial^0\Om$, since the system is the usual Laplace equation in $\Om$ (so that boundary harmonic maps are smooth in the interior of the domain). At the end of the next subsection, we shall give examples of  weakly $1/2$-harmonic maps from $\R^2$ into $\mathbb{S}^1$ which are not continuous at the origin (see Proposition~\ref{homog1/2harm}), the prototypical one being $v(x)=\frac{x}{|x|}$. Obviously, such examples do not exclude some partial regularity to hold, but dealing with general weak solutions we actually do not expect any reasonable kind of regularity by analogy with the classical harmonic map equation into a manifold and the famous counterexample of {\sc Rivi\`ere}~\cite{Riv} (of a weak harmonic map from the three dimensional ball into $\mathbb{S}^2$ which is everywhere discontinuous). In the context of $1/2$-harmonic maps into spheres, we are not aware of an analoguous counterexample to regularity, but we believe that it should exist.  
As it is the case for harmonic maps into a manifold, it is then reasonnable to ask for an extra assumption on a weak $1/2$-harmonic map to derive at least partial regularity. The (usual) assumption we make is 
either {\it energy minimality} or {\it stationarity}, {\it i.e.}, criticality under inner variations. We now recall these notions for $1/2$-harmonic maps and for boundary harmonic maps.

 \begin{definition}
 Let $v\in \widehat{H}^{1/2}(\omega;\mathbb{R}^{m})$ be such that $|v|=1$ a.e. in $\omega$. We say that $v$ is a {\it minimizing 1/2-harmonic map into $\mathbb{S}^{m-1}$} in $\omega$ if  
 $$\mathcal{E}(v,\omega)\leq \mathcal{E}(\tilde v,\omega)$$
 for all $\tilde v\in \widehat{H}^{1/2}(\omega;\mathbb{R}^{m})$ such that $|\tilde v|=1$ a.e. in $\omega$, and $\tilde v-v$ is compactly supported in~$\omega$. 
 \end{definition}

\begin{definition}\label{defFBminiHarm}
Let $\Omega\subset \R^{n+1}_+$ be an admissible bounded  open set, and let $u\in H^1(\Omega;\mathbb{R}^m)$ be such that $|u|=1$ $\mathscr{H}^n$-a.e. on $\partial^0\Omega$. 
We say that $u$ is a {\it minimizing harmonic map in $\Omega$ with respect to the (partially)  free boundary condition $u(\partial^0 \Omega)\subset \mathbb{S}^{m-1}$} if 
$$\frac{1}{2} \int_\Omega |\nabla u|^2\,\de x\leq \frac{1}{2} \int_\Omega |\nabla \tilde u|^2\,\de x$$
for all $\tilde u \in H^1(\Omega;\mathbb{R}^m)$  such that $|\tilde u|=1$ $\mathscr{H}^n$-a.e. on $\partial^0 \Omega$, and $\tilde u-u$ is compactly supported in $\Omega\cup\partial^0 \Omega$. 
In short, we shall say that $u$ is a {\it minimizing $(\mathbb{S}^{m-1}, \partial^0\Om)$-boundary harmonic map} in~$\Omega$. 
\end{definition}

Obviously, if $v\in \widehat{H}^{1/2}(\omega;\mathbb{R}^{m})$ is a minimizing 1/2-harmonic map into $\mathbb{S}^{m-1}$ in $\omega$, then $v$ is a weak  1/2-harmonic map into $\mathbb{S}^{m-1}$ in $\omega$. In the same way, minimizing $(\mathbb{S}^{m-1}, \partial^0\Om)$-boundary harmonic maps are weak $(\mathbb{S}^{m-1}, \partial^0\Om)$-boundary harmonic maps. As already pursued for the  fractional Ginzburg-Landau equation, the minimality of a 1/2-harmonic map can be transfered to its harmonic extension with the help of Corollary ~\ref{minenergdirchfrac}. 

\begin{proposition}[\bf Minimality transfer]\label{fracminharm}
Let $v\in \widehat H^{1/2}(\omega;\mathbb{R}^{m})$ be a minimizing 1/2-harmonic map into $\mathbb{S}^{m-1}$ in~$\omega$.  
Let $v^\e$ be the harmonic extension of 
 $v$ in $\mathbb{R}^{n+1}_+$ given by \eqref{poisson}. Then $v^\e$ is  a minimizing $(\mathbb{S}^{m-1}, \partial^0\Om)$-boundary  harmonic map in $\Omega$ for every admissible bounded  open set $\Omega\subset\R^{n+1}_+$  such that $\overline{\partial^0\Omega}\subset \omega$.   
\end{proposition}

As mentioned above, the regularity theory for harmonic maps also deals with stationary harmonic maps. 
For boundary harmonic maps, the stationarity criteria has to allow inner variations up to the partially free boundary $\partial^0\Om$,  
 and it leads to the following definition. 

\begin{definition}\label{defstation}
Let $\Omega\subset \R^{n+1}_+$ be an admissible bounded open set, and let $u\in H^1(\Omega;\mathbb{R}^m)$ be a weak $(\mathbb{S}^{m-1}, \partial^0\Om)$-boundary  harmonic map in $\Omega$.  
We say that $u$ is  {\it stationary} if
\begin{equation}\label{compderivstat}
\left[\frac{\de}{\de t} \left(\frac{1}{2}\int_\Omega \big|\nabla(u\circ \Phi_t)\big|^2\,\de x \right)\right]_{t=0} =0 
\end{equation}
for any differentiable 1-parameter group of smooth diffeomorphisms $\Phi_t:\R^{n+1}\to\R^{n+1}$ satisfying 
\begin{itemize}[leftmargin=22pt]
\item $\Phi_0={\rm id}_{\R^{n+1}}$; 
\item $\Phi_t(\R^n)\subset\R^n$; 
\item $\Phi_t-{\rm id}_{\R^{n+1}}$ is compactly supported, and  
${\rm supp}(\Phi_t-{\rm id}_{\R^{n+1}})\cap\overline{\R}^{n+1}_+\subset \Omega\cup\partial^0 \Omega$. 
\end{itemize}
\end{definition}

\begin{remark}\label{stresstens}
It is well known that the usual stationarity definition can be recast in terms of the infinitesimal generator $\mathbf{X}$ of the 1-parameter group $\{\Phi_t\}_{t\in\R}$ of diffeomorphisms (see {\it e.g.} \cite[Chapter~2.2]{Sim}).  In the free boundary context, a map $u\in H^1(\Omega;\mathbb{R}^m)$  is stationary in $\Omega$ up to the free boundary $\partial^0\Om$ if and only if 
\begin{equation}\label{statcondVectF}
\int_{\Om}\bigg( |\nabla u|^2{\rm div}\mathbf{X} -2\sum_{i,j=1}^{n+1} (\partial_i u\cdot\partial_j u)\partial_j\mathbf{X}_i\bigg)\,\de x =0
\end{equation}
for all vector fields $\mathbf{X}=(\mathbf{X}_1,\ldots,\mathbf{X}_{n+1})\in C^1(\overline{\Om};\R^{n+1})$ compactly supported in $\Om\cup\partial^0\Om$ and satisfying $\mathbf{X}_{n+1}=0$ on $\partial^0\Om$ (the left hand side of \eqref{statcondVectF} being exactly minus twice the value of the derivative in~\eqref{compderivstat}). Moreover, 
using the so-called {\it stress-energy tensor} $T=(T_i)_{i=1}^{n+1}$ given by 
$$(T_{i})_j(u):= |\nabla u|^2\delta_{ij} - 2(\partial_i u\cdot\partial_j u)\,,\quad 1\leq j\leq n+1\,,$$
identity \eqref{statcondVectF} can be rewritten as the following system of conservation laws
\begin{equation}\label{eqenergtens}
\begin{cases}
\displaystyle {\rm div}\, T_{i} = 0 & \text{in $\Om\,$ for each $i\in\{1,\ldots,n+1\}$}\,,\\
T_{i}\cdot \nu = 0 & \text{on $\partial^0\Om\,$ for each $i\in\{1,\ldots,n\}$}\,,
\end{cases}
\end{equation}
in the sense of distributions. 
\end{remark}

\begin{remark}\label{depbound}
For any $u\in H^1(\Om;\R^m)$ satisfying $\Delta u=0$ in $\Omega$, it turns out that the integral in the left hand side of \eqref{statcondVectF} only depends on the trace of $(\mathbf{X}_1,\ldots,\mathbf{X}_n)$ on $\partial^0\Om$. Indeed, if $u$ is harmonic in $\Omega$, then we have ${\rm div}\, T=0$ in $\Omega$. Therefore, if $\mathbf{X}$ vanishes on $\partial^0\Om$,  integrating by parts this integral shows  that it is equal to zero (no matter what the trace of $u$ on $\partial^0\Om$ is). 
\end{remark}

\begin{remark}
Note that minimizing boundary harmonic maps are always stationary, but a stationary boundary harmonic map might not be minimizing. On the other hand, if a boundary harmonic map is smooth enough  
up to the free boundary $\partial^0\Om$, then it is stationary. Indeed, under such a smoothness assumption a simple computation shows that a solution of \eqref{eqweakFBCharm} satisfies \eqref{eqenergtens}.   
\end{remark}

Concerning the 1/2-Dirichlet energy, we define the notion of stationarity in the usual way through inner variations. 

\begin{definition}\label{defstationfrac}
Let $v\in \widehat H^{1/2}(\omega;\mathbb{R}^{m})$ be a weak 1/2-harmonic map into $\mathbb{S}^{m-1}$ in~$\omega$. We say that $v$  is  {\it stationary} in $\omega$ if 
$$ \left[\frac{\de}{\de t} \,\mathcal{E}\big(v\circ\phi_t,\omega\big)\right]_{t=0} =0$$
for any differentiable 1-parameter group of smooth diffeomorphisms $\phi_t:\R^{n}\to\R^{n}$ satisfying 
\begin{itemize}[leftmargin=22pt]
\item $\phi_0={\rm id}_{\R^{n}}$; 
\item $\phi_t-{\rm id}_{\R^{n}}$ is compactly supported, and  
${\rm supp}(\phi_t-{\rm id}_{\R^{n}})\subset \omega$. 
\end{itemize}
\end{definition}

Another consequence of Corollary ~\ref{minenergdirchfrac} is that  stationarity for 1/2-harmonic maps implies stationarity for harmonic extensions as stated in the following proposition. 

\begin{proposition}[\bf Stationarity transfer]\label{implstat}
Let $v\in \widehat H^{1/2}(\omega;\mathbb{R}^{m})$ be a stationary weak 1/2-harmonic map into $\mathbb{S}^{m-1}$ in~$\omega$. Let $v^\e$ be the harmonic extension of 
 $v$ in $\mathbb{R}^{n+1}_+$ given by \eqref{poisson}. Then $v^\e$ is  a stationary weak $(\mathbb{S}^{m-1}, \partial^0\Om)$-boundary  harmonic map in $\Omega$ for every admissible bounded  open set $\Omega\subset \R^{n+1}_+$  such that $\overline{\partial^0\Omega}\subset \omega$. 
\end{proposition}

As matter of fact, Proposition \ref{implstat} is directly implied by the following lemma together with Remarks~\ref{stresstens} \& \ref{depbound}.

\begin{lemma}\label{compstatfrac}
Let $\Omega\subset \R^{n+1}_+$ be an admissible bounded  open set such that $\overline{\partial^0\Omega}\subset \omega$. Given $X\in C^1(\R^n;\R^{n})$ compactly supported in $\partial^0\Omega $, let $\{\phi_t\}_{t\in\R}$ be the flow  generated by $X$. For each $v\in \widehat H^{1/2}(\omega;\mathbb{R}^{m})$ we have 
$$ \left[\frac{\de}{\de t} \,\mathcal{E}\big(v\circ\phi_t,\omega\big)\right]_{t=0} = -\frac{1}{2}\int_{\Om}\bigg( |\nabla v^\e|^2{\rm div}\mathbf{X} -2\sum_{i,j=1}^{n+1} (\partial_i v^\e\cdot\partial_j v^\e)\partial_j\mathbf{X}_i\bigg)\,\de x\,,$$
where $\mathbf{X}=(\mathbf{X}_1,\ldots,\mathbf{X}_{n+1})\in C^1(\overline{\Om};\R^{n+1})$ is any vector field compactly supported in $\Om\cup\partial^0\Om$  
satisfying $\mathbf{X}=(X,0)$ on $\partial^0\Om$. 
 \end{lemma}
 
\begin{proof} 
Let $\mathbf{X}=(\mathbf{X}_1,\ldots,\mathbf{X}_{n+1})\in C^1(\overline{\Om};\R^{n+1})$ be an arbitrary  vector field  compactly supported in $\Om\cup\partial^0\Om$ and satisfying $\mathbf{X}=(X,0)$ on $\partial^0\Om$. Then 
consider a compactly supported $C^1$-extension of $\mathbf{X}$ to the whole space $\R^{n+1}$, still denoted by $\mathbf{X}$, such that $\mathbf{X}=(X,0)$ on $\R^n$.  We define $\{\Phi_t\}_{t\in\R}$ to be the flow generated by $\mathbf{X}$, {\it i.e.}, for each $x\in\R^{n+1}$ the map $t\mapsto \Phi_t(x)$ is defined as the solution of 
$$\begin{cases}
\displaystyle\frac{\de}{\de t} \Phi_t(x)= \mathbf{X}\big(\Phi_t(x)\big) \,,\\
\Phi_0(x)=x\,.
\end{cases}
$$
One can easily check that the family $\{\Phi_t\}_{t\in\R}$ is admissible for Definition \ref{defstation}.  
Noticing that $\Phi_t=(\phi_t,0)$ on~$\R^n$, 
 we now infer from Corollary~\ref{minenergdirchfrac} that 
$$\frac{1}{2}\int_\Om \big|\nabla(v^\e\circ\Phi_t)\big|^2\,\de x-\frac{1}{2}\int_\Om |\nabla v^\e|^2\,\de x\geq \mathcal{E}\big(v\circ\phi_t,\omega\big)-\mathcal{E}(v,\omega)\,.$$
Dividing both sides of this inequality by $t>0$, and then letting $t\to 0$, we obtain 
\begin{equation}\label{resultstat}
 -\frac{1}{2}\int_{\Om}\bigg( |\nabla v^\e|^2{\rm div}\mathbf{X} -2\sum_{i,j=1}^{n+1} (\partial_i v^\e\cdot\partial_j v^\e)\partial_j\mathbf{X}_i\bigg)\,\de x\geq  \left[\frac{\de}{\de t} \,\mathcal{E}\big(v\circ\phi_t,\omega\big)\right]_{t=0} \,.
 \end{equation}
Here, straightforward  computations yield
\begin{multline}\label{innervar}
 \left[\frac{\de}{\de t} \,\mathcal{E}\big(v\circ\phi_t,\omega\big)\right]_{t=0}= \frac{(n+1)\gamma_n}{4}\iint_{\omega\times\omega}\frac{ |v(x)-v(y)|^2}{|x-y|^{n+1}}\frac{(x-y)\cdot(X(x)-X(y))}{|x-y|^2}\,\de x\de y\\
 + \frac{(n+1)\gamma_n}{2}\iint_{\omega\times\omega^c}\frac{ |v(x)-v(y)|^2}{|x-y|^{n+1}}\frac{(x-y)\cdot X(x)}{|x-y|^2}\,\de x\de y\\
  -\frac{\gamma_n}{2}\iint_{\omega\times\R^n}\frac{ |v(x)-v(y)|^2}{|x-y|^{n+1}} {\rm div}\,X(x)\,\de x\de y\,.
 \end{multline}
Since $X$ was chosen arbitrary, inequality \eqref{resultstat} holds with $(-X)$ and $(-\mathbf{X})$ instead of $X$ and $\mathbf{X}$ respectively, and we conclude that  inequality \eqref{resultstat}  is in fact an equality. 
\end{proof}

\begin{remark}\label{remsmoothstat}
From formula \eqref{innervar}, one can actually deduce that any sufficiently smooth 1/2-harmonic map $v$ in $\omega$ is stationary.  
\begin{proof}
To see this,  we write the first two terms in the right hand side of \eqref{innervar} as $A$ and $B$, respectively. We have
$$A=\lim_{\delta\downarrow 0}\frac{\gamma_n}{2}\int_\omega\left( \int_{\omega\setminus D_\delta(y)}|v(x)-v(y)|^2X(x)\cdot\nabla_x\big(\frac{-1}{|x-y|^{n+1}}\big)\,\de x\right)\,\de y=:\lim_{\delta\downarrow 0} A_\delta\,.$$
Since $X$ is compactly supported in $\omega$, we can write for $\delta>0$ small enough
$$B= \frac{\gamma_n}{2}\int_{\omega^c}\left( \int_{\omega\setminus D_\delta(y)}|v(x)-v(y)|^2X(x)\cdot\nabla_x\big(\frac{-1}{|x-y|^{n+1}}\big)\,\de x\right)\,\de y\,.$$
Given $y\in\R^n\setminus\partial\omega$ and $\delta>0$ small enough, we  integrate by parts to find that 
\begin{multline*}
\int_{\omega\setminus D_\delta(y)}|v(x)-v(y)|^2X(x)\cdot\nabla_x\big(\frac{-1}{|x-y|^{n+1}}\big)\,\de x= \int_{\omega\setminus D_\delta(y)}\frac{|v(x)-v(y)|^2}{|x-y|^{n+1}}{\rm div}\,X(x)\,\de x \\
+ 2\int_{\omega\setminus D_\delta(y)}\frac{v(x)-v(y)}{|x-y|^{n+1}}\cdot \big(X(x)\cdot\nabla v(x)\big)\,\de x +\frac{\chi_\omega(y)}{\delta^{n+1}}\int_{\partial D_\delta(y)} |v(x)-v(y)|^2\frac{(x-y)\cdot X(x)}{|x-y|}\,\de \mathscr{H}^{n-1}_x\,,
\end{multline*}
where $\chi_\omega$ denotes the characteristic function of the set $\omega$. Therefore, for $\delta$ small we have 
\begin{multline}\label{id1645}
A_\delta+B= \frac{\gamma_n}{2}\int_{\R^n} \int_{\omega\setminus D_\delta(y)}\frac{|v(x)-v(y)|^2}{|x-y|^{n+1}}{\rm div}\,X(x)\,\de x\,\de y\\
+\gamma_n \int_{\R^n}\int_{\omega\setminus D_\delta(y)}\frac{v(x)-v(y)}{|x-y|^{n+1}}\cdot \big(X(x)\cdot\nabla v(x)\big)\,\de x\,\de y\\
+\frac{\gamma_n}{2\delta^{n+1}}\int_\omega 
\int_{\partial D_\delta(y)} |v(x)-v(y)|^2\frac{(x-y)\cdot X(x)}{|x-y|}\,\de \mathscr{H}^{n-1}_x\,\de y=:I_\delta+II_\delta+III_\delta\,.
\end{multline}
Obviously,
$$\lim_{\delta\downarrow 0} I_\delta =\frac{\gamma_n}{2}\iint_{\omega\times\R^n}\frac{ |v(x)-v(y)|^2}{|x-y|^{n+1}} {\rm div}\,X(x)\,\de x\de y\,.$$
Then we rewrite
$$II_\delta= \gamma_n \int_{\omega}\left(\int_{\R^n\setminus D_\delta(x)}\frac{v(x)-v(y)}{|x-y|^{n+1}}\,\de y\right)\cdot \big(X(x)\cdot\nabla v(x)\big)\,\de x\,,$$
so that 
$$\lim_{\delta\downarrow 0} II_\delta = \int_\omega (-\Delta)^{\frac1 2}v(x)\cdot \big(X(x)\cdot\nabla v(x)\big)\,\de x\,.$$
Finally, expanding $v$ around $y$, we easily get that
$$\lim_{\delta\downarrow 0} III_\delta=  \frac{\gamma_n}{2}\int_\omega\left( \int_{\mathbb{S}^{n-1}} |\sigma\cdot\nabla v(y)|^2X(y)\cdot\sigma\,\de\mathscr{H}^{n-1}_\sigma \right)\,.$$
We claim that the integral above vanishes, {\it i.e.}, $\lim_{\delta} III_\delta =0$. Indeed, write for $y\in\omega$ fixed, 
$$\int_{\mathbb{S}^{n-1}} |\sigma\cdot\nabla v(y)|^2X(y)\cdot\sigma\,\de\mathscr{H}^{n-1}_\sigma=\sum_{i=1}^m
\int_{\mathbb{S}^{n-1}} |\overrightarrow{a}_i\cdot\sigma|^2\,\overrightarrow{b}\cdot\sigma\,\de\mathscr{H}^{n-1}_\sigma\,,$$
where we have set $\overrightarrow{a}_i:=\nabla v_i(y)$ and $\overrightarrow{b}:=X(y)$. Given $i\in\{1,\ldots,m\}$,  we can assume  that $\overrightarrow{a}_i$ and $\overrightarrow{b}$ belongs to $\R^2\times\{(0,\ldots,0)\}$ by invariance under rotation. Then, 
$$\int_{\mathbb{S}^{n-1}} |\overrightarrow{a}_i\cdot\sigma|^2\overrightarrow{b}\cdot\sigma\,\de\mathscr{H}^{n-1}_\sigma= C_n\int_{\mathbb{S}^{1}} |\overrightarrow{a}_i\cdot\sigma|^2\overrightarrow{b}\cdot\sigma\,\de\mathscr{H}^{1}_\sigma \,,$$
for a dimensional constant $C_n$. Now observe that the function in the right hand side is a homogeneous polynomial in $\sigma$ of degree~$3$. Hence its integral over $\mathbb{S}^{1}$ vanishes, and  the claim is proved. 

Gathering \eqref{innervar} with \eqref{id1645}, we have thus shown that 
\begin{equation}\label{tim1652}
 \left[\frac{\de}{\de t} \,\mathcal{E}\big(v\circ\phi_t,\omega\big)\right]_{t=0}= \int_\omega (-\Delta)^{\frac1 2}v(x)\cdot \big(X(x)\cdot\nabla v(x)\big)\,\de x\,.
 \end{equation}
Finally, since $v(x)\in \mathbb{S}^{n-1}$ for $x\in\omega$, we have $\partial_j v(x)\in T_{v(x)} \mathbb{S}^{n-1}$ for $x\in\omega$ and $j=1,\ldots,n$. Then the Euler-Lagrange equation \eqref{EL1/2Harm} 
yields $(-\Delta)^{\frac1 2}v\cdot \partial_j v=0$ in $\omega$ for $j=1,\ldots,n$. Therefore the function under the integral in \eqref{tim1652} vanishes in $\omega$, which shows that $v$ is stationary. 
\end{proof}
\end{remark}

Going back to our discussion on the regularity of harmonic maps, we now mention that there is only one case where  full (interior) regularity holds for  general critical points. It is the case where the starting dimension equals two (the conformal dimension), and this is a well known result due to {\sc H\'elein}~\cite{Hel}. In higher dimensions, the (optimal) partial regularity result for minimizing harmonic maps has been obtained 
in the pioneering work of {\sc Schoen \& Uhlenbeck}~\cite{SU}. It has then been extended to stationary harmonic maps by {\sc Bethuel}~\cite{Bet} (see also \cite{Ev}). All this results have an analogue in the context  of harmonic maps with partially free boundary, where the new issue is of course to determine the
 partial regularity at the relative interior of the free boundary $\partial^0\Om$. Higher order regularity starting from continuous solutions has been studied and proved by {\sc Gulliver \& Jost} \cite{GJ}, see also \cite{Bal,Ham}. The minimizing case has been handled independently by {\sc Hardt \& Lin} \cite{HL} and {\sc Duzaar \& Steffen} \cite{DS,DS2}, while full regularity in the conformal dimension and partial regularity under the stationarity condition has been more recently proved by {\sc Scheven} \cite[Theorem~2.2, Theorem 4.1]{Sch}. All these results on partial regularity essentially deals with an estimate on the 
 Hausdorff dimension of the so-called {\it singular set}. The singular set ${\rm sing}(u)$ of a map $u$  is usually defined as the complement of the largest (relative) open set on which $u$ is continuous. It is therefore a relatively closed subset of the domain where $u$ is defined. For a harmonic map with partially free boundary it is then a relatively closed subset of $\Om\cup\partial^0\Om$, and in case of a $(\mathbb{S}^{m-1}, \partial^0\Om)$-boundary  harmonic map it is a relatively closed subset of $\partial^0\Om$ (since regularity holds in $\Om$). In view of Remark~\ref{boundharmVsFreebound}, we can obtain a partial regularity theory for {\it bounded} $\mathbb{S}^{m-1}$-boundary  harmonic maps from the known results  mentioned above about harmonic maps with partially free boundary. In the statements below, ${\rm dim}_{\mathscr{H}}$ denotes the Hausdorff dimension.

 \begin{theorem}[\cite{DS,DS2,GJ,HL,Sch}]\label{thmregfreebd}
 Let $\Omega\subset \R^{n+1}_+$ be an admissible bounded  open set. If $u\in H^1(\Omega;\mathbb{R}^m)\cap L^\infty(\Om)$ is a weak $(\mathbb{S}^{m-1}, \partial^0\Om)$-boundary  harmonic map in 
 $\Omega$, then $u\in C^\infty\big((\Omega\cup\partial^0 \Omega) \setminus {\rm sing}(u)\big)$.  
Moreover, 
\begin{itemize}[leftmargin=22pt] 
\item[\rm (i)] if $n=1$, then ${\rm sing}(u)=\emptyset$; 
\vskip3pt

\item[\rm (ii)] if $n\geq 2$ and $u$ is stationary, then  $\mathscr{H}^{n-1}({\rm sing}(u))=0$; 
\vskip3pt

\item[\rm (iii)] if $u$ is minimizing, then  ${\rm dim}_{\mathscr{H}}({\rm sing}(u))\leq n-2$ for $n\geq 3$, and ${\rm sing}(u)$ is discrete for $n=2$. 
 \end{itemize}
 \end{theorem}
 
 \begin{proof}
 As indicated in Remark~\ref{boundharmVsFreebound}, the sphere $\mathbb{S}^{m-1}$ can be considered as a submanifold of a flat torus $\mathcal{M}:=\R^m/r\Z^m$ with $r>2\|u\|_{L^\infty(\Omega)}$. Then $u$ is a weakly harmonic map in $\Omega$ into the compact (boundaryless) manifold $\mathcal{M}$ satisfying the partially free boundary 
condition $u(\partial^0\Omega)\subset \mathbb{S}^{m-1}$. Items (i) and (ii) follow from \cite[Theorem 2.2, Theorem 4.1]{Sch} (note that for $n=1$, stationarity is not required since the conclusion of \cite[Corollary 1.3]{Sch} trivially holds in this case), while item (iii) is proved in \cite{DS,DS2, HL}. Finally, the higher order regularity away from the singular set is obtained from \cite[Section 4]{GJ}. 
 \end{proof}
 
Thanks to Theorem~\ref{thmregfreebd}, we can now deduce the partial regularity theory for {\it bounded} 1/2-harmonic maps into $\mathbb{S}^{m-1}$.

\begin{proof}[Proof of Theorem \ref{reghalfharmintro}]
 Let $v\in \widehat H^{1/2}(\omega;\mathbb{R}^{m})\cap L^\infty(\R^n)$ be a weak 1/2-harmonic map into $\mathbb{S}^{m-1}$ in~$\omega$, and let  
 $\Omega\subset \R^{n+1}_+$ be an arbitray admissible bounded  open set such that $\overline{\partial^0\Omega}\subset \omega$. 
 By Proposition \ref{fracweakharm}, $v^\e$ is  a weak $(\mathbb{S}^{m-1}, \partial^0\Om)$-boundary  harmonic map in $\Omega$. If $n=1$, then 
 Theorem \ref{thmregfreebd} shows that $v^\e\in C^\infty(\Omega\cup\partial^0\Om)$, whence $v\in C^\infty(\omega)$ by arbitrariness of $\Om$. Let us now assume that $n\geq 2$, and that $v$ is stationary. By Proposition \ref{implstat}, 
 $v^\e$ is stationary. Since ${\rm sing}(v)\cap\partial^0\Om= {\rm sing}(v^\e)\cap\partial^0\Om$,  Theorem \ref{thmregfreebd} yields $\mathscr{H}^{n-1}({\rm sing}(v)\cap\partial^0\Omega)=0$, and the conclusion follows by monotone convergence letting $\partial^0\Omega\uparrow\omega$. In the case where $v$ is minimizing, the argument follows the same lines using Proposition \ref{fracminharm}. 
\end{proof}

\begin{remark}
In recent papers, {\sc Da Lio \& Rivi\`ere} have obtained a direct proof for the full regularity of weak $1/2$-harmonic maps in the case  $\omega=\R$, first for a sphere target in \cite{DaLRi}, and then for more general target manifolds in \cite{DaLRi2}. A main point of their proof is to rewrite the Euler-Lagrange equation \eqref{EL1/2Harm}  in a form that exhibits a special algebraic structure which allows to use some compensated compactness arguments (somehow in the spirit of \cite{Hel}). Even if we have deduced here the regularity theory for 1/2-harmonic maps  from the existing literature, we believe that the analysis in \cite{DaLRi,DaLRi2} could be of first importance for further investigations. 
\end{remark}

\begin{remark}
The relation between 1/2-harmonic maps and harmonic maps with partially free boundary was first noticed by {\sc Moser} \cite{Mos} 
for a different, non-explicit, integro-differential operator (which coincides with $ (-\Delta)^{\frac{1}{2}} $ only in the case $\omega=\R^n$). Roughly speaking, the Dirichlet-to-Neumann operator considered in~\cite{Mos} is associated to a "homogeneous Neumann type condition" on the exterior of $\omega$, while we are dealing with Dirichlet exterior conditions. In this Neumann framework and under an additional technical condition, \cite{Mos} provides a similar partial regularity result  from the regularity theory for harmonic maps with partially free boundary.
\end{remark}

 \begin{remark}
In the context of harmonic map with partially free boundary, {\sc Duzaar  \& Grotowski} \cite{DG} have studied regularity for the mixed boundary value problem which consists in prescribing a 
Dirichlet data on the remaining part $\partial^+\Om$ of the boundary of $\Omega$. Under suitable smoothness conditions on the Dirichlet data, $\partial^+\Om$ and $\partial^0\Om$, {\it and assuming that 
$\partial^+\Om$ and $\partial^0\Om$ meet orthogonaly}, they have proved that minimizing solutions are continuous accross the edge $\overline{\partial^+\Om}\cap\overline{\partial^0\Om}$. 
Unfortunately, this result cannot be used for 1/2-harmonic maps with exterior Dirichlet condition (even for minimizing maps) since the two parts of the boundary would have to meet {\it tangentially}. 
Up to our knowledge, there is no regularity result at the boundary $\partial\omega$ for the 1/2-harmonic map problem with a prescribed exterior Dirichlet condition.    
 \end{remark}

\begin{remark}[\bf Energy Monotonicity]\label{remmonotbdharm} For usual harmonic maps, it is well known that the stationarity property yields a monotonic control on the energy on balls with respect to the radius.  For a stationary boundary harmonic map $u\in H^1(\Om;\R^m)$, the relation \eqref{statcondVectF} leads to the monotonicity of the function 
$$r\in\big(0,{\rm dist}(x_0,\partial^+\Om)\big)\mapsto \frac{1}{2r^{n-1}}\int_{B^+_r(x_0)}|\nabla u|^2\,\de x \,,$$
for every $x_0\in\partial^0\Om$. 
\end{remark}

\begin{remark}[\bf Liouville property] 
As a consequence of the energy monotonicity above, 
if $n\geq 2$ and $u\in \dot H^1(\R^{n+1}_+;\R^m)$ is an entire  stationary boundary
harmonic map ({\it i.e.}, a stationary boundary harmonic map in $B_r^+$ for every $r>0$), then $u$ is constant. 
As a byproduct, if $n\geq 2$ and $v\in \dot H^{1/2}(\R^n;\R^m)$ is an entire stationary 1/2-harmonic map ({\it i.e.}, a stationary 1/2-harmonic map in $D_r$ for every $r>0$), 
then  $v$ is constant. 
\end{remark} 
 
\begin{remark}[\bf General target manifold]\label{generaltarget}
The definition of 1/2-harmonic maps  extends to a more general target submanifold $\mathcal{N}\subset \R^m$ in the following way. Assuming that $\mathcal{N}$ is smooth, compact, and without boundary, the nearest point retraction $\pi_{\mathcal{N}}$ on  $\mathcal{N}$ is smooth in a small tubular neighborhood of $\mathcal{N}$. We then say that a map $v\in \widehat{H}^{1/2}(\omega;\mathbb{R}^{m})$ satisfying $v(x)\in\mathcal{N}$ for a.e. $x\in\omega$,   is  weakly  1/2-harmonic into $\mathcal{N}$ in $\omega$ if 
 $$\left[\frac{\de}{\de t} \,\mathcal{E}\left(\pi_\mathcal{N}(v+t\varphi),\omega\right)\right]_{t=0} =0 $$
 for all $\varphi\in H^{1/2}_{00}(\omega;\mathbb{R}^m)\cap L^\infty(\omega)$ compactly supported in $\omega$. The associated Euler-Lagrange equation reads
 $$ (-\Delta)^{\frac{1}{2}} v \perp T_v\mathcal{N} \quad\text{in $H^{-1/2}_{00}(\omega)$}\,,$$
{\it i.e.}, $\big\langle  (-\Delta)^{\frac{1}{2}} v,\varphi\big\rangle_\omega =0$ for all $\varphi\in H^{1/2}_{00}(\omega;\R^m)$ satisfying $\varphi(x)\in T_{v(x)}\mathcal{N}$ for a.e. $x\in\omega$. 

Concerning the definition of $(\mathcal{N},\partial^0\Omega)$-valued boundary harmonic maps, we simply reproduce Definition~\ref{defFBHarm} replacing $\mathbb{S}^{m-1}$ by the manifold $\mathcal{N}$. The notions of minimality and stationarity for both  boundary harmonic and 1/2-harmonic maps remain unchanged. With these definitions, all the results of this subsection do hold. In particular, we have the same partial regularity theory for $(\mathcal{N},\partial^0\Omega)$-boundary  harmonic maps  as stated in Theorem \ref{thmregfreebd} by the general  regularity results in \cite{HL,DS,DS2,Sch}. As a consequence, we have the same partial regularity theory for 1/2-harmonic maps into $\mathcal{N}$ as stated in  Theorem~\ref{reghalfharmintro}. 

\end{remark}

\subsection{1/2-harmonic lines into $\mathbb{S}^1$} 
 
 We provide in this subsection some explicit examples of 1/2-harmonic maps into $\mathbb{S}^1$ enlightening the geometric flavour of the 1/2-harmonic map equation, 
 as well as its analogy with usual harmonic maps. We start with an explicit representation formula for  {\it all entire 1/2-harmonic maps} from the line $\R$ into $\mathbb{S}^1$ with finite energy.  In the sequel, we identify $\R^2$ with the complex plane $\C$ writing $z=x_1+ix_2$. The open unit disc of $\C$ is denoted by   $\mathbb{D}$, and $\partial\D$ is identified with $\mathbb{S}^1$.

 \begin{theorem}\label{classif1/2harm}
 Let $v\in \dot H^{1/2}(\R;\mathbb{S}^1)$ be a non-constant entire 1/2-harmonic map into $\mathbb{S}^1$, and let $v^\e$ be its harmonic extension to $\mathbb{R}^{2}_+$ given by \eqref{poisson}.  
 There exist some $d\in\N$, $\theta\in\R$, $\{\lambda_k\}_{k=1}^d\subset (0,\infty)$, and $\{a_k\}_{k=1}^d\subset\R$ such that  $v^\e(z)$ or its complex conjugate equals
 \begin{equation}\label{formfracharm}
 e^{i\theta}\prod_{k=1}^d\frac{\lambda_k(z-a_k)-i}{\lambda_k(z-a_k)+i}\,.
 \end{equation}
 In addition,  
 \begin{equation}\label{quantiz}
 \mathcal{E}(v,\R)=[v]^2_{H^{1/2}(\R)}=\frac{1}{2}\int_{\R^2_+}|\nabla v^\e|^2\,dz=\pi d\,.
 \end{equation}
\end{theorem}
 
 \begin{remark}
Theorem \ref{classif1/2harm} shows that the map 
 $$x\in\R\mapsto \left(\frac{x^2-1}{x^2+1},\frac{-2x}{x^2+1}\right)\in \mathbb{S}^1$$
 is a 1/2-harmonic map. Indeed, it corresponds to the case $\theta=0$, $d=1$, $\lambda_1=1$, and $a_1=0$ in \eqref{formfracharm}. Note that this map is precisely the inverse of the  stereographic projection with pole at $(1,0)$ going from the circle into the line. 
 \end{remark}

 The proof of Theorem~\ref{classif1/2harm} is  based on an observation due to {\sc Mironescu \& Pisante} \cite{MirPis} together with  the following preliminary lemma. 
 
\begin{lemma}\label{conform}
 Let $v\in \dot H^{1/2}(\R;\R^2)$ be a nontrivial entire 1/2-harmonic map into $\mathbb{S}^1$, and let $v^\e$ be its harmonic extension to $\mathbb{R}^{2}_+$ given by \eqref{poisson}. 
  Then $v^\e$ is either a conformal or an anti-conformal transformation of $\overline{\R}^2_+$ into the closed unit disc $\overline{\mathbb{D}}$. 
  \end{lemma}

\begin{proof}
First recall that $v^\e$ is smooth in $\overline{\R}^2_+$ by Theorem~\ref{thmregfreebd}, and $v^\e$ takes values in $\overline{\mathbb{D}}$ by  \eqref{bdlinftyext}. 
Let us now consider the Hopf differential $\mathcal{H}$ of $v^\e$ defined by
$$ \mathcal{H}(z):=\big(|\partial_1v^\e|^2-|\partial_2v^\e|^2\big)-2i(\partial_1v^\e\cdot\partial_2v^\e)\,.$$
It is well-known that $v^\e$ is conformal or anti-conformal if and only if $\mathcal{H}\equiv0$. 
We thus have to prove that $\mathcal{H}$ vanishes identically. 

First notice that $\mathcal{H}$ is holomorphic since $v^\e$ is harmonic. Since $v(x)\in\mathbb{S}^1$ we find that $\partial_1v^\e(x)\in T_{v(x)}\mathbb{S}^1$ for every $x\in\R$.  By the boundary equation in \eqref{eqweakFBCharm}, it implies that  $\partial_1v^\e\cdot\partial_2v^\e$ vanishes on $\R$. Setting $g$ to be the imaginary part of $\mathcal{H}$, the function $g$ is  harmonic in $\R^2_+$ by holomorphicity of $\mathcal{H}$. Since $g$ vanishes on $\R$, we  can extend $g$ to the whole plane by odd reflection ({\it i.e.}, setting $g(z):=-g(\overline z)$ if ${\rm Im}(z)<0$), and the resulting function (still denoted by $g$) is harmonic in~$\R^2$.  On the other hand, 
 $v^\e\in \dot H^1(\R^2_+;\R^2)$ by Lemma \ref{normexth1/2}. Hence $g \in L^1(\R^2)$, which leads to $g\equiv 0$ by the mean value property of harmonic functions.  Therefore $\mathcal{H}$ takes real values, and it must be constant by holomorphicity. We then conclude that $\mathcal{H}\equiv 0$ since $\mathcal{H}\in L^1(\R^2_+)$. 
\end{proof}

\begin{remark}[{\bf Minimal surfaces}]\label{minimsurf}
According to Remark \ref{generaltarget}, Lemma \ref{conform} still holds if $\mathbb{S}^1$ is replaced by a smooth compact manifold $\mathcal{N}\subset\R^m$ without boundary, {\it i.e.},  if $v\in \dot H^{1/2}(\R;\R^m)$ is  a nontrivial entire 1/2-harmonic map into $\mathcal{N}$, then $v^\e$ is conformal or anti-conformal. If $m=3$ and $\mathcal{N}$ is a surface, it shows that  the image of $\R^2_+$ by $v^\e$ is a minimal surface in $\R^3$ whose boundary lies in $\mathcal{N}$, and meets   $\mathcal{N}$ orthogonaly, see \cite{struw}. If $\mathcal{N}$ is a closed boundary curve, then the image of $\R^2_+$ by $v^\e$ is a minimal surface spanned by $\mathcal{N}$. 
\end{remark}

\begin{proof}[Proof of Theorem~\ref{classif1/2harm}] By Lemma~\ref{conform} we may assume that $v^\e$ is conformal (otherwise we simply consider the complex conjugate of $v$). 
We introduce the so-called {\it Caley transform} $\mathfrak{C}:\overline{\R}^2_+\to \overline{\mathbb{D}}$ defined by 
$$\mathfrak{C}(z):=\frac{z-i}{z+i} \,.$$
It is well known that $\mathfrak{C}$ is a conformal map which is one-to-one from $\R^2_+$ into $\mathbb{D}$, and one-to-one from $\R$ into $\mathbb{S}^1\setminus\{(1,0)\}$. Setting   
$$w:=v^\e\circ \mathfrak{C}^{-1}\,,$$
we find that $w$ is conformal in $\mathbb{D}$, and thus holomorphic in $\mathbb{D}$. 
Moreover, by conformality of $\mathfrak{C}$, we have $w\in H^1(\mathbb{D};\mathbb{C})$, and more precisely
\begin{equation}\label{energconf}
\int_{\mathbb{D}}|\nabla w|^2\,\de z = \int_{\mathbb{R}^2_+}|\nabla v^\e|^2\,\de z\,.
\end{equation}
In particular $g:=w_{|\mathbb{S}^1}\in H^{1/2}(\mathbb{S}^1;\mathbb{C})$. But $g$ is smooth on $\mathbb{S}^1\setminus\{(0,1)\}$ and $g(z)=v\circ\mathfrak{C}^{-1}(z) $ for every $z\in\mathbb{S}^1\setminus\{(0,1)\}$, so that $g\in H^{1/2}(\mathbb{S}^1;\mathbb{S}^1)$. Now, $w$ being holomorphic, it is also harmonic in $\mathbb{D}$. It is therefore the harmonic extension to the disc of the map $g$. By a result of {\sc Brezis \& Nirenberg} \cite{BN1,BN2}, it implies that $|w(z)|\to 1$ uniformly as $|z|\to 1$.  For a nonconstant holomorphic map this latter property implies that it is a (finite) Blaschke product, see \cite{MirPis} and the references therein.   In other words, we can find a positive $d\in\N$, $\tilde \theta\in\R$, and $\alpha_1,\ldots,\alpha_d\in\mathbb{D}$ such that 
\begin{equation}\label{blaschke}
w(z)=e^{i\tilde \theta}\prod_{k=1}^d\frac{z-\alpha_k}{1-\overline{\alpha_k}z} \,.
\end{equation}
Then \eqref{formfracharm} follows from the previous formula with 
$$\lambda_k:=\frac{|1-\alpha_k|^2}{1-|\alpha_k|^2}\,,\quad a_k:= -\frac{2{\rm Im}(\alpha_k)}{|1-\alpha_k|^2}\,,\quad \theta:=\tilde\theta+\sum^d_{k=1} \theta_k\text{ where $\theta_k\in\R$ and } e^{i\theta_k}=
\frac{1-\alpha_k}{1-\overline{\alpha_k}}\,.$$
On the other hand, it is easy to check from \eqref{blaschke} (see {\it e.g.} \cite{MirPis}) that 
$$\int_{\mathbb{D}}|\nabla w|^2\,\de z =2\pi d\,, $$
which combined with  \eqref{energconf} implies \eqref{quantiz} by Lemma \ref{normexth1/2}. 
\end{proof}

\begin{remark}[{\bf 1/2-harmonic circles}]\label{harmcirc}
By analogy with our definition of $\mathcal{E}$ in \eqref{defenergE}, we can consider on $H^{1/2}(\mathbb{S}^1;\C)$ the energy 
$$\mathcal{E}(g,\mathbb{S}^1):=\frac{\gamma_1}{4}\iint_{\mathbb{S}^1\times\mathbb{S}^1}\frac{|g(x)-g(y)|^2}{|x-y|^2}\,\de x\de y\,.$$
It is well known that for every $g\in H^{1/2}(\mathbb{S}^1;\C)$, 
$$\mathcal{E}(g,\mathbb{S}^1)=\frac{1}{2}\int_{\mathbb{D}}|\nabla w_g|^2\,\de x \,,$$
where $w_g\in H^1(\mathbb{D};\C)$ denotes the harmonic extension of $g$ to the disc. In view of identity \eqref{energconf} and Lemma \ref{normexth1/2}, we have 
$$\mathcal{E}(g,\mathbb{S}^1)= \mathcal{E}(g\circ\mathfrak{C},\R) \quad\text{for all  $g\in H^{1/2}(\mathbb{S}^1;\C)$}\,,$$
where $\mathfrak{C}$ is the Caley transform. Defining (weak) 1/2-harmonic maps from $\mathbb{S}^1$ into $\mathbb{S}^1$ as critical points of $\mathcal{E}(\cdot ,\mathbb{S}^1)$ with respect to perturbations in the target (as in Definition \ref{defhalfharm}, see \eqref{ptcritS1} below), we  deduce that 
$g$ is such a 1/2-harmonic map if and only if $g\circ\mathfrak{C}$ is a (finite energy) 1/2-harmonic map from $\R$ into $\mathbb{S}^1$. In terms of the harmonic extension $w_g$, the Euler-Lagrange equation for 1/2-harmonic maps from $\mathbb{S}^1$ into $\mathbb{S}^1$ writes 
\begin{equation}\label{eqharmcirc}
\frac{\partial w_g}{\partial\nu}\wedge g= 0 \quad \text{on $\mathbb{S}^1$}\,,
\end{equation}
where $\frac{\partial}{\partial\nu}$ is the exterior normal derivative on $\partial\mathbb{D}\simeq\mathbb{S}^1$. As a consequence of Theorem~\ref{classif1/2harm}, 
we find that  $g$ is a 1/2-harmonic map from $\mathbb{S}^1$ into $\mathbb{S}^1$ if and only if $g$ is the restriction to $\mathbb{S}^1$ of a (finite) Blaschke product ({\it i.e.}, $g$ is of the form \eqref{blaschke}). 
This fact fact has also been discovered independently by {\sc Berlyand, Mironescu, Rybalko, \&  Sandier} \cite{BMRS} (with essentially the same proof), and by {\sc Da Lio \& Rivi\`ere}~\cite{DaLRi3}. Note that the explicit formula \eqref{blaschke} shows that the energy is quantized by the topological degree, {\it i.e.}, 
$$\mathcal{E}(g,\mathbb{S}^1)=\pi|{\rm deg}(g)| \,,$$
for any  1/2-harmonic map $g$ from $\mathbb{S}^1$ into $\mathbb{S}^1$. By the result of {\sc Mironescu \& Pisante} \cite{MirPis}, it shows 
in particular that every 1/2-harmonic map from $\mathbb{S}^1$ into $\mathbb{S}^1$ is minimizing in its own homotopy class. Here the homotopy classes are classified by the topological 
degree for  $H^{1/2}(\mathbb{S}^1;\mathbb{S}^1)$-maps as defined in \cite{BN1,BN2}. These properties of {\it 1/2-harmonic circles into $\mathbb{S}^1$} are in clear analogy with the theory of usual 
harmonic maps where it is well known that  {\it harmonic 2-spheres into $\mathbb{S}^2$} are minimizing in their own homotopy class and have an energy quantized by the degree. 
\end{remark}

With the help of Theorem \ref{classif1/2harm} we can now give an explicit representation of 0-homogeneous maps wich are weakly 1/2-harmonic from the unit disc of $\R^2$ into $\mathbb{S}^1$. 
Those maps provide examples of singular weak 1/2-harmonic maps. In particular, it shows that $\frac{x}{|x|}$ is a weak 1/2-harmonic map into $\mathbb{S}^1$.  It would be interesting to determine which maps are minimizing or stationary, and then to compare the result with~\cite{BCL} 
(which deals with 0-homogeneous harmonic maps from the unit ball in $\R^3$ into $\mathbb{S}^2$).

\begin{proposition}\label{homog1/2harm}
For $n=2$, let $v\in \widehat H^{1/2}(D_1;\mathbb{R}^2)$ be a 0-homogeneous map in all of $\R^2$ such that $|v|=1$ a.e. in $\R^2$.  
Then $v$ is a weak 1/2-harmonic map into $\mathbb{S}^1$ in $D_1$ if and only if $v(x)=g(\frac{x}{|x|})$ for a 1/2-harmonic circle $g$ into $\mathbb{S}^1$ in the sense of Remark~\ref{harmcirc}. 
\end{proposition}

\begin{proof}
{\it Step 1.} Let us consider an arbitrary map $v\in \widehat H^{1/2}(D_1;\mathbb{R}^2)$ which is 0-homogeneous. By a rescaling  argument 
we first deduce that $v\in H^{1/2}_{\rm loc}(\R^2;\R^2)$, and in view of Remark~\ref{H1/2loctoH1loc},  
$v^\e\in H^1_{\rm loc}(\overline\R^3_+;\R^2)$. Since $v^\e$ clearly inherits the 0-homogeneity of $v$, we infer that the map $\widetilde v^{\,\e}:={v^\e}_{|\mathbb{S}^2_+}$ belongs to $H^1(\mathbb{S}^2_+;\R^2)$ where $\mathbb{S}^2_+:=\mathbb{S}^2\cap\R^3_+$. Therefore $g:={\widetilde{v}^{\,\e}}_{\;\;|\partial \mathbb{S}^2_+}\in H^{1/2}(\mathbb{S}^1;\R^2)$ identifying $\partial \mathbb{S}^2_+$ with $\mathbb{S}^1$. 
Obviously, we have $v=g(\frac{x}{|x|})$ and $v^\e(x)=\widetilde v^{\,\e}(\frac{x}{|x|})$.  From the harmonicity of $v^\e$, we deduce that $\widetilde v^{\,\e} $ is the unique solution of 
$$\begin{cases}  
\Delta_{\mathbb{S}^2} \widetilde v^{\,\e} =0 & \text{on $\mathbb{S}^2_+$}\,,\\
\widetilde v^{\,\e}=g & \text{on $\partial \mathbb{S}^2_+$}\,,
\end{cases}
$$
where $\Delta_{\mathbb{S}^2}$ denotes the Laplace-Beltrami operator on $\mathbb{S}^2$. 

We now introduce the one-to-one conformal mapping $\mathfrak{S}:\overline{\mathbb{D}}\to \overline{\mathbb{S}^2_+}$ defined by 
$$\mathfrak{S}(x):=\left(\frac{2x_1}{|x|^2+1}, \frac{2x_2}{|x|^2+1},\frac{1-|x|^2}{|x|^2+1}\right) \,.$$
Note that $\mathfrak{S}$ is the identity on $\mathbb{S}^1$. By conformality of  $\mathfrak{S}$ the map $w_g:=\widetilde v^{\,\e}\circ \mathfrak{S}\in H^1(\mathbb{D}:\R^2)$ is harmonic in $\mathbb{D}$, and $w_g=g$ on $\mathbb{S}^1$ in the trace sense. It is therefore the harmonic extension of $g$ to the disc  $\mathbb{D}$, and we infer from Remark~\ref{harmcirc} and the conformality of  $\mathfrak{S}$ that
$$\frac{1}{2}\int_{\mathbb{S}^2_+}|\nabla_T \widetilde v^{\,\e}|^2\,\de \mathcal{H}^2 =\frac{1}{2}\int_{\mathbb{D}}|\nabla w_g|^2\,\de x=\mathcal{E}(g,\mathbb{S}^1)\,,$$
where $\nabla_T$ denotes the tangential gradient. 
\vskip5pt

\noindent{\it Step 2.} Assume that $v\in \widehat H^{1/2}(D_1;\mathbb{R}^2)$ is a 0-homogeneous 1/2-harmonic map into $\mathbb{S}^1$ in~$D_1$. We claim that the associated map $g$ as defined in Step 1 is a 1/2-harmonic circle into $\mathbb{S}^1$, {\it i.e.}, 
\begin{equation}\label{ptcritS1}
 \left[\frac{\de}{\de t} \,\mathcal{E}\left(\frac{g+t\varphi}{|g+t\varphi|}\,,\mathbb{S}^1\right)\right]_{t=0} =0 
 \end{equation}
 for all $\varphi\in H^{1/2}(\mathbb{S}^1;\mathbb{R}^2)\cap L^\infty(\mathbb{S}^1)$. Arguing as in the proof of Proposition~\ref{caracthalfharm}, we infer that it is enough to prove \eqref{ptcritS1}  for all $\varphi\in H^{1/2}(\mathbb{S}^1;\mathbb{R}^2)\cap L^\infty(\mathbb{S}^1)$ such that 
 $\varphi\cdot g =0$ a.e. on~$\mathbb{S}^1$. For such a function $\varphi$, we have 
 $$ \left[\frac{\de}{\de t} \,\mathcal{E}\left(\frac{g+t\varphi}{|g+t\varphi|}\,,\mathbb{S}^1\right)\right]_{t=0} = \left[\frac{\de}{\de t} \,\mathcal{E}\left(g+t\varphi\,,\mathbb{S}^1\right)\right]_{t=0}= \left[\frac{\de}{\de t} \left(\frac{1}{2}\int_{\mathbb{D}}|\nabla(w_g+tw_\varphi)|^2\,\de x\right)\right]_{t=0}\,,  $$
where $w_\varphi$ is the harmonic extension of $\varphi$ to the disc $\mathbb{D}$ (notice that $w_\varphi\in H^1(\mathbb{D};\R^2)\cap L^\infty(\mathbb{D})$). 
Setting $\Phi:=w_\varphi\circ \mathfrak{S}^{-1}\in H^1(\mathbb{S}^2_+;\R^2)\cap L^\infty(\mathbb{S}^2_+)$, we can argue as in Step 1 to obtain  
$$\frac{1}{2}\int_{\mathbb{D}}|\nabla(w_g+tw_\varphi)|^2\,dx=  \frac{1}{2}\int_{\mathbb{S}^2_+}|\nabla_T (\widetilde v^{\,\e}+t\Phi)|^2\,\de \mathscr{H}^2\,.$$
Consequently, 
\begin{equation}\label{2104}
 \left[\frac{\de}{\de t} \,\mathcal{E}\left(\frac{g+t\varphi}{|g+t\varphi|}\,,\mathbb{S}^1\right)\right]_{t=0} =
 \left[\frac{\de}{\de t} \left( \frac{1}{2}\int_{\mathbb{S}^2_+}|\nabla_T (\widetilde v^{\,\e}+t\Phi)|^2\,\de \mathscr{H}^2\right)\right]_{t=0} 
 =\int_{\mathbb{S}^2_+}\nabla_T \widetilde v^{\,\e} \cdot \nabla_T \Phi\,\de \mathscr{H}^2\,.
 \end{equation}
 Let us now consider a function $\eta\in C^1((0,+\infty))$ with compact support in the interval $(0,1)$ satisfying $\int_0^1\eta(r)r^2\,\de r=1$. Define for $x\in\R^3_+$, 
 $\widetilde \Phi:=\eta(|x|)\Phi(\frac{x}{|x|})$. Then $\widetilde \Phi\in H^1(B_1^+;\R^2)\cap L^\infty(B_1^+)$, $\widetilde \Phi$ has compact support in $B_1^+\cup D_1$, and 
 $\widetilde \Phi\cdot v^\e=0$ $\mathscr{H}^2$-a.e. on $D_1$. Since $v$ is assumed to be 1/2-harmonic, we deduce from  Proposition \ref{fracweakharm} that 
\begin{equation}\label{2102}
\int_{B^+_1} \nabla v^\e\cdot \nabla \widetilde \Phi\,\de x =0\,.
\end{equation}
 On the other hand, by homogeneity of $v^\e$, we have 
 \begin{equation}\label{2103}
 \int_{B^+_1} \nabla v^\e\cdot \nabla \widetilde \Phi\,\de x=\left(\int_0^1\eta(r)r^2\,\de r\right) \int_{\mathbb{S}^2_+} \nabla_T \widetilde v^{\,\e}\cdot \nabla_T \Phi\,\de \mathscr{H}^2
 =\int_{\mathbb{S}^2_+}\nabla_T \widetilde v^{\,\e} \cdot \nabla_T \Phi\,\de \mathscr{H}^2\,. 
 \end{equation}
Gathering \eqref{2104}-\eqref{2102}-\eqref{2103} leads to \eqref{ptcritS1}, and the claim is proved. 
\vskip5pt

\noindent{\it Step 3.} Consider  a 1/2-harmonic circle $g:\mathbb{S}^1\to\mathbb{S}^1$, and set $v:=g(\frac{x}{|x|})$. From the smoothness of $g$, we infer that $v\in H^{1/2}_{\rm loc}(\R^2;\R^2)\cap L^\infty(\R^2)\subset \widehat H^{1/2}(D_1;\mathbb{R}^2)$, and $v$ is smooth away from the origin. Moreover, we deduce from Step 1 that 
$$v^\e(x)=w_g\circ\mathfrak{S}^{-1}\left(\frac{x}{|x|}\right) \,.$$
As a consequence,
$$\frac{\partial v^\e}{\partial \nu}(x)= \frac{1}{|x|} \frac{\partial w_g}{\partial\nu_{p}}(p) \quad\text{on $\R^2\setminus\{0\}$}\,,$$
where $p:=\mathfrak{S}^{-1}\left(\frac{x}{|x|}\right)\in \mathbb{S}^1$ and  $\nu_p$ denotes the exterior normal to $\partial \mathbb{D}\simeq \mathbb{S}^1$ at the point $p$. In view of \eqref{eqharmcirc} we have
$$ \frac{\partial v^\e}{\partial \nu}(x)\wedge v(x)= \frac{1}{|x|} \frac{\partial w_g}{\partial\nu_{p}}(p)\wedge g(p)=0 \quad\text{on $\R^2\setminus\{0\}$}\,.$$
From Lemma \ref{repnormderfraclap} we then infer that $ (-\Delta)^{\frac{1}{2}} v\wedge v=0$ in $\R^2\setminus\{0\}$, and thus $v$ satisfies 
\eqref{varformwkfracharm} in $D_1\setminus \overline D_\rho$ for every $0<\rho<1$. Since $v$ belongs to $\widehat H^{1/2}(D_1;\R^2)$, a standard capacity argument shows that \eqref{varformwkfracharm} actually holds in $D_1$. Hence $v$ is a weak 1/2-harmonic map in $D_1$ by Proposition~\ref{caracthalfharm}.
\end{proof}

\vskip10pt

\section{Small energy estimate for Ginzburg-Landau boundary reactions}\label{epsregtitle}                        

In this section we perform a preliminary analysis on the asymptotic, as $\eps\to 0$, of critical points of the Ginzburg-Landau boundary energy $E_\varepsilon$ defined in \eqref{defEnergGLB}. 
The first step we make here is to prove an {\it epsilon}-{\it regularity} type of estimate which allows to control the 
regularity of solutions under the assumption that the energy, suitably renormalized, is small. This is the purpose of the following theorem which is a corner stone of the present paper. The entire section is devoted to its proof.

\begin{theorem}\label{epsreg}
Let $\varepsilon>0$ and $R>0$  such that $\varepsilon\leq R$. There exist constants $\eta_0>0$ and $C_0>0$  independent of $\eps$ and $R$ such that for each map 
$u_\eps \in C^{2}(B_R^+\cup D_R;\mathbb{R}^m)$ satisfying $|u_\eps|\leq 1$ and solving 
\begin{equation}\label{tim1439}
\begin{cases}
\Delta u_\eps= 0 & \text{in $B_R^+$}\,,\\[5pt]
\displaystyle \frac{\partial u_\eps}{\partial \nu}=\frac{1}{\varepsilon}(1-|u_\eps|^2)u_\eps & \text{on $D_R$}\,,
\end{cases}
\end{equation}
the condition $E_\varepsilon(u_\eps,B_R^+)\leq \eta_0 R^{n-1}$ implies
\begin{equation}\label{tim2206}
\sup_{ B^+_{R/4}} |\nabla u_\eps|^2 +\sup_{ D_{R/4}} \frac{(1-|u_\eps|^2)^2}{\varepsilon^2}\leq \frac{C_0}{R^2}\,\eta_0\,.
\end{equation}
\end{theorem}

The proof of Theorem~\ref{epsreg} is divided into three main parts according to the following subsections. For the local Ginzburg-Landau equation, the argument leading to the analoguous estimate is essentially based on the so-called {\it Bochner inequality} satified by the energy density, see {\it e.g.}~\cite{CS}. In our case, such inequality does not seem to be avaible, and to prove \eqref{tim2206} we better use a compactness approach in the spirit of \cite{W}.


 \subsection{Energy monotonicity and clearing-out lemma}
 
As usual in Ginzburg-Landau or harmonic map problems, the first main ingredient to derive regularity estimates is a useful monotonicity formula.  This is the purpose of the following lemma. 

 \begin{lemma}[\bf Monotonicity formula]\label{monotform}
Let $R>0$ and  $u_\eps \in  C^{2}(B_R^+\cup D_R;\mathbb{R}^m)$ 
satisfying~\eqref{tim1439}. 
Then, for every $x_0\in D_R$ and every $0<\rho<r<{\rm dist}(x_0,\partial D_R)$,
\begin{multline*}
\frac{1}{r^{n-1}}E_\varepsilon\big(u_\eps,B^+_r(x_0)\big)-\frac{1}{\rho^{n-1}}E_\varepsilon\big(u_\eps,B^+_\rho(x_0)\big) \\
=\int_\rho^r \frac{1}{t^{n-1}} \left(\int_{\partial^+ B_t(x_0)}\left|\frac{\partial u_\eps}{\partial \nu}\right|^2\,\de \mathscr{H}^n\right)\,\de t\\
+\frac{1}{4\varepsilon}\int_\rho^r\frac{1}{t^n}\left(\int_{D_t(x_0)}(1-|u_\eps|^2)^2\,\de\mathscr{H}^n\right)\,\de t\,.
\end{multline*}
\end{lemma}

\begin{proof}
{\it Step 1.} Consider a smooth vector field $\mathbf{X}=(\mathbf{X}_1,\ldots,\mathbf{X}_{n+1}):\R^{n+1}\to\R^{n+1}$ compactly supported in $B_R$ and satisfying $\mathbf{X}_{n+1}=0$ on $\R^n$.   
For $t\in\mathbb{R}$ small, the map $\Phi_t(x):=x-t \mathbf{X}(x)$ defines a smooth diffeomorphism from $B_R$ into $B_R$ satisfying $\Phi_t(D_R)\subset D_R$, and $\Phi_t(B^+_R)\subset B^+_R$. Setting $u_t:=u_\eps\circ\Phi_t$, standard computations (see {\it e.g.} \cite[Chapter 2.2]{Sim}) yield 
\begin{equation}\label{tim1516}
\left[\frac{\de}{\de t} \left(\frac{1}{2}\int_{B_R^+} \big|\nabla u_t \big|^2\,\de x \right)\right]_{t=0} 
=\frac{1}{2}\int_{B_R^+}\left( |\nabla u_\eps|^2{\rm div}\mathbf{X} -2\sum_{i,j=1}^{n+1} (\partial_i u_\eps\cdot\partial_j u_\eps)\partial_j\mathbf{X}_i\right)\,\de x\,,
\end{equation}
and
$$\left[\frac{\de}{\de t} \left(\frac{1}{4\varepsilon}\int_{D_R} (1-|u_t|^2)^2\,\de\mathscr{H}^n  \right)\right]_{t=0} 
= \frac{1}{4\varepsilon}\int_{D_R} (1-|u_\eps|^2)^2{\rm div}_{\R^n}\mathbf{X}\,\de \mathscr{H}^n \,.$$
Using \eqref{tim1439} we integrate by parts in \eqref{tim1516} to find  
$$\left[\frac{\de}{\de t} \left(\frac{1}{2}\int_{B_R^+} \big|\nabla u_t \big|^2\,\de x \right)\right]_{t=0} 
= - \frac{1}{4\varepsilon}\int_{D_R} (1-|u_\eps|^2)^2{\rm div}_{\R^n}\mathbf{X}\,\de \mathscr{H}^n \,,$$
whence 
\begin{equation}\label{statGL}
\left[\frac{\de}{\de t}\, E_\varepsilon(u_t,B_R^+)\right]_{t=0}=0\,. 
\end{equation}
\vskip3pt

\noindent{\it Step 2.} Let $x_0\in D_R$, and $0<r<{\rm  dist}(x_0,\partial D_R)$. Without loss of generality we may assume that $x_0=0$ (to simplify the notation). Let $\eta\in C^\infty(\R;[0,1])$ be an even function such that $\eta(t)=0$ for $|t|\geq r$. Using the vector field $\mathbf{X}(x)=\eta(|x|)x$ in Step 1, we find 
\begin{multline}\label{tim1843}
\frac{n-1}{2}\int_{B_R^+}|\nabla u_\eps|^2\eta(|x|)\,\de x+ \frac{1}{2}\int_{B_R^+}|\nabla u_\eps|^2\eta^\prime(|x|)|x|\,\de x \\
-\int_{B_R^+}\left|\frac{\partial u_\eps}{\partial |x|}\right|^2 \eta^\prime(|x|)|x|\,\de x 
+\frac{n}{4\varepsilon}\int_{D_R}(1-|u_\eps|^2)^2\eta(|x|)\,\de \mathscr{H}^n\\
+\frac{1}{4\varepsilon}\int_{D_R}(1-|u_\eps|^2)^2\eta^\prime(|x|)|x|\,\de \mathscr{H}^n =0\,.
\end{multline}
Then, given an arbitrary $t\in(0,r)$, we consider a sequence $\{\eta_k\}$ of functions as above such that $\eta_k$ converges weakly* in $BV$ as $k\to\infty$ to the characteristic function of the interval $[-t,t]$. Using $\eta_k$ as a test function in \eqref{tim1843} and letting $k\to\infty$ leads to 
\begin{multline*}
\frac{n-1}{2}\int_{B_{t}^+}|\nabla u_\eps|^2\,\de x -  \frac{t}{2}\int_{\partial^+B_t}|\nabla u_\eps|^2\,\de \mathscr{H}^n
+t \int_{\partial^+ B_t}\left|\frac{\partial u_\eps}{\partial \nu}\right|^2\,\de \mathscr{H}^n\\
+\frac{n}{4\varepsilon}\int_{D_t}(1-|u_\eps|^2)^2\,\de \mathscr{H}^n-\frac{t}{4\varepsilon}\int_{\partial D_t}(1-|u_\eps|^2)^2\,\de \mathscr{H}^{n-1} =0\,.
\end{multline*}
Dividing by $t^n$ and rearranging terms, we end up with 
$$\frac{\de}{\de t} \left(\frac{1}{t^{n-1}}E_\varepsilon(u_\eps,B_t^+) \right)=  \frac{1}{t^{n-1}} \int_{\partial^+ B_t}\left|\frac{\partial u_\eps}{\partial \nu}\right|^2\,\de \mathscr{H}^n
+\frac{1}{4\varepsilon t^n}\int_{D_t}(1-|u_\eps|^2)^2\,\de \mathscr{H}^n\,.$$
Integrating this equality between $\rho>0$ and $r$ yields to the announced result.
\end{proof}

\begin{remark}[\bf Liouville property]
As a consequence of Lemma~\ref{monotform} and the stationarity equation \eqref{statGL}, any entire (smooth) solution $u_\varepsilon:\overline{\R}^{n+1}_+\to \R^m$ of the Ginzburg-Landau boundary equation ({\it i.e.}, $u_\varepsilon$ solves \eqref{tim1439} for every radius $R>0$) satisfying $E_\varepsilon(u_\varepsilon,\R^{n+1}_+)<\infty$, has to be  {\it constant}.  In view of \eqref{eqvepsext} and Lemma~\ref{normexth1/2}, the same Liouville property holds for entire solutions of the fractional Ginzburg-Landau equation having finite energy. More precisely,  if $v_\varepsilon:\R^n\to\R^m$ is a (smooth) solution of \eqref{eqfractGL} in $\R^n$ satisfying $\mathcal{E}_\varepsilon(v_\varepsilon,\R^n)<\infty$, then $v_\varepsilon$ is constant. The argument goes as follows.
\begin{proof}
If $n\geq 2$, Lemma~\ref{monotform} implies $E_\eps(u_\eps,B_\rho^+)\leq C(r/\rho)^{1-n}$ for every $0<\rho\leq r$, and the constancy of $u_\eps$ follows letting $r\to\infty$. The case $n=1$ is slightly more involved, and relies on \eqref{statGL}. For $n=1$ we consider a nonnegative cut-off function 
$\zeta\in C^\infty(\R^{2})$ such that $\zeta=1$ on $B_1$ and $\zeta=0$ on $\R^{2}\setminus B_2$. Setting $\zeta_k(x):=\zeta(x/k)$, we use the vector field ${\bf X}(x)=\zeta_k(x)x$ in \eqref{statGL} to obtain
\begin{multline*}
\frac{1}{2}\int_{B_{2k}^+\setminus B_k^+}\left( |\nabla u_\eps|^2(x\cdot\nabla\zeta_k) -2\sum_{i,j=1}^{2}(\partial_iu_\eps\cdot\partial_ju_\eps)x_i\partial_j\zeta_k\right)\,\de x \\
+\frac{1}{4\eps}\int_{D_{2k}}(1-|u_\eps|^2)^2\zeta_k\,\de\mathscr{H}^1 +\frac{1}{4\eps}\int_{D_{2k}\setminus D_k}(1-|u_\eps|^2)^2x_1\partial_1\zeta_k\,\de\mathscr{H}^1=0\,.
\end{multline*}
Since $|\nabla\zeta_k|\leq C/k$, this identity yields
$$\frac{1}{\eps}\int_{D_{k}}(1-|u_\eps|^2)^2\,\de\mathscr{H}^1 \leq CE_\eps(u_\eps, B^+_{2k}\setminus B_k^+)\mathop{\longrightarrow}\limits_{k\to\infty} 0\,,$$
and we deduce that $|u_\eps|\equiv 1$. We then infer that $u_\eps$ is a bounded solution of 
$$\begin{cases}
\Delta u_\eps =0 & \text{in $\R^{n+1}_+$}\,,\\[5pt]
\displaystyle \frac{\partial u_\eps}{\partial\nu} = 0 & \text{on $\R^n$}\,,
\end{cases}
$$
and the constancy of $u_\eps$ follows from Liouville Theorem on bounded entire harmonic functions through the usual reflection argument across $\partial \R^{n+1}_+\simeq\R^n$. 
\end{proof}
\end{remark}

The second key ingredient in Ginzburg-Landau problems is the so-called {\it clearing-out property} which essentially asserts that a solution must take  values  in a small neighborhood of the potential-well whenever the energy is small enough. This is precisely the object of the 
following lemma. The proof uses some ingredients from \cite[Lemma 3.1]{Sch} suitably modified to fit the Ginzburg-Landau setting. 
 
\begin{lemma}[\bf Clearing-out]\label{clear2}
Let  $\varepsilon\in (0,1]$. There exists a constant $\eta_1>0$ independent of~$\eps$ such that for each map 
$u_\eps\in  C^{2}(B_1^+\cup D_1;\mathbb{R}^m)$ satisfying $|u_\eps|\leq 1$ and \eqref{tim1439} with $R=1$, 
the condition $E_\varepsilon(u_\eps,B^+_1)\leq \eta_1$ 
implies 
\begin{equation}\label{cleareps}
 |u_\eps|\geq \frac{1}{2} \quad \text{in }  \overline{B}_{1/2}^+\,.
\end{equation}
\end{lemma}

\begin{proof}
{\it Step 1.} We assume in this first step that $\varepsilon\geq 1/2$. We claim that we can find $\eta_2>0$ only depending on $n$ and $m$  such that if 
$E_\varepsilon(u_\eps,B^+_1)\leq \eta_2$, then $ |u_\eps|\geq 1/2$ in $\overline B^+_{1/2}$. Indeed, for an arbitrary map $u_\eps$ satisfying the statement of the theorem, since $|u_\eps|\leq 1$ and $\varepsilon\geq 1/2$, we deduce from  \cite[Lemma~2.2]{CSM} 
that $\| u_\eps\|_{ C^{2,\beta}(B^+_{1/2})}\leq C_\beta$ for a constant $C_\beta$ only depending on $n$, $m$, and $\beta\in(0,1)$.  We now proceed by contradiction assuming that there exist sequences $\{\varepsilon_k\}\subset[1/2,1]$, $\{x_k\}\subset \overline B^+_{1/2}$, and $\{u_k\}\subset  C^{2}(B_1^+\cup D_1;\mathbb{R}^m)$ satisfying $|u_k|\leq 1$ and \eqref{tim1439} with $R=1$,  
such that $E_{\varepsilon_k}(u_k,B_1^+)\to 0$ and $|u_k(x_k)|<1/2$. Then we can find a (not relabeled) subsequence  such that $u_k$ converges uniformly on $\overline B^+_{1/2}$. 
Since  $E_{\varepsilon_k}(u_k,B_1^+)\to 0$ the limit has to be a constant of modulus one. As a consequence, $|u_k|\to 1$  uniformly on $\overline B^+_{1/2}$, which contradicts the assumption  $|u_k(x_k)|<1/2$. 
\vskip3pt

\noindent{\it Step 2.} We now consider the case $\varepsilon<1/2$. 
Let us fix an arbitrary point $x_0=(x'_0,(x_0)_{n+1})$ in $B_{1/2}^+$, and set $x_1:=(x'_0,0)\in D_{1/2}$. By the Monotonicity Formula in Lemma~\ref{monotform} we have 
\begin{equation}\label{monotloc}
\frac{1}{r^{n-1}}E_\varepsilon (u_\eps,B^+_{r}(x_1)) \leq (1-|x_1|)^{1-n} E_\varepsilon (u_\eps,B^+_{1-|x_1|}(x_1)) 
\leq2^{n-1} E_\varepsilon (u_\eps,B^+_{1}) \leq 2^{n-1}\eta_1
\end{equation}
for every $0<r<1-|x_1|$.  We  set 
$$R:=\frac{1}{3}|x_1-x_0|\,.$$
\vskip5pt

\noindent{\it Case 1)} Let us first  assume that $R\leq \varepsilon/7$. Since $\varepsilon<1/2<1-|x_1|$ we have
$$\frac{1}{\varepsilon^{N-1}}E_\varepsilon(u_\eps,B^+_\varepsilon(x_1))\leq 2^{n-1}\eta_1\,.$$
Next we define for $y\in B_1^+\cup D_1$, 
$$v_\varepsilon(x):=u_\varepsilon(x_1+\varepsilon y) \,,$$
so that 
$$\begin{cases}
\Delta v_\varepsilon = 0 & \text{in $B_1^+$}\,,\\[8pt]
\displaystyle \frac{\partial v_\varepsilon}{\partial \nu}=(1-|v_\varepsilon|^2)v_\varepsilon & \text{on $D_1$}\,,
\end{cases}$$
and 
$$E_1(v_\varepsilon,B^+_1)=\frac{1}{\varepsilon^{n-1}}E_\varepsilon(w,B^+_\varepsilon(x_1))\leq 2^{n-1}\eta_1\,.$$
Choosing $\eta_1$ such that $2^{n-1}\eta_1\leq \eta_2$, we infer from Step 1 that 
$|v_\varepsilon|\geq 1/2$ in $B^+_{1/2}$. Hence $|u_\eps|\geq 1/2$ in  $B^+_{\varepsilon/2}(x_1)$, and thus $|u_\eps(x_0)|\geq 1/2$. 
\vskip5pt

\noindent{\it Case 2)} We assume that $R>\varepsilon/7$, so that $\varepsilon/R<7$. Noticing that $5R<1-|x_1|$, we have $B_{5R}^+(x_1)\subset B_1^+$. Let us set 
$$\bar u_\eps:= \frac{1}{|B_{5R}^+(x_1)|}\int_{B_{5R}^+(x_1)} u_\eps(y)\,\de y\,. $$
We consider for $x\in B_{2R}(x_0)\subset B^+_{5R}(x_1)$, 
$$v_\eps(x):=u_\eps(x)-\bar u_\eps\,.$$

Denote by $G_{x_0}$ the fundamental solution in $\mathbb{R}^{n+1}$ of the Laplace equation with pole at~$x_0$. We recall that 
\begin{equation}\label{estiGreen}
|\nabla G_{x_0}(x)|\leq \frac{C}{|x-x_0|^n} \qquad \text{for all $x\in\mathbb{R}^{n+1}\setminus\{x_0\}$}\,.
\end{equation}
Then consider a smooth cut-off function $\zeta\in  C^{\infty}_c(B_{2R}(x_0);[0,1])$ such that $\zeta\equiv 1$ in $B_R(x_0)$,  $\zeta\equiv 0$ in $B_{2R}(x_0)\setminus B_{3R/2}(x_0)$, and 
$|\nabla \zeta|\leq C/R$. 

We estimate
\begin{align}
\nonumber |v_\eps(x_0)|^2 &= \int_{B_{2R}(x_0)} G_{x_0}\Delta(|v_\eps|^2\zeta^2)\,\de x\\
\nonumber& =-\int_{B_{2R}(x_0)} \nabla G_{x_0} \cdot \nabla(|v_\eps|^2\zeta^2)\,\de x\\
\nonumber &=-2 \int_{B_{3R/2}(x_0)} \zeta^2\big(\sum_{j=1}^n \partial_j G_{x_0}\partial_j v_\eps\cdot v_\eps\big)\,\de x -2\int_{A_R(x_0)} |v_\eps|^2 \zeta\,  \nabla G_{x_0}\cdot \nabla \zeta\,\de x\\
\label{estivsqx0} &\leq  \|v_\eps\|_{L^\infty(B_{3R/2}(x_0))} \int_{B_{3R/2}(x_0)} |\nabla G_{x_0}| |\nabla u_\eps|\,\de x + \frac{C}{R^{n+1}}\int_{A_R(x_0)} |v_\eps|^2\,\de x \,,
\end{align}
where $A_R(x_0):=B_{2R}(x_0)\setminus B_R(x_0)$.

By harmonicity of $u_\eps$, $\nabla u_\eps$ is harmonic in $B_1^+$. Since $B_{R/4}(x)\subset B^+_{5R}(x_1)$ for all $x\in B_{3R/2}(x_0)$, we deduce that 
$$\nabla u_\eps(x)=\frac{1}{|B_{R/4}(x)|}\int_{B_{R/4}(x)}\nabla u_\eps(y)\,\de y \qquad\text{for all $x\in B_{3R/2}(x_0)$}\,.$$
From Jensen's inequality and \eqref{monotloc}, we then infer that 
$$|\nabla u_\eps(x)|^2\leq \frac{1}{|B_{R/4}(x)|}\int_{B_{R/4}(x)}|\nabla u_\eps(y)|^2\,\de y \leq \frac{C}{R^{n+1}}E_\varepsilon(u_\eps,B^+_{5R}(x_1))\leq \frac{C}{R^2} \eta_1$$
for all $x\in B_{3R/2}(x_0)$. Consequently,
\begin{equation}\label{ptgradest}
|\nabla u_\eps(x)|\leq \frac{C}{R}\sqrt{\eta_1} \qquad\text{for all $x\in B_{3R/2}(x_0)$}\,.
\end{equation}
Using \eqref{estiGreen} together with \eqref{ptgradest}, we obtain that 
\begin{equation}\label{gradGgradu}
\int_{B_{3R/2}(x_0)} |\nabla G_{x_0}| |\nabla u_\eps|\,\de x\leq C\sqrt{\eta_1}\,.
\end{equation}
Estimate  \eqref{ptgradest} also yields
\begin{equation}\label{estivlinfty}
\|v_\eps\|_{L^\infty(B_{3R/2}(x_0))}\leq |v_\eps(x_0)| +\frac{3R}{2}\|\nabla u_\eps\|_{L^\infty(B_{3R/2}(x_0))}\leq   |v_\eps(x_0)| +C\sqrt{\eta_1}\,.
\end{equation}
Next we infer from Poincar\'e Inequality that 
$$\int_{B_{2R}(x_0)\setminus B_R(x_0)} |v_\eps|^2\,\de x\leq \int_{B^+_{5R}(x_1)}|u_\eps-\bar u_\eps|^2\,\de x\leq C R^2  \int_{B^+_{5R}(x_1)}|\nabla u_\eps|^2\,\de x\,,$$
so that   
\begin{equation}\label{estimeanv}
\frac{1}{R^{n+1}}\int_{B_{2R}(x_0)\setminus B_R(x_0)} |v_\eps|^2\,\de x\leq \frac{C}{R^{n-1}}E_\varepsilon(u_\eps,B^+_{5R}(x_1))\leq C \eta_1\,.
\end{equation}
Gathering \eqref{estivsqx0}, \eqref{gradGgradu}, \eqref{estivlinfty}, and \eqref{estimeanv}, we conclude that 
$$|v_\eps(x_0)|^2\leq C \sqrt{\eta_1}\big(|v_\eps(x_0)|+\sqrt{\eta_1}\big)\leq \frac{1}{2}|v_\eps(x_0)|^2+C\eta_1\,,$$
and thus 
$$|u_\eps(x_0)-\bar u_\eps|\leq C\sqrt{\eta_1}\,. $$
In particular,
\begin{equation}\label{estiux0}
\big|1-|u_\eps(x_0)|\big|\leq \big|1-|\bar u_\eps|\big|+C \sqrt{\eta_1}\,.
\end{equation}

Now we claim that
\begin{equation}\label{claimclear}
\big|1-|\bar u_\eps|\big|\leq C\sqrt{\eta_1}\,. 
\end{equation}
Setting for $y\in\mathbb{R}^m$, $d(y):=|1-|y||$, the function $d$ is 1-Lipschitz. Since $d(\bar u_\eps)\leq d(u_\eps(x))+|u_\eps(x)-\bar u_\eps|$, we have 
\begin{equation}\label{preestd}
d(\bar u_\eps)\leq \frac{1}{|B^+_{5R}(x_1)|}\int_{B^+_{5R}(x_1)} d(u_\eps(x))\,\de x+\frac{1}{|B^+_{5R}(x_1)|}\int_{B^+_{5R}(x_1)} |u_\eps(x)-\bar u_\eps|\,\de x\,.
\end{equation}
Then we use Jensen and Poincar\'e inequalities together with \eqref{monotloc}, to derive
\begin{multline}\label{estidpoinc}
 \left(\frac{1}{|B^+_{5R}(x_1)|}\int_{B^+_{5R}(x_1)} |u_\varepsilon(x)-\bar u_\varepsilon|\,\de x\right)^2  \leq \frac{1}{|B^+_{5R}(x_1)|}\int_{B^+_{5R}(x_1)} |u_\varepsilon(x)-\bar u_\varepsilon|^2\,\de x\\
\leq \frac{C}{R^{n-1}}\int_{B^+_{5R}(x_1)}|\nabla u_\varepsilon|^2\,\de x \leq C\eta_1\,.
\end{multline}
Using again Jensen  and Poincar\'e inequalities for the function $d\circ u_\eps$, estimate \eqref{monotloc}, and the facts that $|\nabla d|\leq 1$ and $\varepsilon/R<7$, we obtain
\begin{align}
\nonumber \bigg(\frac{1}{|B^+_{5R}(x_1)|} &\int_{B^+_{5R}(x_1)} d\circ u_\eps \,\de x\bigg)^2\leq \frac{1}{|B^+_{5R}(x_1)|}\int_{B^+_{5R}(x_1)} |d\circ u_\eps|^2 \,\de x\\
\nonumber &\leq C\left(\frac{1}{R^n}\int_{D_{5R}(x_1)}|d\circ u_\eps|^2\,\de \mathscr{H}^n+\frac{1}{R^{n-1}}\int_{B^+_{5R}(x_1)} |\nabla(d\circ u_\eps)|^2\,\de x\right)\\
\nonumber &\leq \frac{C}{R^{n-1}}\left(\frac{1}{\varepsilon}\int_{D_{5R}(x_1)}(1-|u_\eps|^2)^2\,\de \mathscr{H}^n+\int_{B^+_{5R}(x_1)} |\nabla u_\eps|^2\,\de x\right) \\
\label{estidpoinc2}&\leq C \eta_1\,.
\end{align}
Combining \eqref{preestd}, \eqref{estidpoinc}, and \eqref{estidpoinc2}, yields \eqref{claimclear}

Now \eqref{estiux0} and \eqref{claimclear} imply that $\big|1-|u_\eps(x_0)|\big|\leq C\sqrt{\eta_1}$, and thus $|u_\eps(x_0)|\geq 1/2$ whenever $\eta_1$ is chosen small enough, smallness depending only on the dimensions $n$ and $m$. 
\end{proof}


\subsection{Proof of Theorem~\ref{epsreg}, Step 1.} By rescaling variables, it suffices to consider the case $R=1$ and $\varepsilon\in (0,1]$. 
We shall choose $\eta_0\leq \eta_1$ where $\eta_1$ is given by  Lemma~\ref{clear2}, so that 
\begin{equation}\label{clearmodu}
|u_\varepsilon|\geq \frac{1}{2}\qquad\text{in } \overline{B}^+_{1/2}\,.
\end{equation}
Arguing as in the proof of Lemma \ref{clear2}, we infer that 
\begin{equation}\label{contren}
\frac{1}{r^{n-1}}E_\varepsilon(u_\varepsilon,B^+_r(x))\leq 2^{n-1}\eta_0\qquad\text{for all $x\in D_{1/2}$ and $0<r<1-|x|$}\,. 
\end{equation}
We are going  to prove that if $\eta_0$ is small enough, smallness depending only  on $n$ and $m$, then
$$|\nabla u_\varepsilon|^2\leq C\eta_0 \qquad\text{in $\overline{B}^+_{1/4}$}\,. $$
In view of the Neumann boundary condition, it clearly implies the full estimate~\eqref{tim2206}. 
We shall achieve it considering two different cases. In the spirit of \cite{CS}, we start with the localization of an interior region where $|\nabla u_\varepsilon|^2$ is large. 
\vskip3pt

Since $u_\varepsilon$ is smooth, we can find $\sigma_\varepsilon\in(0,1/2)$ such that 
\begin{equation}\label{defigma0}
(1/2-\sigma_\varepsilon)^2\sup_{ B^+_{\sigma_\varepsilon}} |\nabla u_\varepsilon|^2 = \max_{0\leq \sigma\leq 1/2}\left((1/2-\sigma)^2\sup_{ B^+_{\sigma}} |\nabla u_\varepsilon|^2\right)\,,
\end{equation}
and $x_\varepsilon\in B_{\sigma_\varepsilon}^+$ such that 
\begin{equation}\label{condx0}
2 |\nabla u_\varepsilon(x_\varepsilon)|^2\geq  \sup_{x\in B^+_{\sigma_\varepsilon}}  |\nabla u_\varepsilon|^2=:e_\varepsilon \,.
\end{equation}
Denote by $\bar x_\varepsilon\in D_{\sigma_\varepsilon}$ the orthogonal projection of $x_\varepsilon$ on $\mathbb{R}^{n}$, and set 
$$\rho_\varepsilon:=\frac{1}{2}\big(1/2-\sigma_\varepsilon\big)\qquad\text{and}\qquad R_\varepsilon:=\frac{1}{3}|x_\varepsilon-\bar x_\varepsilon|\,. $$
We now distinguish two cases. 
\vskip5pt

\noindent{\it Case 1).} We first assume that $R_\varepsilon>\rho_\varepsilon/7$. Since $\nabla u_\varepsilon$ is harmonic in $B_1^+$, we have 
$$\nabla u_\varepsilon(x_\varepsilon)=\frac{1}{|B_{R_\varepsilon}(x_\varepsilon)|}\int_{B_{R_\varepsilon}(x_\varepsilon)}\nabla u_\varepsilon \,\de x\,. $$
Then we deduce from the harmonicity of $u_\varepsilon$, Jensen's inequality, \eqref{contren}, and  \eqref{condx0} that 
$$e_\varepsilon\leq  \frac{2}{|B_{R_\varepsilon}(x_\varepsilon)|}\int_{B_{R_\varepsilon}(x_\varepsilon)}  |\nabla u_\varepsilon|^2\,\de x
\leq \frac{C}{R_\varepsilon^{n+1}} \int_{B_{5 R_\varepsilon}^+(\bar x_\varepsilon)}|\nabla u_\varepsilon|^2\,\de x\leq \frac{C}{R_\varepsilon^2}\eta_0\leq \frac{C}{\rho_\varepsilon^2}\eta_0\,.$$
Here we have also used that $B_{R_\varepsilon}(x_\varepsilon)\subset B^+_{5 R_\varepsilon}(\bar x_\varepsilon)\subset B_1^+$. Consequently,  
$$\rho_\varepsilon^2e_\varepsilon=\frac{1}{4}(1/2-\sigma_\varepsilon)^2e_\varepsilon\leq C\eta_0\,. $$
Choosing $\sigma=1/4$ in \eqref{defigma0}  yields
$$\sup_{ B^+_{1/4}}  |\nabla u_\varepsilon|^2\leq C\eta_0\,,$$
which is the announced estimate. 
\vskip5pt

\noindent{\it Case 2).} We now assume that $R_\varepsilon\leq \rho_\varepsilon/7$, and we claim that if $\eta_0$ is small enough, only depending on the dimensions $n$ and $m$, then
\begin{equation}\label{estrbare0}
R_\varepsilon\leq \frac{1}{7\sqrt{e_\varepsilon}}\,. 
\end{equation}
Indeed, assume that \eqref{estrbare0} does not hold. Then, arguing as in Case 1), we obtain 
$$e_\varepsilon\leq  \frac{C}{R_\varepsilon^2}\eta_0\leq C e_\varepsilon \eta_0\,.$$
Hence $1\leq C\eta_0$, which is impossible whenever $\eta_0$ is small enough. 
\vskip3pt

Next we  assume that $\eta_0$ is sufficiently small so that \eqref{estrbare0} holds. Then, 
$$|x_\varepsilon - \bar x_\varepsilon|\leq \min\left\{\frac{1}{2\sqrt{e_\varepsilon}}\,, \frac{\rho_\varepsilon}{2}\right\}\,. $$
Noticing that $B_{\rho_\varepsilon}^+(\bar x_\varepsilon)\subset B_{\sigma_\varepsilon+\rho_\varepsilon}^+$, we deduce that
\begin{equation}\label{estilargerball}
\sup_{ B^+_{\rho_\varepsilon}(\bar x_\varepsilon)}  |\nabla u_\varepsilon|^2\leq \sup_{ B^+_{\sigma_\varepsilon+\rho_\varepsilon}}  |\nabla u_\varepsilon|^2\,. 
\end{equation}
Since $\sigma_\varepsilon+\rho_\varepsilon\in(0,1/2)$, we infer from \eqref{defigma0} and \eqref{estilargerball} that 
\begin{equation}\label{pointwestigrd}
 \sup_{ B^+_{\rho_\varepsilon}(\bar x_\varepsilon)}  |\nabla u_\varepsilon|^2\leq 4 e_\varepsilon\,.
 \end{equation}
Let us now introduce the quantity
$$r_\varepsilon:=\rho_\varepsilon\sqrt{e_\varepsilon}\,. $$
{\it We claim that} 
\begin{equation}\label{r0less1}
r_\varepsilon\leq 1\,. 
\end{equation}
The proof of \eqref{r0less1} is postponed to the next subsection, and we  complete now the argument. 
Then, taking \eqref{r0less1} for granted,  \eqref{defigma0} yields
$$ \max_{0\leq \sigma\leq 1/2}\left((1/2-\sigma)^2\sup_{ B^+_{\sigma}} |\nabla u_\varepsilon|^2\right)=(1/2-\sigma_\varepsilon)^2\sup_{ B^+_{\sigma_\varepsilon}} |\nabla u_\varepsilon|^2 =4\rho_\varepsilon^2e_\varepsilon\leq 4\,.$$
Choosing $\sigma=1/4$ in the inequality above, we obtain 
$$\sup_{ B^+_{1/4}} |\nabla u_\varepsilon|^2\leq  64\leq C\eta_0\,, $$
as desired.
\qed 


\subsection{Proof of Theorem~\ref{epsreg}, Step 2.}  

 We are now going to  prove \eqref{r0less1} by contradiction, assuming that $r_\varepsilon>1$. The following proposition is the key point of the argument. 

\begin{proposition}\label{epsnotvanish}
If $\eta_0>0$ is small enough (depending only on $n$and $m$), then there exists a constant $\varsigma_0>0$ (depending only on $n$ and $m$) such that 
for each $\varepsilon>0$ and each map $v_\eps\in  C^2(\overline B_1^+;\R^m)$ solving 
\begin{equation}\label{systvn2}
\begin{cases}
\Delta v_\eps= 0&\text{in $B_1^+$}\,,\\[8pt]
\displaystyle \frac{\partial v_\eps}{\partial \nu}=\frac{1}{\varepsilon} (1-|v_\eps|^2)v_\eps & \text{on $D_1$}\,,
\end{cases}
\end{equation}
with the estimates
\begin{equation}\label{modgradvn2} 
1/2\leq |v_\eps|\leq 1 \quad\text{and}\quad |\nabla v_\eps|\leq 2 \quad\text{in $\overline{B}^+_1$}\,,
\end{equation}
and such that 
\begin{equation}\label{condgradvn2} 
|\nabla v_\eps(z_\eps)|^2\geq \frac{1}{2}
\end{equation}
for some point $z_\eps\in B^+_{1/2}$, the condition 
\begin{equation}\label{tim1738}
E_\eps(v_\eps,B_1^+)\leq 2^{n-1}\eta_0
\end{equation}
implies $\eps\geq \varsigma_0$. 
\end{proposition}

\begin{proof}
We argue by contradiction assuming that there exist a sequence $\eps_k\downarrow 0$, and corresponding 
solutions $\{v_{\varepsilon_k}\}_{k\in\mathbb{N}}$ of \eqref{systvn2} satisfying \eqref{modgradvn2},   \eqref{condgradvn2}, and   \eqref{tim1738}. 
To simplify the notation, we shall write $v_k:=v_{\eps_k}$ and $z_k:=z_{\eps_k}$. Note that $v_k$ is smooth in $B_1^+\cup D_1$ by Theorem \ref{regint}. 
\vskip3pt
Thanks to \eqref{modgradvn2} we can consider the smooth functions 
$$a_k:=|v_k|\quad\text{and}\quad w_k:=\frac{v_k}{|v_k|}\,.$$
Noticing that 
\begin{equation}\label{decompgradsq}
|\nabla v_k|^2=|\nabla a_k|^2+a_k^2|\nabla w_k|^2\,,
\end{equation}
we deduce from \eqref{modgradvn2} that 
\begin{equation}\label{gradboundronwn2}
1/2\leq a_k\leq 1\,,\quad |\nabla a_k|\leq 2\,,\quad |w_k|=1\,,\quad |\nabla w_k|\leq 4\,.
\end{equation}
In addition, system \eqref{systvn2} yields
\begin{equation}\label{eqron2}
\begin{cases}
-\Delta a_k +|\nabla w_k|^2a_k=0 & \text{in $B_1^+$}\,,\\[5pt]
\displaystyle \frac{\partial a_k}{\partial\nu}=\frac{1}{\varepsilon_k}(1-a_k^2)a_k & \text{on $D_1$}\,,
\end{cases}
\end{equation}
and 
\begin{equation}\label{eqwn2}
\begin{cases}
-{\rm div}(a_k^2\nabla w_k)=a_k^2|\nabla w_k|^2w_k  & \text{in $B_1^+$}\,,\\[5pt]
\displaystyle \frac{\partial w_k}{\partial\nu}=0 & \text{on $D_1$}\,.
\end{cases}
\end{equation}
In view of the boundary condition in \eqref{eqwn2}, we can extend $a_k$ and $w_k$ to $B_1$ by even reflection across $D_1$, {\it i.e.}, 
\begin{equation}\label{reflecaw}
a_k(x)=a_k(x',-x_{n+1}) \;\text{ and }\;Êw_k(x)=w_k(x',-x_{n+1})\; \text{ for } x\in B_1\setminus B_1^+\,,
\end{equation}
and then derive from \eqref{eqwn2},  
\begin{equation}\label{syswnwholball2}
-{\rm div}(a_k^2\nabla w_k)=a_k^2|\nabla w_k|^2w_k\quad\text{in }B_1\,.
\end{equation}
From \eqref{gradboundronwn2} and the boundary condition in \eqref{eqron2}, we  infer that 
\begin{equation}\label{conan12}
a_k\to a_* \text{ in $ C^{0,\alpha}(\overline B_1)$ for every $0<\alpha<1$}\,,
\end{equation}
for a (not relabeled) subsequence and  a function $a_*$ satisfying 
\begin{equation}\label{propainfty2}
\frac1 2\leq a_* \leq 1\,,\quad |\nabla a_*|\leq 2\,,\quad a_*=1\text{ on $D_1$}\,. 
\end{equation}
From \eqref{gradboundronwn2}, \eqref{syswnwholball2},  and standard elliptic regularity, we can find a further subsequence (not relabeled) such that  
\begin{equation}\label{conwn2}
w_k\to w_* \text{ in $ C^{1,\alpha}_{\rm loc}( B_1)\cap  C^{0,\alpha}(\overline B_1)$ for every $0<\alpha<1$}\,,
\end{equation}
for a map $w_*$ satisfying $|w_*|=1$, $|\nabla w_*|\leq 4$,  and solving 
\begin{equation}\label{eqwstar2}
\begin{cases}
-{\rm div}(a_*^2\nabla w_*)=a_*^2|\nabla w_*|^2w_*  & \text{in $B^+_1$}\,,\\[5pt]
\displaystyle \frac{\partial w_*}{\partial\nu}=0 & \text{on $D_1$}\,.
\end{cases}
\end{equation}
In view of \eqref{propainfty2}, we infer from \cite[Theorem 8.32]{GT} that 
\begin{equation}\label{C1alphcontrwinf2}
\|w_*\|_{ C^{1,\alpha}(B_r)} \leq C_{\alpha,r} 
\end{equation}
for every $0<\alpha<1$ and every $0<r<1$, where $C_{\alpha,r}$ denotes a constant which only depends on $n$, $m$, $\alpha$, and $r$. 

In turn, we infer from the $ C^{1,\alpha}_{\rm loc}$-convergence of $w_k$ and \eqref{eqron2} that 
\begin{equation}\label{conan22}
a_k\to a_* \text{ in $ C^{1,\alpha}_{\rm loc}( B_1^+)$ for every $0<\alpha<1$}\,,
\end{equation} 
again by elliptic regularity. As a consequence, $a_*$ solves 
\begin{equation}\label{eqrostar2}
\begin{cases}
-\Delta a_*+|\nabla w_*|^2a_*=0 & \text{in $B_1^+$}\,,\\[5pt]
 a_*=1 & \text{on $D_1$}\,.
\end{cases}
\end{equation}
Hence $a_*\in  C^{1,\alpha}_{\rm loc}(B^+_1\cup D_1)$ for every $0<\alpha<1$, and more precisely   \cite[Corollary~8.36]{GT} 
yields 
\begin{equation}\label{C1alphcontrainf2}
\|a_*\|_{ C^{1,\alpha}(B^+_r)} \leq C_{\alpha,r} 
\end{equation} 
 for every $0<\alpha<1$ and every $0<r<1$. 
 \vskip3pt
 
Finally the convergence of $a_k$ and $w_k$ implies that $v_k\to v_*:=a_* w_*$ in $ C^{1,\alpha}_{\rm loc}( B_1^+)\cap  C^{0,\alpha}( B_1^+)$. In addition  $v_*\in   C^{1,\alpha}_{\rm loc}( B_1^+\cup D_1)$ satisfies for every $0<\alpha<1$ and every $0<r<1$, 
\begin{equation}\label{C1alphcontrvinf2}
\|v_*\|_{ C^{1,\alpha}( B^+_r)} \leq C_{\alpha,r} \,, 
\end{equation}
 thanks to \eqref{C1alphcontrwinf2} and \eqref{C1alphcontrainf2}. 
\vskip5pt

We now claim that  $v_k$ actually converges to $v_*$ in the $ C^{1,\alpha}$-topology locally up to the boundary~$D_1$. 
Then we shall see that the estimates leading to the $C^{1,\alpha}$-convergence can be iterated  to obtain a  $C^{\ell}_{\rm loc}(B_1^+\cup D_1)$-convergence for every $\ell\in\N$. We shall complete the proof of Proposition \ref{epsnotvanish} right after proving the two following lemmas.

\begin{lemma}\label{convc1alph}
We have $v_k\to v_*$  in $ C^{1,\alpha}_{\rm loc}(B^+_{1}\cup D_{1})$ for every $0<\alpha<1$. 
\end{lemma}

\begin{proof} By \eqref{conwn2} we have $w_k\to w_*$ in $ C_{\rm loc}^{1,\alpha}(B^+_{1}\cup D_1)$, and it is enough to prove that 
$a_k\to a_*$  in $ C^{1,\alpha}_{\rm loc}(B^+_{1}\cup D_{1})$ for every $0<\alpha<1$. In view of \eqref{conan22}, it remains to show  the desired  
convergence near $D_1$. We fix $x_0\in D_1$, $0<\alpha<1$, and $0<r<{\rm dist}(x_0,\partial D_1)$ arbitrary. Without loss of 
generality we may assume that $x_0=0$. 

 We introduce the function 
$$h_k:=1-a_k\,,$$
and we notice that \eqref{gradboundronwn2} and the boundary condition in \eqref{eqron2} yield
\begin{equation}\label{encadr}
0\leq h_k \leq 3\eps_k \quad\text{on $D_1$\,.}
\end{equation}
Then we define for $y\in D_{r/4}$ fixed, 
\begin{equation}\label{defvarthet}
\vartheta_{k,y}(x):= h_k(x+y)-h_k(x) \quad\text{for $x\in \overline B_{r/2}^+$}\,.
\end{equation}
By \eqref{gradboundronwn2} 
the Lipschitz constant of $h_k$ is bounded by 2, and hence  
\begin{equation}\label{bdvarthet}
\|\vartheta_{k,y}\|_{L^\infty( B_{r/2}^+)}\leq 2 |y|\,. 
\end{equation}
In addition $\vartheta_{k, y}$ satisfies 
\begin{equation*}
\begin{cases}
-\Delta \vartheta_{k,y}= f_{k,y} & \text{in $B_{r/2}^+$}\,,\\[8pt]
\displaystyle \frac{\varepsilon_k}{a_k(1+a_k)}\frac{\partial \vartheta_{k,y}}{\partial\nu} + \vartheta_{k,y}= \varepsilon_k g_{k,y} & \text{on $D_{r/2}$\,,}
\end{cases}
\end{equation*}
with 
$$f_{k,y}(x):= |\nabla w_k(x+y)|^2a_k(x+y)- |\nabla w_k(x)|^2a_k(x)\,,$$
and 
$$g_{k,y}(x):=\bigg(\frac{1}{a_k(x)(1+a_k(x))} - \frac{1}{a_k(x+y)(1+a_k(x+y))}\bigg)\frac{\partial h_k(x+y)}{\partial\nu}\,.$$
From \eqref{gradboundronwn2}, \eqref{conwn2}, and \eqref{C1alphcontrwinf2}, 
we infer that 
\begin{equation}\label{estifny}
\|f_{k,y}\| _{L^\infty( B_{r/2}^+)}\leq C_{\alpha,r}|y|^\alpha \,, 
\end{equation}
and
\begin{equation}\label{estigny}
 \|g_{k,y}\| _{L^\infty( D_{r/2})}\leq C_{\alpha,r} |y|^\alpha\,.
 \end{equation}
\vskip3pt

Next we consider  the unique (variational) solution $\zeta_k\in H^1(B_{r/2}^+)$ of
$$\begin{cases}
-\Delta \zeta_k= 1 & \text{in $B_{r/2}^+$}\,,\\
\zeta_k=1 & \text{on $\partial^+ B_{r/2}$}\,,\\[8pt]
\displaystyle \frac{4\varepsilon_k}{3}\frac{\partial \zeta_k}{\partial\nu} + \zeta_k= 0 & \text{on $D_{r/2}$}\,,
\end{cases} $$
{\it i.e.},  $\zeta_k$ is the unique critical point of the strictly convex and coercive functional
$$J_k(\zeta):= \frac{1}{2}\int_{B^+_{r/2}} |\nabla \zeta|^2\,\de x -\int_{B^+_{r/2}}\zeta\,\de x+\frac{3}{8\eps_k}\int_{D_{r/2}}\zeta^2\,\de\mathscr{H}^n$$
defined over the affine space $\{\zeta\in H^{1}(B^+_{r/2}): \zeta = 1\text{ on $\partial^+B_{r/2}$}\}$. 
By Lemma \ref{supersol} in Appendix B,  $\zeta_k\in  C^{0,\beta}(\overline B_{r/2}^+)\cap  C^\infty(\overline B_{r/2}^+ \setminus \partial D_{r/2})$ 
for some $\beta\in(0,1)$, and  
\begin{equation}\label{controlsupsolnew}
0\leq \zeta_k\leq C_{r}\, \varepsilon_k\qquad\text{on } D_{r/4}\,,
\end{equation}
for a constant $C_r$ which only depends on $n$ and $r$.

Let us now define
\begin{equation}\label{defkappak}
\kappa_k(r,y):=\max\left\{ \|\vartheta_{k,y}\|_{L^\infty( B_{r/2}^+)} , \|f_{k,y}\| _{L^\infty( B_{r/2}^+)} \right\} \leq C_{\alpha,r} |y|^\alpha
\end{equation}
(recall \eqref{bdvarthet}, \eqref{estifny}, and that $|y|\leq r/4$), and  consider the functions 
$$H^+_k(x):= \kappa_k(r,y) \zeta_k(x)-\vartheta_{k,y}(x) + \varepsilon_k  \|g_{k,y}\| _{L^\infty( D_{r/2})}\,,$$
and 
$$H^-_k(x):= \kappa_k(r,y) \zeta_k(x)+ \vartheta_{k,y}(x) + \varepsilon_k  \|g_{k,y}\| _{L^\infty( D_{r/2})}\,. $$
By construction $H^\pm_k$ satisfies 
$$\begin{cases}
-\Delta  H^\pm_k\geq 0 & \text{in $B_{r/2}^+$}\,,\\[8pt]
 H^\pm_k\geq 0 & \text{on $\partial^+ B_{r/2}$}\,,\\[8pt]
\displaystyle \frac{\varepsilon_n}{a_k(1+a_k)}\frac{\partial  H^\pm_k}{\partial\nu} +  H^\pm_k\geq 0 & \text{on $D_{r/2}$\,.}
\end{cases} $$
By the maximum principle and the Hopf boundary lemma, it follows that $H^\pm_k\geq 0$ in $B_{r/2}^+$ (see {\it e.g.} the proof of Lemma \ref{supersol}). 
Therefore, 
\begin{equation}\label{estiunifvarthet}
\big|\vartheta_{k,y}(x)\big|  \leq  \kappa_k(r,y) \zeta_k(x) + \varepsilon_k  \|g_{k,y}\| _{L^\infty( D_{r/2})} \,. 
\end{equation}
\vskip3pt

Gathering \eqref{estigny}, \eqref{controlsupsolnew}, and \eqref{defkappak},  together this last estimate, we deduce that 
$$\left| \frac{h_k(x+y)-h_k(x)}{ \varepsilon_k}\right| \leq C_{\alpha,r} |y|^\alpha\quad \text{for every $x,y\in D_{r/4}$}\,. $$
In view of  \eqref{encadr},  we have thus proved that 
\begin{equation}\label{tim1146}
\left\|\frac{1-a_k}{ \varepsilon_k}  \right\|_{ C^{0,\alpha}(D_{r/4})}\leq C_{\alpha,r}\,. 
\end{equation}
From  of \eqref{gradboundronwn2} we conclude that 
\begin{equation}\label{holdcontr1-ansureps}
\left\|\frac{1}{ \varepsilon_k} (1-a^2_k)a_k \right\|_{ C^{0,\alpha}(D_{r/4})}\leq C_{\alpha,r}\,. 
\end{equation}
Applying Lemma~\ref{regNeum} in Appendix B to equation \eqref{eqron2} together with estimates \eqref{conwn2}, \eqref{C1alphcontrwinf2}, and \eqref{holdcontr1-ansureps}, we finally deduce that 
$$\|a_k\|_{ C^{1,\alpha}( B_{r/8}^+)}\leq C_{\alpha,r}\,, $$ 
and thus $a_k\to a_*$ in $ C^{1,\alpha^\prime}(B_{r/8}^+)$ for every $0<\alpha^\prime<\alpha$. 
\end{proof}

\begin{lemma}\label{ImprovedCV}
We have $v_k\to v_*$  in $ C^{\ell}_{\rm loc}(B^+_{1}\cup D_{1})$ for every $\ell\in\N$. 
\end{lemma}

\begin{proof}
In view of Lemma \ref{convc1alph} we can argue by induction assuming that for an integer $\ell\geq1$, $w_k\to w_*$ and $a_k\to a_*$ in $ C^{\ell,\alpha}_{\rm loc}(B^+_{1}\cup D_{1})$ for every $\alpha\in(0,1)$. We aim to prove that  $w_k\to w_*$ and $a_k\to a_*$ in $ C^{\ell+1,\alpha}_{\rm loc}(B^+_{1}\cup D_{1})$ for every $\alpha\in(0,1)$.
We proceed with the notations used in the proof of Lemma~\ref{convc1alph}. 
\vskip5pt

\noindent{\it Step 1: Improved convergence of $\{w_k\}$.} Using the even reflection across $D_1$ as stated in \eqref{reflecaw}, we first rewrite equation \eqref{syswnwholball2} as 
\begin{equation}\label{deltawk}
-\Delta w_k= \frac{2}{a_k}\nabla a_k\cdot\nabla w_k+ |\nabla w_k|^2w_k\quad\text{in $B_1$}\,.
\end{equation}
(Recall that $a_k\geq 1/2$.) Let us now fix a multi-index $\beta=(\beta_1,\ldots\beta_n)\in\N^n$  of length $|\beta|:=\sum_i\beta_i=\ell-1$, and set 
$$\partial^\beta:=\partial_1^{\beta_1}\partial_2^{\beta_2} \ldots \partial_n^{\beta_n}\,.$$ 
Using the boundary condition $\partial_\nu w_k=0$ on $D_1$, we notice that the term $\frac{2}{a_k}\nabla a_k\cdot\nabla w_k$ is actually continuous across $D_1$. As a consequence, $\{\partial^\beta(\frac{2}{a_k}\nabla a_k\cdot\nabla w_k+|\nabla w_k|^2w_k)\}$ is bounded in $C^{0,\alpha}_{\rm loc}(B_1)$ for every $0<\alpha<1$, by the induction hypothesis. By elliptic regularity, it implies that $\{\partial^{\beta}w_k\}$ is bounded in $C^{2,\alpha}_{\rm loc}(B_1)$ for every $0<\alpha<1$. 

If $\ell=1$ the required improved convergence readily follows, and we may now assume that $\ell\geq 2$. In view of the arbitrariness of the multi-index $\beta$ above, we have  showed that 
\begin{enumerate}
\item[(A0)] $\{\partial^{\beta}  w_k\}$ is bounded in $C^{0,\alpha}_{\rm loc}(B^+_1\cup D_1)$ for  $\beta\in\N^n$ such that $|\beta|=\ell+1$ and $0<\alpha<1$;
\item[(B0)] $\{\partial_{n+1}\partial^{\beta} w_k\}$ is bounded in $C^{0,\alpha}_{\rm loc}(B^+_1\cup D_1)$ for $\beta\in\N^n$ such that $|\beta|=\ell$ and $0<\alpha<1$;
\item[(C0)] $\{\partial^2_{n+1}\partial^{\beta} w_k\}$ is bounded in $C^{0,\alpha}_{\rm loc}(B^+_1\cup D_1)$ for  $\beta\in\N^n$ such that $|\beta|=\ell-1$ and $0<\alpha<1$.
\end{enumerate}
Next we want to estimate the remaining derivatives of $w_k$ of order $\ell+1$. We proceed by induction assuming that for an integer $0\leq \gamma\leq \ell-2$, 
\begin{enumerate}
\item[(A$\gamma$)] $\{\partial^\gamma_{n+1}\partial^{\beta}  w_k\}$ is bounded in $C^{0,\alpha}_{\rm loc}(B^+_1\cup D_1)$ for  $\beta\in\N^n$ such that $\gamma+|\beta|=\ell+1$ and $0<\alpha<1$;
\item[(B$\gamma$)] $\{\partial^{\gamma+1}_{n+1}\partial^{\beta} w_k\}$ is bounded in $C^{0,\alpha}_{\rm loc}(B^+_1\cup D_1)$ for  $\beta\in\N^n$ such that $\gamma+|\beta|=\ell$ and $0<\alpha<1$;
\item[(C$\gamma$)] $\{\partial^{\gamma+2}_{n+1}\partial^{\beta} w_k\}$ is bounded in $C^{0,\alpha}_{\rm loc}(B^+_1\cup D_1)$ for  $\beta\in\N^n$ such that $\gamma+|\beta|=\ell-1$ and $0<\alpha<1$,
\end{enumerate}
and we aim to show that $\{\partial^{\gamma+3}_{n+1}\partial^{\beta} w_k\}$ is bounded in $C^{0,\alpha}_{\rm loc}(B^+_1\cup D_1)$ for  $\beta\in\N^n$ such that $\gamma+|\beta|=\ell-2$ and $0<\alpha<1$. Using equation \eqref{deltawk}, we have 
$$-\partial^{\gamma+3}_{n+1}\partial^{\beta} w_k= \Delta^\prime\big(\partial^{\gamma+1}_{n+1}\partial^{\beta} w_k)+\partial^{\gamma+1}_{n+1}\partial^{\beta}\left(\frac{2}{a_k}\nabla a_k\cdot\nabla w_k+ |\nabla w_k|^2w_k\right)\quad\text{in $B_1^+$}\,,$$
where $\Delta^\prime:=\sum_{j=1}^n\partial^2_j\,$. By (B$\gamma$) and the induction hypothesis, the right hand side of the above equation remains bounded in $C^{0,\alpha}_{\rm loc}(B^+_1\cup D_1)$, and thus $\{\partial^{\gamma+3}_{n+1}\partial^{\beta} w_k\}$ is indeed bounded in $C^{0,\alpha}_{\rm loc}(B^+_1\cup D_1)$. 

In conclusion, we have thus proved that $\{w_k\}$ remains bounded in $C^{\ell+1,\alpha}_{\rm loc}(B^+_1\cup D_1)$ for every exponent $0<\alpha<1$, and the improved convergence follows .
\vskip5pt

\noindent{\it Step 2: Improved convergence of $\{a_k\}$.} 
By  elliptic regularity (applied to \eqref{eqron2}), the improved convergence of $\{w_k\}$ implies 
the convergence of $a_k$ in $ C^{\ell+1,\alpha}_{\rm loc}(B_1^+)$ for every $0<\alpha<1$. It then remains to improve the convergence of $a_k$ (or equivalently $h_k$) up to the boundary $D_1$. As in the proof of Lemma \ref{convc1alph}, we consider an arbitrary $x_0\in D_1$, a radius $0<2r<{\rm dist}(x_0,\partial D_1)$, and $\alpha\in(0,1)$. Again we can assume that $x_0=0$ without loss of generality. Let us now fix a multi-index $\beta=(\beta_1,\ldots\beta_n)\in\N^n$  of length $|\beta|:=\sum_i\beta_i=\ell$. 
Recalling that $h_k:=1-a_k$, we notice that $\partial^\beta h_k$ solves 
\begin{equation}\label{equpartialbetah}
\begin{cases}
-\Delta \big( \partial^\beta h_k\big)=\partial^\beta\big(|\nabla w_k|^2a_k\big) & \text{in $B_{2r}^+$}\,,\\[8pt]
\displaystyle \frac{\varepsilon_k}{a_k(1+a_k)}\frac{\partial\big( \partial^\beta h_k\big)}{\partial\nu} + \partial^\beta h_k= \varepsilon_k R_{\beta,k} & \text{on $D_{2r}$\,,}
\end{cases}
\end{equation}
where 
$$R_{\beta,k} := \frac{1}{a_k(1+a_k)}\frac{\partial\big( \partial^\beta h_k\big)}{\partial\nu}  - \partial^\beta \left(\frac{1}{a_k(1+a_k)}\frac{\partial  h_k}{\partial\nu} \right) $$
involves only derivatives of $a_k$ of order less than or equal to $\ell$. Recalling that $a_k\geq 1/2$, we deduce from the induction hypothesis and the boundedness of $\{w_k\}$ in $C^{\ell+1,\alpha}(B^+_{2r})$ that 
\begin{equation}\label{bdmbeta}
m_\beta:=\sup_k\left(\| \partial^\beta h_k\|_{L^\infty(B^+_{2r})}+\|R_{\beta,k} \|_{C^{0,\alpha}(D_{2r})}+\big \|\partial^\beta\big(|\nabla w_k|^2a_k\big) \big \|_{C^{0,\alpha}(B^+_{2r})}\right)<\infty\,.
\end{equation}
Arguing as in the proof of Lemma \ref{convc1alph} to derive \eqref{estiunifvarthet}, we infer that $|\partial^\beta h_k |\leq C_r m_\beta \eps_k$ on $D_{r}$, which then yields 
\begin{equation}\label{higherordnormder}
\left\| \frac{\partial\big( \partial^\beta h_k\big)}{\partial\nu} \right\|_{L^\infty(D_{r})} \leq C_r m_\beta\,. 
\end{equation}
Next we consider for $y\in D_{r/4}$ the function $\vartheta_{k,y}$ defined in \eqref{defvarthet}. By the induction hypothesis, we have $\|\partial^\beta\vartheta_{k,y}\|_{L^\infty(B^+_{r/2})}\leq c_1|y|^\alpha$ for some constant $c_1$ independent of $k$. On the other hand, $ \partial^\beta \vartheta_{k,y}$ solves 
\begin{equation*}
\begin{cases}
-\Delta \big( \partial^\beta \vartheta_{k,y}\big)=f^\beta_{k,y}& \text{in $B_{r/2}^+$}\,,\\[8pt]
\displaystyle \frac{\varepsilon_k}{a_k(1+a_k)}\frac{\partial\big( \partial^\beta \vartheta_{k,y}\big)}{\partial\nu} + \partial^\beta h_k= \varepsilon_k g^\beta_{k,y} & \text{on $D_{r/2}$\,,}
\end{cases}
\end{equation*}
with 
$$f^\beta_{k,y}(x):= \partial^\beta\big(|\nabla w_k|^2a_k\big)(x+y)- \partial^\beta\big(|\nabla w_k|^2a_k\big)(x)\,,$$
and 
$$g^\beta_{k,y}(x):=R_{\beta,k}(x+y)-R_{\beta,k}(x)+\bigg(\frac{1}{a_k(x)(1+a_k(x))} - \frac{1}{a_k(x+y)(1+a_k(x+y))}\bigg)\frac{\partial \big(\partial^\beta h_k)}{\partial\nu}(x+y)\,.$$
By the induction hypothesis and \eqref{higherordnormder}, we have 
$$\|f^\beta_{k,y}\|_{L^\infty(B^+_{r/2})} + \|g^\beta_{k,y}\|_{L^\infty(D_{r/2})}\leq c_2|y|^\alpha $$
for a constant $c_2$ independent of $k$. Arguing precisely as in the proof of Lemma \ref{convc1alph}, we deduce that 
$$\left| \frac{\partial^\beta h_k(x+y)-\partial ^\beta h_k(x)}{ \varepsilon_k}\right| \leq c_3 |y|^\alpha\quad \text{for every $x,y\in D_{r/4}$}\,, $$
and a constant $c_3$ independent of $k$. Whence, 
$$\sup_k\left\|\frac{\partial^\beta h_k}{\eps_k}\right\|_{C^{0,\alpha}(D_{r/4})} <\infty\,,$$
and inserting this estimate together with \eqref{bdmbeta} in  equation \eqref{equpartialbetah}, we conclude from Lemma~\ref{regNeum} that $\{\partial^\beta h_k\}$ remains bounded in $C^{1,\alpha}(B^+_{r/8})$. 

In view of the arbitrariness of the multi-index $\beta$ above, we have showed that 
\begin{enumerate}
\item[(D0)] $\{\partial^{\beta}  h_k\}$ is bounded in $C^{0,\alpha}(B^+_{r/8})$ for every $\beta\in\N^n$ such that $|\beta|=\ell+1$;
\item[(E0)] $\{\partial_{n+1}\partial^{\beta} h_k\}$ is bounded in $C^{0,\alpha}(B^+_{r/8})$ for every $\beta\in\N^n$ such that $|\beta|=\ell$.
\end{enumerate}
To complete the proof we need to estimate the remaining derivatives of $h_k$ of order $\ell+1$. Again, we proceed by induction assuming that for an integer $0\leq \gamma\leq \ell-1$, 
\begin{enumerate}
\item[(D$\gamma)$] $\{\partial_{n+1}^\gamma\partial^{\beta} \! h_k\}$ is bounded in $C^{0,\alpha}(B^+_{r/8})$ for every $\beta\in\N^n$ such that $\gamma+|\beta|=\ell+1$;
\item[(E$\gamma)$] $\{\partial^{\gamma+1}_{n+1}\partial^{\beta} h_k\}$ is bounded in $C^{0,\alpha}(B^+_{r/8})$ for every $\beta\in\N^n$ such that $\gamma+|\beta|=\ell$,
\end{enumerate}
and we aim to show that $\{\partial^{\gamma+2}_{n+1}\partial^{\beta} h_k\}$ is bounded in $C^{0,\alpha}(B^+_{r/8})$ for every multi-index $\beta\in\N^n$ such that $\gamma+|\beta|=\ell-1$. Using equation \eqref{equpartialbetah}, we notice that 
\begin{equation}\label{eqinducgamma}
\partial^{\gamma+2}_{n+1}\partial^{\beta} h_k=-\partial^\gamma_{n+1}\Delta^\prime \partial^{\beta}h_k -\partial_{n+1}^\gamma\partial^{\beta}\big(|\nabla w_k|^2a_k\big) \quad\text{in $B^+_{r/8}$}\,. 
\end{equation}
 Clearly the right hand side of \eqref{eqinducgamma} is bounded in $C^{0,\alpha}(B^+_{r/8})$ by (D$\gamma$) and the induction hypothesis, and thus  $\{\partial^{\gamma+2}_{n+1}\partial^{\beta} h_k\}$ is indeed bounded in $C^{0,\alpha}(B^+_{r/8})$. 

In conclusion, the sequence $\{a_k\}$ is bounded in $C^{\ell+1,\alpha}(B^+_{r/8})$, and hence $a_k\to a_*$ in  $C^{\ell+1,\alpha^\prime}(B^+_{r/8})$ for every $0<\alpha^\prime<\alpha$. 
\end{proof}

\noindent{\it Proof of Proposition \ref{epsnotvanish} completed.} 
Up to a subsequence, we have $z_k\to z_*$ for some point $z_*\in \overline B_{1/2}^+$. 
Thanks to Lemma~\ref{convc1alph} we have $\nabla v_k(z_k)\to \nabla v_*(z_*)$. Then  \eqref{condgradvn2} leads 
to $|\nabla v_*(z_*)|^2\geq 1/2$. From \eqref{C1alphcontrvinf2} we infer that 
$$|\nabla v_*|^2\geq 1/4\qquad\text{in $B_1^+\cap B_{\varrho_0}(z_*)$} $$
for some radius $0<\varrho_0\leq 1/4$ which only depends on $n$ and $m$.  
Using assumption \eqref{tim1738} we now estimate  
\begin{equation}\label{tim2240}
\frac{\varrho_0^{n+1}|B_1^+|}{4}\leq \int_{B^+_1}|\nabla v_*|^2\,dx \leq \liminf_{k\to\infty} \int_{B^+_1}|\nabla v_k|^2\,dx\leq 2^n \eta_0\,.
\end{equation}
We then find a contradiction if $\eta_0$ is small enough, and the proposition is proved.   
\end{proof}

\begin{proof}[Proof of  \eqref{r0less1} completed.] We argue by contradiction assuming that $r_\varepsilon>1$. 
We consider the rescaled function 
$$\tilde u_\eps(x):=u_{\varepsilon}\left(\frac{x}{\sqrt{e_{\varepsilon}}}+\bar x_{\varepsilon}\right) \qquad \text{for $x\in \overline{B}^+_1$}\,,$$
so that $\tilde u_\eps\in  C^2(\overline B_1^+;\R^m)$ solves 
$$\begin{cases}
\Delta \tilde u_\eps= 0&\text{in $B_1^+$}\,,\\[8pt]
\displaystyle \frac{\partial \tilde u_\eps}{\partial \nu}=\frac{1}{\tilde \varepsilon} (1-|\tilde u_\eps|^2)\tilde u_\eps & \text{on $D_1$}\,,
\end{cases}$$
with $\tilde \varepsilon:=\varepsilon\sqrt{e_{\varepsilon}}\,$. 
Moreover, \eqref{clearmodu} and \eqref{pointwestigrd} imply  that 
$1/2\leq |\tilde u_\eps|\leq 1$ and $|\nabla \tilde u_\eps|\leq 2$ in $\overline{B}^+_1$. 
Considering the point $z_\eps:= \sqrt{e_{\varepsilon}}(x_{\varepsilon}- \bar x_{\varepsilon})\in B_{1/2}^+$
(which actually belongs to the half axis $\{0\}\times \mathbb{R}_+$ and satisfies $|z_\eps|=3R_{\varepsilon}\sqrt{e_{\varepsilon}}\leq\frac{3}{7}$), 
we observe that \eqref{condx0} yields
\begin{equation}\label{tim2242}
|\nabla \tilde u_\eps(z_\eps)|^2\geq \frac{1}{2}\,. 
\end{equation}
On the other hand, \eqref{contren} leads to 
\begin{equation}\label{tim2242bis}
E_{\tilde \eps}(\tilde u_\eps,B_1)= (\sqrt{e_\eps})^{n-1} E_\eps\big(u_\eps,B^+_{(\sqrt{e_\eps})^{-1}}(\bar x_\eps)\big)\leq 2^{n-1}\eta_0\,. 
\end{equation}
If $\eta_0$ is small enough, we conclude from Proposition~\ref{epsnotvanish}  that $\tilde\eps\geq\varsigma_0$ where $\varsigma_0>0$ only depends on $n$ and $m$.  Applying  \cite[Lemma~2.2]{CSM}  we infer that 
$\|\tilde u_\varepsilon\|_{ C^{2,\alpha}( B^+_{3/4})}\leq C_{\alpha}$ for every $0<\alpha<1$, where  $C_\alpha$ only depends on $\alpha$, $n$, and $m$. 
Arguing as in \eqref{tim2240}, we then find a contradiction between \eqref{tim2242} and \eqref{tim2242bis} whenever $\eta_0$ is sufficiently small. 
\end{proof}

\vskip10pt

\section{Asymptotics for Ginzburg-Landau boundary reactions}\label{asymptGLBR}                            

 \subsection{Convergence to boundary harmonic maps and defect measures}
 
 With Theorem \ref{epsreg} in hands, we are now able to give a preliminary description of  both weak limits as $\eps\downarrow0$ of critical points of  $E_\varepsilon$, and the possible defect measure arising in the weak convergence process. This is the object of the following theorem.

 \begin{theorem}\label{asymptneum}
Let $\Omega\subset\R^{n+1}_+$ be an admissible bounded  open set. Let $\varepsilon_k\downarrow0$ be an arbitrary sequence,  and let  
$\{u_k\}_{k\in\mathbb{N}}\subset H^1(\Omega;\mathbb{R}^m)\cap L^\infty(\Omega)$ be such that for every $k\in\N$, $|u_k|\leq 1$, and $u_k$ weakly solves  
$$ \begin{cases}
\Delta u_k= 0 & \text{in $\Omega$}\,,\\[8pt]
\displaystyle \frac{\partial u_k}{\partial \nu}=\frac{1}{\varepsilon_k}(1-|u_k|^2)u_k & \text{on $\partial^0\Omega$}\,. 
\end{cases}
$$
If  $\sup_k E_{\varepsilon_k}(u_k,\Omega)<\infty$, then there exist a (not relabeled) subsequence and a bounded weak $(\mathbb{S}^{m-1},\partial^0\Omega)$-boundary  
harmonic map $u_*$ in $\Omega$ such that $u_k\rightharpoonup u_*$ weakly in $H^1(\Omega)$ as $k\to\infty$. In addition, there exist a finite 
nonnegative Radon measure $\mu_{\rm sing}$ on $\partial^0\Om$ and a relatively closed set $\Sigma\subset \partial^0\Omega$ 
of finite $(n-1)$-dimensional Hausdorff measure such that  
\vskip5pt
\begin{itemize}[leftmargin=22pt]
\item[\rm (i)] $\displaystyle|\nabla u_k|^2\mathscr{L}^{n+1}\LL\Om  \rightharpoonup 
|\nabla u_*|^2\mathscr{L}^{n+1}\LL\Om +\mu_{\rm sing} $ 
weakly* as Radon measures on $\Omega\cup\partial^0\Omega$;   
\vskip5pt

\item[\rm  (ii)] $\displaystyle\frac{(1-|u_k|^2)^2}{\varepsilon_k}\to 0$ in $L^1_{\rm loc}(\partial^0\Om)$; 
\vskip5pt

\item[ \rm  (iii)] $\Sigma={\rm supp}(\mu_{\rm sing})\cup {\rm sing}(u_*)$; 
\vskip5pt

\item[\rm  (iv)] $u_k\to u_*$ in $ C^{\ell}_{\rm loc}\big((\Omega\cup\partial^0\Omega)\setminus \Sigma\big)$ for every $\ell\in\N$;
\vskip5pt 

\item[\rm  (v)] if $n=1$ the set $\Sigma$ is finite and $u_*\in  C^\infty(\Omega\cup\partial^0\Omega)$. 

\end{itemize}
\end{theorem}

\begin{proof}
{\it Step 1.} First notice that Theorem \ref{regint} yields $u_k\in C^\infty(\Omega\cup\partial^0\Om)$. 
By the uniform energy bound and the assumption $|u_k|\leq 1$, we can find a (not relabeled) subsequence  such that 
$u_k\rightharpoonup u_*$ weakly in $H^1(\Omega)$ for some map $u_*\in H^1(\Omega;\R^m)$.  
Since $|u_k|\leq 1$  and $u_k$ is harmonic in $\Omega$, we deduce that $u_k\to u_*$ in $ C^\ell_{\rm loc}(\Omega)$ for every $\ell\in\mathbb{N}$,   
$u_*$ is harmonic in $\Omega$, and $|u_*|\leq 1$ in $\Omega$.  On the other hand, $|u_k|\to 1$ in $L^2(\partial^0\Omega)$, and we infer from the 
compact imbedding $H^1(\Omega)\hookrightarrow L^2(\partial^0\Omega)$ that $|u_*|=1$ $\mathscr{H}^n$-a.e. on $\partial^0\Omega$. 
\vskip3pt

It then remains to analyse the asymptotic behavior of $u_k$ near $\partial^0\Omega$. 
Setting 
$$\mu_k:= |\nabla u_k|^2\mathscr{L}^{n+1}\LL\Om +\frac{1}{2\varepsilon_k}(1-|u_k|^2)^2\mathscr{H}^n \LL\partial^0\Omega\,,$$ 
we have $\sup_k\mu_k(\Om\cup\partial^0\Om)<\infty$. Hence we can find a further subsequence such that 
\begin{equation}\label{tim2253}
\mu_k\rightharpoonup \mu:=|\nabla u_*|^2\mathscr{L}^{n+1}\LL\Om+\mu_{\rm sing}\,,
\end{equation}
weakly* as Radon measures on $\Omega\cup\partial^0\Omega$ for some nonnegative $\mu_{\rm sing}$ on $\Omega\cup\partial^0\Omega$. 
Notice that the local smooth convergence of $u_k$ to $u_*$ in $\Omega$ implies that 
\begin{equation}\label{tim1751}
{\rm supp}(\mu_{\rm sing})\subset \partial^0\Omega
\end{equation}   
(here ${\rm supp}(\mu_{\rm sing})$ denotes the relative support of $\mu_{\rm sing}$ in $\Om\cup\partial^0\Om$). 
By Lemma \ref{monotform},  we have 
\begin{equation}\label{tim2332}
\rho^{1-n}\mu_k(B_\rho(x))\leq r^{1-n}\mu_k(B_r(x))
\end{equation}
for every $x\in\partial^0\Om$ and every $0<\rho<r<{\rm dist}(x,\partial^+\Om)$. Therefore,
\begin{equation}\label{tim1213}
\rho^{1-n}\mu(B_\rho(x))\leq r^{1-n}\mu(B_r(x))
\end{equation}
for every $x\in\partial^0\Om$ and every $0<\rho<r<{\rm dist}(x,\partial^+\Om)$. As a consequence, the $(n-1)$-dimensional density 
\begin{equation}\label{existdens}
\Theta^{n-1}(\mu,x):=\lim_{r\downarrow 0}\, \frac{\mu(B_r(x))}{{\boldsymbol\omega}_{n-1}r^{n-1}}
\end{equation}
exists and is finite at every point $x\in\partial^0\Omega$. Here ${\boldsymbol\omega}_{n-1}$ denotes the volume of the unit ball in~$\R^{n-1}$ if $n\geq 2$, and ${\boldsymbol\omega}_0=1$.  
Note that \eqref{tim2253} and \eqref{tim2332} yield
\begin{equation}\label{upbddensity}
\Theta^{n-1}(\mu,x)\leq \frac{C}{\big({\rm dist}(x,\partial^+\Om)\big)^{n-1}} \,\sup_{k\in\N} E_{\eps_k}(u_k,\Om)<\infty \quad\text{for all $x\in\partial^0\Om$}\,. 
\end{equation}
On the other hand, by the smooth convergence of $u_k$ toward $u_*$ in $\Omega$, 
$$\Theta^{n-1}(\mu,x)=0\quad\text{for all $x\in\Omega$}\,. $$
In addition, we observe that $x\in\partial^0\Om\mapsto \Theta^{n-1}(\mu,x)$ is upper semicontinuous. 
\vskip3pt

Next we define the concentration set
$$\Sigma:=\bigg\{x\in \partial^0\Omega :  \inf_r\big\{ \liminf_{k\to\infty}\, r^{1-n}\mu_k(B_r(x)) : 
0<r<{\rm dist}(x,\partial^+\Om)\big\}\geq \eta_0\bigg\}\,, $$
where $\eta_0>0$ is the constant given by Theorem~\ref{epsreg}. From \eqref{tim2332} and \eqref{tim1213} we infer that 
$$\Sigma =\bigg\{x\in \partial^0\Omega :  \lim_{r\downarrow 0} \,\liminf_{k\to\infty}\, r^{1-n}\mu_k(B_r(x)) \geq \eta_0\bigg\} \\
 = \bigg\{x\in \partial^0\Omega :  \lim_{r\downarrow 0}\, r^{1-n}\mu(B_r(x)) \geq \eta_0\bigg\} \,,$$
and consequently, 
\begin{equation}\label{tim1721}
\Sigma=\bigg\{x\in \partial^0\Om: \Theta^{n-1}(\mu,x)\geq \frac{\eta_0}{{\boldsymbol\omega}_{n-1}} \bigg\} \,. 
\end{equation}
In particular, $\Sigma$ is a relatively closed subset of $\partial^0\Om$ (since $\Theta(\mu,\cdot)$ is upper semicontinuous). 
Moreover, by a well known property of upper densities (see {\it e.g.} \cite[Theorem 2.56]{AFP}), we have
\begin{equation}\label{train}
 \frac{\eta_0}{{\boldsymbol\omega}_{n-1}}\mathscr{H}^{n-1}(\Sigma)  \leq \mu(\Sigma)<\infty\,.
 \end{equation} 
If $n=1$, it obviously implies that $\Sigma$ is  finite. 
\vskip5pt

\noindent{\it Step 2.} Let us now show that $u_k\to u_*$ in 
$ C^{\ell}_{\rm loc}\big((\Omega\cup\partial^0\Omega)\setminus \Sigma\big)$ for every $\ell\in\N$.  In view of 
the local smooth convergence of $u_k$ in $\Omega$, it suffices to prove the claim near $\partial^0\Om$. 
To this purpose, let us fix $x_0 \in \partial^0\Om\setminus\Sigma$ and $0<r<{\rm dist}(x_0,\partial^+\Om\cup\Sigma)$ such that 
$r^{1-n}\mu(B_r(x_0))<\eta_0$. By  the monotonicity in \eqref {tim1213}, we may assume without loss of generality  that $\mu(\partial B_r(x_0))=0$. 
Then $\lim_k\mu_k(B_r(x_0))=\mu(B_r(x_0))$, which in turn implies that $r^{1-n}\mu_k(B_r(x_0))\leq \eta_0$ for $k$ sufficiently large. Taking $k$ even larger we have $\varepsilon_k\leq r$, and we can then  apply Theorem \ref{epsreg} and Lemma \ref{clear2} to deduce that $|\nabla u_k|\leq C_r$ and $1/2\leq |u_k|\leq 1$ in $B^+_{r/4}(x_0)$.  
This is now enough to reproduce the convergence proof in Proposition~\ref{epsnotvanish}, Lemma~\ref{convc1alph}, and Lemma \ref{ImprovedCV}. 
It shows that $u_k\to u_*$ in $ C^{\ell}_{\rm loc}\big(B^+_{r/4}(x_0)\cup D_{r/4}(x_0)\big)$ for every $\ell\in\N$. We finally notice 
that Theorem \ref{epsreg} also provides the estimate $(1-|u_k|^2)^2\leq C_r\eps_k^2$ in $D_{r/4}(x_0)$.  As a consequence,
\begin{equation}\label{tim1755}
\lim_{k\to\infty} \frac{1}{\eps_k}\int_{D_{r/4}(x_0)}(1-|u_{k}|^2)^2\,\de \mathscr{H}^n= 0\,, 
\end{equation}
a fact that we shall use later on.
\vskip5pt

\noindent{\it Step 3.} Let us now prove that $u_*$ is a weak $\mathbb{S}^{m-1}$-boundary harmonic map in $\Omega\cup\partial^0\Om$. We distinguish the two cases 
$n=1$ and $n\geq 2$. 
\vskip3pt

\noindent{\it Case 1, $n=1$.} Let $\Phi\in H^1(\Om;\R^m)\cap L^\infty(\Om)$ with compact support  in $\Om\cup\partial^0\Om$ such that 
$\Phi(x)\in T_{u_*(x)}\mathbb{S}^{m-1}$ for $\mathscr{H}^n$-a.e. $x\in\partial^0\Om$. By Step 1 the set ${\rm supp}\,\Phi \cap\Sigma$ 
contains finitely many points $b_1,\ldots,b_L$. Then, fix an arbitray cut-off function 
$\zeta \in  C^\infty(\R^{n+1};[0,1])$ such that $\zeta=0$ in a small neighborhood of each $b_l$. 
We set $\widetilde\Phi:=\zeta\Phi$, so that $\widetilde \Phi$ has compact support in $(\Omega\cup\partial^0\Om)\setminus \Sigma$. 
From the convergence of $u_k$ established in Step~2, we have $|(1-|u_k|^2)||\widetilde\Phi|\leq C\varepsilon_k$ on $\partial^0\Om$, and thus 
$$\lim_{k\to\infty} \frac{1}{\varepsilon_k}\int_{\partial^0\Om} (1-|u_k|^2)u_k\cdot\widetilde\Phi\,\de \mathscr{H}^n
=\lim_{k\to\infty} \frac{1}{\varepsilon_k}\int_{\partial^0\Om} (1-|u_k|^2)(u_k-u_*)\cdot\widetilde\Phi\,\de \mathscr{H}^n=0\,,$$
by dominated convergence. On the other hand, 
$$\lim_{k\to\infty} \int_\Om\nabla u_k\cdot\nabla\widetilde\Phi\,\de x=  \int_\Om\nabla u_*\cdot\nabla\widetilde\Phi\,\de x\,,$$
and we deduce that 
$$ \int_\Om\nabla u_*\cdot\nabla\widetilde\Phi\,\de x =0\,.$$
Given an arbitrary $\delta>0$,  we now choose the cut-off function $\zeta$ of the form $\zeta(x)=\chi_\delta(x)\widehat\zeta(x)$ 
where  $\widehat\zeta \in  C^\infty(\R^{n+1};[0,1])$ satisfies $\widehat \zeta=0$ 
in a small neighborhood of $b_l$ only for $l\geq 2$, and $\chi_\delta\in C_c^\infty(\R^{n+1};[0,1])$ 
satisfies $\chi_\delta=0$ in $B_\delta(b_1)$,  $\chi_\delta=1$ outside $B_{2\delta}(b_1)$, and $|\nabla\chi_\delta|\leq C/\delta$. 
Setting $\widehat\Phi=\widehat\zeta\Phi$, we have
\begin{equation}\label{tim1503}
\int_{\Om}\chi_\delta\nabla u_*\cdot\nabla\widehat\Phi\,dx + \int_{\Om\cap B^+_{2\delta}(b_1)} \widehat\Phi\cdot(\nabla\chi_\delta\cdot\nabla u_*)\,\de x=0\,. 
\end{equation}
Using Cauchy-Schwarz inequality we estimate
$$\left|\int_{\Om\cap B^+_{2\delta}(b_1)} \widehat\Phi\cdot(\nabla\chi_\delta\cdot\nabla u_*)\,\de x \right|\leq C\|\Phi\|_{L^\infty(\Om)}
\left(\int_{\Om\cap B^+_{2\delta}(b_1)}|\nabla u_*|^2\,\de x\right)^{1/2}\mathop{\longrightarrow}\limits_{\delta\downarrow0}0\,.$$ 
Therefore, letting $\delta\downarrow0$ in \eqref{tim1503} leads to 
$$\int_{\Om}\nabla u_*\cdot\nabla\widehat\Phi\,\de x =0\,.
$$
We then repeat the argument for each point $b_l$ to reach the conclusion 
\begin{equation}\label{heur1612}
\int_{\Om}\nabla u_*\cdot\nabla\Phi\,\de x =0 \,.
\end{equation}
Hence $u_*$ is a weak $\mathbb{S}^{m-1}$-boundary harmonic map in $\Om\cup\partial^0\Om$, and $u_*\in  C^\infty(\Om\cup\partial^0\Om)$ by Theorem~\ref{thmregfreebd}. 
\vskip3pt

\noindent{\it Case 2, $n\geq 2$.} Consider an arbitrary $\Phi$ as in Case 1, 
and write $K:={\rm supp}\,\Phi$. From Step 1 we know 
that $\mathscr{H}^{n-1}(\Sigma\cap K)<\infty$. By \cite[Theorem 3, p.154]{EvGa} it implies that ${\rm cap}_2(\Sigma\cap K)=0$ where ${\rm cap}_2$ denotes the Newtonian capacity. Moreover, the proof of \cite[Theorem 3, p.154]{EvGa} provides a sequence of functions $\{\chi_l\}_{l\in\N}$ such 
that $\chi_l\in \dot H^1(\R^{n+1})$, $0\leq \chi_l\leq 1$, $\chi_l=0$ in a neighborhood of $\Sigma\cap K$, $\chi_l\to 1$ a.e. as $l\to\infty$, and 
$$\lim_{l\to\infty}\int_{\R^{n+1}}|\nabla\chi_l|^2\,\de x =0\,. $$
Arguing as in Case 1, we obtain that
\begin{equation}\label{her1614}
0=\int_{\Om}\nabla u_*\cdot\nabla(\chi_l\Phi)\,\de x= 
\int_{\Om}\chi_l\nabla u_*\cdot\nabla\Phi\,dx + \int_{\Om} \Phi\cdot(\nabla\chi_l\cdot\nabla u_*)\,\de x\,,
\end{equation}
and we estimate 
$$ \left|\int_{\Om} \Phi\cdot(\nabla\chi_l\cdot\nabla u_*)\,\de x \right|\leq C\|\Phi\|_{L^\infty(\Om)}\|\nabla u_*\|_{L^2(\Om)} 
\left(\int_{\R^{n+1}}|\nabla \chi_l|^2\,\de x\right)^{1/2}\mathop{\longrightarrow}\limits_{l\to\infty}0\,.$$
Letting $l\to \infty$ in \eqref{her1614} then shows that \eqref{heur1612} holds, whence $u_*$ is a weak $\mathbb{S}^{m-1}$-boundary harmonic map in $\Omega\cup\partial^0\Om$.
\vskip5pt

\noindent{\it Step 4.} We conclude the proof by showing that $\Sigma={\rm supp}(\mu_{\rm sing})\cup {\rm sing}(u_*)$. 
If $x_0$ does not belong to ${\rm supp}(\mu_{\rm sing})\cup {\rm sing}(u_*)$, we can find $r_0>0$ such that 
$\mu_{\rm sing}(B_r(x_0))=0$ for all $r<r_0$. Hence, 
$$r^{1-n}\mu(B_r(x_0))=\lim_{n\to\infty} r^{1-n}\mu_k(B_r(x_0))= \frac{r^{1-n}}{2}\int_{B_r(x_0)\cap\Om}|\nabla u_*|^2\,\de x$$
for all $r<r_0$. Since $u_*$ is smooth in a neighborhood of $x_0$ by Theorem \ref{thmregfreebd}, we deduce that $\Theta^{n-1}(\mu,x_0)=0$, 
and thus $x_0\not\in\Sigma$ by \eqref{tim1721}. 

Let us now assume that $x_0\not \in\Sigma$. If $x_0\in\Om$ then $x_0\not\in {\rm supp}(\mu_{\rm sing})\cup {\rm sing}(u_*)$ 
by \eqref{tim1751} and the smoothness of $u_*$ in $\Om$. If $x_0\in\partial^0\Om$ we deduce from \eqref{tim1755} and the convergence established in Step~2 
that $x_0\not\in {\rm sing}(u_*)$, and 
$$\mu(B_r(x_0))=\lim_{k\to\infty} \mu_k(B_r(x_0))= \frac{1}{2}\int_{B_r(x_0)\cap\Om}|\nabla u_*|^2\,\de x$$
for a radius $r>0$ sufficiently small. Therefore, $\mu_{\rm sing}(B_r(x_0))=0$, and thus $x_0$ does not belong to ${\rm supp}(\mu_{\rm sing})$. 
\vskip5pt

\noindent{\it Step 5.} In view of \eqref{tim2253}, it only remains to prove (ii). 
Let $K$ be a compact subset of $\partial^0\Om$, and set $\delta_0:=\frac{1}{2}{\rm dist}(K,\partial^+\Om)$. For $\delta\in(0,\delta_0/2)$, we define  
 $\Sigma_\delta:=\{x\in\partial^0\Om: {\rm dist}(x,\Sigma)\leq\delta\}$. Then, $K\cap\Sigma_\delta$ is a compact set.  From Step 4 and the local 
 smooth convergence of $u_k$ toward $u_*$ outside~$\Sigma$, we deduce  Êthat
 $\eps_k^{-1}\int_{K\setminus\Sigma_\delta} (1-|u_k|^2)^2\,\de \mathscr{H}^n\to 0$.  
 On the other hand, for any $x_0\in K\cap\Sigma_\delta$ we have $B^+_\delta(x_0)\subset\Om$,  and we infer from the monotonicity formula in Lemma \ref{monotform} that 
$$ \frac{1}{4\eps_k}\int_{\delta}^{2\delta}\frac{1}{t^n}\int_{D_t(x_0)} (1-|u_k|^2)^2\,\de \mathscr{H}^n\de t\leq \frac{1}{\delta_0^{n-1}} E_{\eps_k}(u_k,B^+_{\delta_0}(x_0))\,.$$
Hence,
$$\limsup_{k\to\infty} \frac{1}{4\eps_k}\int_{D_\delta(x_0)} (1-|u_k|^2)^2\,\de \mathscr{H}^n\leq 
\frac{2^n\delta^{n}}{\delta_0^{k-1}} \sup_{n\in\N} E_{\eps_k}(u_k,\Om)\,.$$
Then, by a standard covering argument, we deduce that
$$ \lim_{k\to\infty} \frac{1}{\eps_k}\int_{\Sigma_\delta} (1-|u_k|^2)^2\,\de \mathscr{H}^n =0\,,$$
and thus $\eps_k^{-1}\int_{K} (1-|u_k|^2)^2\,\de \mathscr{H}^n\to 0$ as $k\to\infty$.  
\end{proof}

To complete this subsection, we now prove the $(n-1)$-rectifiability of the defect measure $\mu_{\rm sing}$ through the celebrated {\sc Preiss} criteria \cite{Pr}. 
 
\begin{proposition}\label{rectimeas} 
Assume that $n\geq 2$. In Theorem~\ref{asymptneum} the set $\Sigma$  is countably $\mathscr{H}^{n-1}$-rectifiable, and the defect measure $\mu_{\rm sing}$ satisfies
\begin{equation}\label{tim2258}
\mu_{\rm sing}= \theta  \mathscr{H}^{n-1}\LL\Sigma 
\end{equation}
for some positive Borel function $\theta: \Sigma\to (0,\infty)$. 
\end{proposition}
 
 \begin{proof}
 By a well known property of Sobolev functions (see {\it e.g.} \cite[(3.3.28)]{Zi}), we have 
\begin{equation}\label{goodptu*}
\lim_{r\downarrow 0}\frac{1}{r^{n-1}}\int_{\Om\cap B_r^+(x)}|\nabla u_*|^2\,\de x=0 \quad\text{for $\mathscr{H}^{n-1}$-a.e. $x\in\Sigma$}\,.
\end{equation}
 Therefore  \eqref{existdens} yields 
 \begin{equation}\label{tim2221}
 \Theta^{n-1}(\mu_{\rm sing},x):=\lim_{r\downarrow 0}\frac{\mu_{\rm sing}(B_r(x))}{{\boldsymbol\omega}_{n-1}r^{n-1}}= \Theta^{n-1}(\mu,x) \quad\text{for $\mathscr{H}^{n-1}$-a.e. $x\in\Sigma$}\,.
 \end{equation}
 On the other hand, we derive from \eqref{upbddensity} that 
 $$ \Theta^{*,n-1}(\mu_{\rm sing},x):=\limsup_{r\downarrow 0}\frac{\mu_{\rm sing}(B_r(x))}{{\boldsymbol\omega}_{n-1}r^{n-1}}\leq   \frac{C}{\big({\rm dist}(x,\partial^+\Om)\big)^{n-1}} \,\sup_{k\in\N} E_{\eps_k}(u_k,\Om)<\infty$$
for all $x\in\Sigma$. Since ${\rm supp}(\mu_{\rm sing})\subset \Sigma$, 
we infer from \cite[Theorem~2.56]{AFP} that $\mu_{\rm sing}$ is absolutely continuous with respect to $\mathscr{H}^{n-1}\LL\Sigma$. Then we deduce from \eqref{tim2221} and  \eqref{existdens}-\eqref{upbddensity}-\eqref {tim1721} that 
$$ \Theta^{n-1}(\mu_{\rm sing},\cdot) = \Theta^{n-1}(\mu,\cdot) \in (0,\infty)  \quad\text{$\mu_{\rm sing}$-a.e.}\,,$$
and according to {\sc Preiss} rectifiability criteria \cite{Pr}, it implies that $\mu_{\rm sing}$ is a $(n-1)$-rectifiable measure (see {\it e.g.} \cite[Definition 2.59]{AFP}). 
Next we infer from \cite[Theorem~2.83]{AFP}, \eqref{tim2221}, and \eqref{upbddensity}-\eqref {tim1721} that $\mu_{\rm sing}$ is of the form \eqref{tim2258} with $\theta(x):= \Theta^{n-1}(\mu,x)\in(0,\infty)$. As a consequence, $\Sigma$ is a countably $\mathscr{H}^{n-1}$-rectifiable set. 
 \end{proof}

 \begin{remark}[\bf Approximate tangent space]\label{remapproxplane}
As a consequence of Proposition~\ref{rectimeas} and the fact that $\Sigma\subset \partial^0\Om\subset\R^n$, the measure $\mu_{\rm sing}$ admits an $(n-1)$-dimensional approximate tangent  space $T_x\Sigma\subset \R^n$  at $x$ with multiplicity $\theta(x)$ for $\mathscr{H}^{n-1}$-a.e. $x\in\Sigma$ (see {\it e.g.} \cite[Theorem~2.83]{AFP}). More precisely, if we denote by ${\mathbf G}_{n,n-1}$ the Grassmann manifold of unoriented  $(n-1)$-dimensional planes in $\R^{n}$, then for $\mathscr{H}^{n-1}$-a.e. $x\in\Sigma$ there exists  a plane $T_x\Sigma\in {\mathbf G}_{n,n-1}$ such that 
$$\lim_{r\downarrow 0} \frac{1}{r^{n-1}} \int_{\partial^0\Om}\phi\left(\frac{y-x}{r}\right)\,\de\mu_{\rm sing}(y)=\theta(x)\int_{T_x\Sigma}\phi(y)\,\de\mathscr{H}^{n-1}(y)$$
for all $\phi\in C^0_c(\R^n)$. 
\end{remark}


 \subsection{Stationarity defect and generalized varifolds}
 
In this subsection we assume that $n\geq2$, and our discussion starts from Theorem~\ref{asymptneum}. In this theorem, we point out that the limiting 
$(\mathbb{S}^{m-1},\partial^0\Omega)$-boundary harmonic map $u_*$ is {\it a priori} only weakly harmonic, and might not be stationary, {\it i.e.}, it might not satisfy~\eqref{statcondVectF}. On the other hand, $u_*$ arises as a weak limit of (smooth) critical points of the boundary Ginzburg-Landau energy $E_\eps$, so that the lack  of stationarity of $u_*$ should be quantified. For the classical Ginzburg-Landau system, the analoguous question is treated by {\sc Lin \& Wang} in \cite{LW3}   where they show that the  possible stationarity defect is related to the defect measure through an explicit formula. The main objective in  this subsection is to prove that a similar formula holds in the Ginzburg-Landau boundary context, and this is the object of the following theorem. 
 
 \begin{theorem}\label{statdef}
Assume that $n\geq 2$.  In Theorem~\ref{asymptneum},   the limiting  $(\mathbb{S}^{m-1},\partial^0\Omega)$-boundary harmonic map $u_*$ 
and the defect measure $\mu_{\rm sing}$ represented in \eqref{tim2258} satisfy
$$\int_{\Om}\bigg( |\nabla u_*|^2{\rm div}\mathbf{X} -2\sum_{i,j=1}^{n+1} (\partial_i u_*\cdot\partial_j u_*)\partial_j\mathbf{X}_i\bigg)\,\de x
+\int_\Sigma \theta\, {\rm div}_{\Sigma}\mathbf{X}\,\de\mathscr{H}^{n-1}=0$$
for all vector fields $\mathbf{X}=(\mathbf{X}_1,\ldots,\mathbf{X}_{n+1})\in C^1(\overline\Om;\R^{n+1})$ compactly supported in $\Om\cup\partial^0\Om$ and satisfying $\mathbf{X}_{n+1}=0$ on $\partial^0\Om$. 
 \end{theorem}

In the spirit of  \cite{LW3}, the proof of Theorem~\ref{statdef} relies on the notion of {\it generalized varifold} introduced by {\sc Ambrosio-Soner} \cite{AS}. To simplify the proof of Theorem~\ref{statdef}, we shall assume that the admissible open 
set $\Omega\subset \R^{n+1}_+$ satisfies $\Omega=\widetilde \Omega\cap\R^{n+1}_+$ for 
 some Lipschitz bounded open set $\widetilde \Omega\subset\R^{n+1}$ which is symmetric with respect to the hyperplane $\R^n=\{x_{n+1}=0\}$. Since all the arguments are local in nature, the general case can be handled in a similar way with minor modifications. 
 According to this symmetry assumption, we introduce some notations. For $B\subset \R^{n+1}$, we write $B^+:=B\cap\R^{n+1}_+$ and $B^-:=B\cap\R^{n+1}_-$, where $\R^{n+1}_-:=\R^n\times(-\infty,0)$. For a map $u\in H^1(\Om;\R^m)$, we denote by $\tilde u \in H^1(\tom;\R^m)$ the extension of $u$ to $\tom$ obtained by even reflection across $\partial^0\Om$, {\it i.e.}, 
 $$\tilde u(x',x_{n+1}):=\begin{cases}
 u(x',x_{n+1}) & \text{for $x=(x',x_{n+1})\in\Om$}\,,\\
 u(x',-x_{n+1}) & \text{for $x=(x',x_{n+1})\in\Om^-:=\widetilde\Om^-$}\,,
 \end{cases}$$
and we write $u^-:=\tilde u_{|\Om^-}$.  
\vskip3pt
 
To recall from \cite{AS} the concept of generalized varifolds, we need to introduce the following (compact and convex) set of matrices
$$\mathbf{A}_{n-1}:=\big\{A\in\R^{(n+1)\times(n+1)} : \text{$A$ is symmetric, ${\rm trace}(A)=n-1\,$,}\; -(n+1)I_{n+1}\leq A\leq I_{n+1}\,\big\} \,,$$
where $I_{n+1}$ denotes the identity square matrix of size $n+1$.

\begin{definition}
A $(n-1)$-dimensional generalized varifold $V$ on $\tom$ is a nonnegative Radon measure on $\tom\times\mathbf{A}_{n-1}$. 
The class of all generalized $(n-1)$-varifolds on $\tom$ is denoted by $\mathbf{V}^*_{n-1}(\tom)$. For $V\in\V$ we denote by $\|V\|$ the weight of $V$ defined as the first marginal of $V$, {\it i.e.},  $\|V\|:=\pi_\sharp V$ where $\pi:\tom\times\mathbf{A}_{n-1}\to\tom $ is the canonical projection. (Notice that $\|V\|$ is a Radon measure on $\tom$.)  The first variation $\delta V$ of a generalized varifold $V\in\V$ is the element of $( C^1_c(\tom;\R^{n+1}))^*$ defined by 
$$\langle\delta V, \mathbf{X}\rangle:=-\int_{\tom\times\mathbf{A}_{n-1} } A : \nabla \mathbf{X}\,\de V \quad\text{for all $\mathbf{X}\in C^1_c(\tom;\R^{n+1})$}\,.$$
If $\delta V=0$, then $V$ is said to be {\it stationary}. 
\end{definition}

\begin{remark}[\bf  Weak convergence of varifolds]
The convergence in $\V$ is understood as weak* convergence of Radon measures on $\tom\times\mathbf{A}_{n-1}$. In particular, if 
$V_k\rightharpoonup V$ in $\V$, then $\delta V_k\rightharpoonup \delta V$ weakly* in $( C^1_c(\tom;\R^{n+1}))^*$.
\end{remark}

\begin{remark}[\bf Disintegration and barycenter]
Given  $V\in\V$, we denote by $\{V_x\}_{x\in \tom}$ a {\it disintegration} of $V$, {\it i.e.}, $\{V_x\}_{x\in \tom}$ is a family of probability  measures on $\mathbf{A}_{n-1}$ such that $x\mapsto V_x$ is $\|V\|$-measurable, and 
\begin{equation}\label{tim1644}
\int_{\tom\times\mathbf{A}_{n-1}} f(x,A) \,\de V = \int_{\tom}\left(\int_{\mathbf{A}_{n-1}}f(x,A)\,\de V_x\right)\,\de\|V\|
\end{equation}
for any bounded Borel function $f:\tom\times \mathbf{A}_{n-1}\to\R$. The measurability condition on $\{V_x\}_{x\in \tom}$ means that $x\mapsto V_x(B)$ is 
$\|V\|$-measurable for every Borel set $B\subset \mathbf{A}_{n-1}$. This fact ensures that the inner integral in the right hand side of \eqref{tim1644} is $\|V\|$-measurable, so that its integral is well defined. We refer to \cite{AFP} for the existence and the uniqueness of $\{V_x\}_{x\in \tom}$ modulo $\|V\|$-null sets.  Throughout the subsection, we may use the disintegrated notation  $V=V_x\|V\|$. 
We also denote by $\bar A_V(x)$ the barycenter of $V$ at $x$ defined by 
$$\bar A_V(x):= \int_{\mathbf{A}_{n-1}}A\,\de V_x\,.$$ 
Then $x\mapsto \bar A_V(x)$ is $\|V\|$-measurable, and  $\bar A_V(x)\in \mathbf{A}_{n-1}$ since $V_x$ is a probability measure. Moreover, we can rewrite the action of the first variation $\delta V$ as 
$$\langle\delta V, X\rangle:=-\int_{\tom} \bar A_V(x) :\nabla \mathbf{X}\,\de \|V\|$$
for all $\mathbf{X}\in C^1_c(\tom;\R^{n+1})$. 
\end{remark}

We may now present the way to relate our problem to generalized varifolds. We start with the construction of a generalized varifold starting from a Sobolev map. 
For  $u\in H^1(\tom;\R^m)$, we denote by $V_u\in \V$ the generalized varifold given by
$$V_u:=\frac{1}{2}\delta_{A_u}|\nabla u|^2\mathscr{L}^{n+1}\LL\tom\,,$$
 where $\delta_{A_u(x)}$ is the Dirac mass concentrated at the matrix $A_u(x)$ defined by  
 $$A_u(x):=\begin{cases} 
 \displaystyle I_{n+1}-2\frac{(\nabla u(x))^T(\nabla u(x))}{|\nabla u(x)|^2} & \text{if $|\nabla u(x)|\not=0$}\,,\\[8pt]
 I_{n-1} & \text{otherwise}\,,
 \end{cases}$$
 and $I_{n-1}$ is the matrix of the orthogonal projection on $\R^{n-1}\simeq\R^{n-1}\times\{(0,0)\}$. 
 One may easily check that $A_u(x)\in \mathbf{A}_{n-1}$, so that $V_u$ indeed belongs to $\V$.  
 Next we add the Ginzburg-Landau potential in the construction above to build generalized varifolds based on the Ginzburg-Landau boundary energy. More precisely, for $\eps>0$ and an arbitrary map $u\in H^1(\tom;\R^m)\cap L^4(\partial^0\Om)$, we set 
 $$V^\eps_u:=V_u + \frac{1}{2\eps}\delta_{I_{n-1}}(1-|u|^2)^2\mathscr{H}^n\LL\partial^0\Om\in\V \,,$$
so that 
$$\|V_{u}^\eps\|(B)=\frac{1}{2}\int_{B} |\nabla u|^2\,\de x+\frac{1}{2\eps}\int_{B\cap\partial^0\Om}(1-|u|^2)^2\,\de \mathscr{H}^n$$
for any open subset $B\subset\tom$.  The first variation of $V^\eps_u$ is then given by
 $$\langle\delta V_u^\eps,\mathbf{X}\rangle=-\frac{1}{2}\int_{\tom}\bigg( |\nabla u|^2{\rm div}\mathbf{X} -2\sum_{i,j=1}^{n+1}( \partial_i u\cdot\partial_j u)\partial_j\mathbf{X}_i\bigg)\,\de x
 -\frac{1}{2\eps}\int_{\partial^0\Omega}(1-|u|^2)^2{\rm div}_{\R^{n-1}}\mathbf{X}\,\de \mathscr{H}^n$$
 for all $\mathbf{X}\in C^1_c(\tom;\R^{n+1})$. In the case where $u=\tilde u_\eps$ and $u_\eps$ is a critical point of the Ginzburg-Landau boundary energy $E_\eps$, the first variation reduces to the following expression. 
 
 \begin{lemma}\label{firstvar}
 Given $\eps>0$, if $u_\eps\in H^1(\Om;\R^m)\cap L^\infty(\Om)$ is a critical point of $E_\eps$ in $\Om$, then 
 $$\langle\delta V^\eps_{\tilde u_\eps},\mathbf{X}\rangle=\frac{1}{2\eps}\int_{\partial^0\Om}(1-|u_\eps|^2)^2\partial_n\mathbf{X}_n\,\de \mathscr{H}^n$$
 for all  $\mathbf{X}\in C^1_c(\tom;\R^{n+1})$. 
 \end{lemma}
 
 \begin{proof}
 If $u_\eps\in H^1(\Om;\R^m)\cap L^\infty(\Om)$ is a critical point of $E_\eps$ in $\Om$, then $u_\eps\in C^\infty(\Om\cup\partial^0\Om)$ by  Theorem~\ref{regint}. It obviously implies   
 $u^-_\eps\in C^\infty(\Om^-\cup\partial^0\Om)$,  and 
\begin{equation}\label{tim2329}
 \begin{cases}
 \Delta u^-_\eps = 0 &\text{in $\Om^-$}\,,\\[8pt]
 \displaystyle \frac{\partial u^-_\eps}{\partial \nu}= \frac{1}{\eps}(1-|u_\eps|^2)u_\eps & \text{on $\partial^0\Om$}\,.
 \end{cases}
\end{equation}
Let us now consider a vector field $\mathbf{X}=(\mathbf{X}_1,\ldots,\mathbf{X}_{n+1})\in C^1_c(\tom;\R^{n+1})$. Arguing as in the proof of Lemma \ref{monotform}, we integrate 
by parts to find
\begin{multline*}
\int_{\Om}\bigg( |\nabla u_\eps|^2{\rm div}\mathbf{X} -2\sum_{i,j=1}^{n+1} (\partial_i u_\eps\cdot\partial_j u_\eps)\partial_j\mathbf{X}_i\bigg)\,\de x \\
= -\frac{1}{2\eps}\int_{\partial^0\Om}(1-|u_\eps|^2)^2{\rm div}_{\R^n}\mathbf{X}\,\de\mathscr{H}^n 
-\int_{\partial^0\Om} |\nabla u_\eps|^2\mathbf{X}_{n+1}\,\de \mathscr{H}^n\\
+\frac{2}{\eps^2}\int_{\partial^0\Om}(1-|u_\eps|^2)^2|u_\eps|^2\mathbf{X}_{n+1}\,\de \mathscr{H}^n\,.
\end{multline*}
Similarly, \eqref{tim2329} yields 
\begin{multline*}
\int_{\Om^-}\bigg( |\nabla u^-_\eps|^2{\rm div}\mathbf{X} -2\sum_{i,j=1}^{n+1}(\partial_i u^-_\eps\cdot\partial_j u^-_\eps)\partial_j\mathbf{X}_i \bigg)\,\de x \\
= -\frac{1}{2\eps}\int_{\partial^0\Om}(1-|u_\eps|^2)^2{\rm div}_{\R^n}\mathbf{X}\,\de \mathscr{H}^n
+\int_{\partial^0\Om} |\nabla u^-_\eps|^2\mathbf{X}_{n+1}\,\de \mathscr{H}^n\\
-\frac{2}{\eps^2}\int_{\partial^0\Om}(1-|u_\eps|^2)^2|u_\eps|^2\mathbf{X}_{n+1}\,\de \mathscr{H}^n\,.
\end{multline*}
Therefore,
$$\langle\delta V_{\tilde u_\eps}^\eps,\mathbf{X}\rangle=\frac{1}{2\eps}\int_{\partial^0\Om}(1-|u_\eps|^2)^2\partial_n\mathbf{X}_n\,\de \mathscr{H}^n
+\frac{1}{2}\int_{\partial^0\Om}\big(|\partial_{n+1} u_\eps|^2-|\partial_{n+1} u^-_\eps|^2\big)\mathbf{X}_{n+1}\,\de \mathscr{H}^n\,.$$
Since $\partial_{n+1}u^-_\eps=-\partial_{n+1}u_\eps$ by \eqref{tim2329},  the conclusion follows from this last equality.   
 \end{proof}
 
 We now relate a weak limit  $u_*$ and its defect measure $\mu_{\rm sing}$ to the weak limit of the generalized varifolds  $V^\eps_{\tilde u_\eps}$, which has to be stationary by Lemma \ref{firstvar} and the vanishing property of the Ginzburg-Landau potential as $\eps\to0$. 
 
 \begin{corollary}\label{statlimit}
 Let $\{u_k\}_{k\in\N}$ be the subsequence given by Theorem \ref{asymptneum}, and let $\mu_{\rm sing}$ be the defect measure represented in \eqref{tim2258}. Up to a further subsequence (not relabeled), $V^{\eps_k}_{\tilde u_{k}}\rightharpoonup V_*$ for some stationary $V_*\in\V$. In addition, 
 \begin{equation}\label{decompvar}
 V_*=V_{\tilde u_*}+V_{\rm sing}\,, 
 \end{equation}
where $V_{\rm sing}\in \V$ is supported by $\Sigma\times\mathbf{A}_{n-1}$, and $\|V_{\rm sing}\|=\mu_{\rm sing}$.  
\end{corollary}
 
 \begin{proof}
 {\it Step 1.} By symmetry with respect to $\{x_{n+1}=0\}$, we have $\|V^{\eps_k}_{\tilde u_{k}}\|(\tom) = 2 E_{\eps_k}(u_k,\Om)$, and thus  
 $$\sup_{k\in\N} \|V^{\eps_k}_{\tilde u_{k}}\|(\tom)<\infty\,.$$
 Hence we can find a (not relabeled) further subsequence  such that $V^{\eps_k}_{\tilde u_k}\rightharpoonup V_*$ for some   $V_*\in\V$. Then $\delta V^{\eps_k}_{\tilde u_k}\rightharpoonup\delta V_*$ as distributions. We claim that $\delta V_*=0$. Indeed, for an arbitrary $\mathbf{X}\in C^1_c(\tom;\R^{n+1})$, we deduce from item (ii) in  Theorem \ref{asymptneum} and 
Lemma \ref{firstvar} that 
$$\langle \delta V_*, \mathbf{X}\rangle=\lim_{k\to\infty} \frac{1}{2\eps_k}\int_{\partial^0\Om}(1-|u_k|^2)^2\partial_n\mathbf{X}_n\,\de \mathscr{H}^n=0\,.$$
\vskip3pt

\noindent{\it Step 2.} For  $\Phi\in C^0_c(\tom)$, we denote by $\widehat \Phi$ the reflection of $\Phi$ with respect to 
the hyperplane $\{x_{n+1}=0\}$, {\it i.e.}, $\widehat \Phi(x',x_{n+1}):=\Phi(x',-x_{n+1})$ for $x\in\tom$. Noticing that 
$$\int_{\tom} \Phi\,\de \|V^{\eps_k}_{\tilde u_{k}}\|=\frac{1}{2}\int_\Om |\nabla u_k|^2(\Phi+\widehat\Phi)\,\de x+\frac{1}{2\eps_k}\int_{\partial^0\Om}(1-|u_k|^2)^2\Phi\,\de \mathscr{H}^n\,, $$
we infer from Theorem~\ref{asymptneum}  that 
$$\lim_{k\to\infty}\int_{\tom} \Phi\,\de \|V^{\eps_k}_{\tilde u_{k}}\| = \frac{1}{2}\int_\Om |\nabla u_*|^2(\Phi+\widehat\Phi)\,\de x
+\int_{\Sigma}\Phi\,\de \mu_{\rm sing} 
=\frac{1}{2}\int_{\tom} |\nabla \tilde u_*|^2\Phi\,\de x
+\int_{\Sigma}\Phi\,\de \mu_{\rm sing}\,.$$
On the other hand, $\|V^{\eps_k}_{\tilde u_{k}}\|\rightharpoonup \|V_*\|$ weakly* as Radon measures on $\tom$, and hence 
\begin{equation}\label{tim1836}
\|V_*\|= \|V_{\tilde u_*}\| +\mu_{\rm sing}\,. 
\end{equation}
\vskip3pt

Next consider  $\Psi\in C^0_c(\tom\times \mathbf{A}_{n-1})$ with compact support in $(\tom\setminus\Sigma)\times \mathbf{A}_{n-1}$. Then, 
$$\int_{\tom\times \mathbf{A}_{n-1}} \Psi\,\de V^{\eps_k}_{\tilde u_k}=\frac{1}{2}\int_{\tom} \Psi\big(x,A_{\tilde u_k}(x)\big)|\nabla \tilde u_k|^2\,\de x 
+\frac{1}{4\eps_k}\int_{\partial^0\Om}\Psi(x,I_{n-1})(1-|u_k|^2)^2\,\de \mathscr{H}^n\,. $$
From the convergences established in items (ii) \& (iv) of Theorem~\ref{asymptneum}, we deduce that 
\begin{multline*}
\int_{\tom\times \mathbf{A}_{n-1}} \Psi\,\de V_*=\lim_{k\to\infty} \int_{\tom\times \mathbf{A}_{n-1}} \Psi\,\de V^{\eps_k}_{\tilde u_k} \\ 
= \lim_{k\to\infty} \frac{1}{2}\int_{\tom\cap\{|\nabla\tilde u_*|\not=0\}} \Psi\big(x,A_{\tilde u_k}(x)\big)|\nabla \tilde u_k|^2\,\de x\\
 = \frac{1}{2}\int_{\tom} \Psi\big(x,A_{\tilde u_*}(x)\big)|\nabla \tilde u_*|^2\,\de x
 = \int_{\tom\times \mathbf{A}_{n-1}} \Psi\,\de V_{\tilde u_*}\,.
\end{multline*}
As a consequence, $V_*(\mathcal{K})=V_{\tilde u_*}(\mathcal{K})$ for every  compact set $\mathcal{K}\subset (\tom\setminus\Sigma)\times \mathbf{A}_{n-1}$.
By inner regularity, it implies 
$$V_*\big(\mathcal{B}\setminus(\Sigma\times\mathbf{A}_{n-1})\big)=V_{\tilde u_*}\big(\mathcal{B}\setminus(\Sigma\times\mathbf{A}_{n-1})\big)=
V_{\tilde u_*}(\mathcal{B})  $$
for all Borel sets $\mathcal{B}\subset \tom\times \mathbf{A}_{n-1}$ (in the last equality, we have used  the fact that $\|V_{\tilde u_*}\|$ is absolutely continuous with respect to $\mathscr{L}^{n+1}$). Therefore, 
$$V_*\LL\big((\tom\setminus\Sigma)\times \mathbf{A}_{n-1}\big)= V_{\tilde u_*}\,. $$
Finally, setting
$$V_{\rm sing}:= V_*\LL(\Sigma\times \mathbf{A}_{n-1}) \,,$$
we clearly have \eqref{decompvar}, and $\|V_{\rm sing}\|=\mu_{\rm sing}$ holds by \eqref{tim1836}. 
\end{proof}

The last main step to establish Theorem \ref{statdef} is the following geometrical property on the singular part of the limiting varifold.

 \begin{lemma}\label{matrproj}
 In Corollary \ref{statlimit}, the barycenter of generalized varifold $V_{\rm sing}$ satisfies
$$\bar A_{V_{\rm sing}}(x)=A_{T_x\Sigma} \quad \text{for $\mathscr{H}^{n-1}$-a.e. $x\in\Sigma$}\,,$$
 where $A_{T_x\Sigma}$ is the matrix of the orthogonal projection on the approximate tangent plane $T_x\Sigma$  of $\Sigma$ at  $x$ according to Remark~\ref{remapproxplane}. 
 \end{lemma}
 
 \begin{proof}
 Let $x_0\in \Sigma$ be such that $\Sigma$ admits an approximate tangent plane $T_{x_0}\Sigma$ at $x_0$ with multiplicity $\theta(x_0)$, $x_0$ is a Lebesgue point of  $A_{V_{\rm sing}}$ with respect to $\mu_{\rm sing}$, and \eqref{goodptu*} holds at~$x_0$. Those properties are satisfied $\mathscr{H}^{n-1}$-a.e. on $\Sigma$. 
 
Next consider an arbitrary $\mathbf{X}\in  C^1_c(\R^{n+1};\R^{n+1})$, and set $\mathbf{X}_r(x):=r\mathbf{X}(\frac{x-x_0}{r})$. From Corollary~\ref{statlimit} and the choice of $x_0$, we infer that 
\begin{align*}
0= \lim_{r\downarrow0}\,\frac{1}{r^{n-1}}\langle \delta V_*,\mathbf{X}_r\rangle & =  \lim_{r\downarrow0}\,\frac{1}{r^{n-1}}\langle \delta V_{\rm sing},\mathbf{X}_r\rangle \\
 &= \lim_{r\downarrow0}\,\frac{1}{r^{n-1}}\int_{\partial^0\Om} \bar A_{V_{\rm sing}}(x) : \nabla \mathbf{X}\left(\frac{x-x_0}{r}\right)\,\de \mu_{\rm sing}\\
 &= \theta(x_0)\int_{T_{x_0}\Sigma} \bar  A_{V_{\rm sing}}(x_0) : \nabla \mathbf{X} \,\de\mathscr{H}^{n-1}\,.
 \end{align*}
 Since $\theta(x_0)>0$ and $\mathbf{X}$ is arbitrary,  we deduce that 
 \begin{equation}\label{tim2203}
 \bar A_{V_{\rm sing}}(x_0)  \int_{T_{x_0}\Sigma} \nabla\phi\,\de \mathscr{H}^{n-1} =0\quad\text{for all $\phi\in C^1_c(\R^{n+1})$}\,.
 \end{equation}
  We claim that 
\begin{equation}\label{claimker}
\mathcal{F}:=\left\{ \int_{T_{x_0}\Sigma} \nabla\phi\,\de \mathscr{H}^{n-1} :  \phi\in C^1_c(\R^{n+1})  \right\}= {\rm Ker} \,A_{T_{x_0}\Sigma} \,.
\end{equation}
 Before proving  \eqref{claimker} we complete the argument. From the last two identities we derive that at least two eigenvalues of $\bar A_{V_{\rm sing}}(x_0)$ vanish. On the other hand,  $\bar A_{V_{\rm sing}}(x_0)\in\mathbf{A}_{n-1}$ so that  ${\rm trace}(\bar A_{V_{\rm sing}}(x_0))=n-1$ and $\bar A_{V_{\rm sing}}(x_0)\leq I_{n+1}$.  
It  implies that the remaining eigenvalues are actually equal to one. Hence $\bar A_{V_{\rm sing}}(x_0)$ is a matrix of an orthogonal projection over an $(n-1)$-dimensional plane. Now \eqref{tim2203} and \eqref{claimker} show that $\bar A_{V_{\rm sing}}(x_0)$ and $A_{T_{x_0}\Sigma}$ have the same kernel, whence $\bar A_{V_{\rm sing}}(x_0)=A_{T_{x_0}\Sigma}$. 
 
 To prove \eqref{claimker}, we argue as follows. First notice that we may assume without loss of generality that $T_{x_0}\Sigma=\R^{n-1}\times\{(0,0)\}$. Since the admissible  $\phi$'s are compactly supported, we obtain 
 $$ \int_{T_{x_0}\Sigma} \partial_j\phi\,\de \mathscr{H}^{n-1}=0 \quad\text{for all $j\in\{1,\ldots,n-1\}$}\,,$$
 and the inclusion $\mathcal{F}\subset {\rm Ker} \,A_{T_{x_0}\Sigma}$ follows. To prove the reverse inclusion it suffices to use admissible functions of the form $\phi(x)=\chi(x_n)\psi(x^{\prime\prime})$ or $\phi(x)=\chi(x_{n+1})\psi(x^{\prime\prime})$ where we write $x=(x^{\prime\prime},x_n,x_{n+1})$,  
 $\chi\in C^1_c(\R)$ satisfies $\chi^\prime(0)=1$, and $\psi\in C^1_c(\R^{n-1})$ is such that  $\int_{\R^{n-1}}\psi =1$. 
 \end{proof}
 
 \begin{remark}
In view of Proposition \ref{rectimeas} and Lemma \ref{matrproj}, the generalized varifold 
$$\bar V_{\rm sing}:= \delta_{\bar A_{V_{\rm sing}}}\,\mu_{\rm sing}$$
is actually a real $(n-1)$-rectifiable varifold in the classical sense, see \cite{Sim}. 
 \end{remark}
 
\begin{proof}[Proof of Theorem \ref{statdef}]
We first infer from Corollary \ref{statlimit}, Lemma \ref{matrproj}, and \eqref{tim2258} that 
\begin{multline}\label{preformthm}
\langle \delta V_*, \mathbf{X}\rangle  = -\frac{1}{2}\int_{\tom}\bigg( |\nabla \tilde u_*|^2{\rm div}\mathbf{X} -2\sum_{i,j=1}^{n+1} (\partial_i \tilde u_*\cdot\partial_j \tilde u_*)\partial_j\mathbf{X}_i\bigg)\,\de x
-\int_{\Sigma} A_{T_x\Sigma} : \nabla \mathbf{X}\,\de\mu_{\rm sing} \\
 = -\frac{1}{2}\int_{\tom}\bigg( |\nabla \tilde u_*|^2{\rm div}\mathbf{X} -2\sum_{i,j=1}^{n+1} (\partial_i \tilde u_*\cdot\partial_j \tilde u_*)\partial_j\mathbf{X}_i\bigg)\,\de x
-\int_{\Sigma} \theta\, {\rm div}_{\Sigma}\mathbf{X}\,\de\mathscr{H}^{n-1} = 0\,,
\end{multline}
for all $\mathbf{X}\in  C^1_c(\tom;\R^{n+1})$ (we recall that ${\rm div}_{\Sigma}\mathbf{X}(x):=A_{T_x\Sigma} : \nabla \mathbf{X}(x)$). 
\vskip3pt

Let us now consider an arbitrary vector field $\mathbf{X}=(\mathbf{X}_1,\ldots,\mathbf{X}_{n+1})\in C^1(\overline\Om;\R^{n+1})$ compactly supported in $\Om\cup\partial^0\Om$ and satisfying $\mathbf{X}_{n+1}=0$ on $\partial^0\Om$. We then extend $\mathbf{X}$ to $\tom$ by setting
$$\widehat{\mathbf{X}}(x):=(\tilde{\mathbf{X}}_1(x),\ldots,\tilde{\mathbf{X}}_n(x),-\tilde{\mathbf{X}}_{n+1}(x)) \,,$$
where we recall that $\tilde{\mathbf{X}}_j$ is the extension of $\mathbf{X}_j$ to $\tom$ obtained by even reflection accross $\partial^0\Omega$. 
Then  $\widehat{\mathbf{X}}$ is Lipschitz continuous and compactly supported in $\tom$. Notice also that $\partial_i\widehat{\mathbf{X}}_j$ is continuous in $\tom$ for all indices $i,j\in\{1,\ldots,n\}$. 
Then, by a standard mollification argument, we can find a sequence $\{\mathbf{X}^k\}_{k\in\N}\subset   C^1_c(\tom;\R^{n+1})$ such that 
$\mathbf{X}^k\to \widehat{\mathbf{X}}$ uniformly on $\tom$, $\partial_i\mathbf{X}_j^k\to \partial_i\widehat{\mathbf{X}}_j$ uniformly on $\tom$  for all $i,j\in\{1,\ldots,n\}$, $\nabla \mathbf{X}^k\to \nabla \widehat{\mathbf{X}}$ a.e. on $\tom$ with  $\|\nabla \mathbf{X}^k\|_{L^\infty(\tom)}\leq \|\nabla \widehat{\mathbf{X}}\|_{L^\infty(\tom)}$. Applying \eqref{preformthm} to each $\mathbf{X}^k$ and letting $k\to\infty$, we derive by dominated convergence that 
$$\frac{1}{2}\int_{\tom}\bigg( |\nabla \tilde u_*|^2{\rm div}\widehat{\mathbf{X}} -2\sum_{i,j=1}^{n+1} (\partial_i \tilde u_*\cdot\partial_j \tilde u_*)\partial_i\widehat{\mathbf{X}}_j\bigg)\,\de x
+\int_{\Sigma} \theta\, {\rm div}_{\Sigma}\widehat{\mathbf{X}}\,\de\mathscr{H}^{n-1} = 0 \,.$$
Finally, by symmetry of $\tilde u_*$ and $\widehat{\mathbf{X}}$, we have 
$$\frac{1}{2}\int_{\tom}\bigg( |\nabla \tilde u_*|^2{\rm div}\widehat{\mathbf{X}} -2\sum_{i,j=1}^{n+1} (\partial_i \tilde u_*\cdot\partial_j \tilde u_*)\partial_j\widehat{\mathbf{X}}_i\bigg)\,\de x\\ 
=\int_{\Om}\bigg( |\nabla u_*|^2{\rm div} \mathbf{X} -2\sum_{i,j=1}^{n+1} (\partial_i  u_*\cdot\partial_j  u_*)\partial_j \mathbf{X}_i\bigg)\,\de x\,,$$
and 
$$\int_{\Sigma} \theta\, {\rm div}_{\Sigma}\widehat{\mathbf{X}}\,\de\mathscr{H}^{n-1}=\int_{\Sigma} \theta\, {\rm div}_{\Sigma}\mathbf{X}\,\de\mathscr{H}^{n-1}\,,$$
and the conclusion follows. 
\end{proof}
 
  \vskip10pt

\section{Asymptotics for the fractional Ginzburg-Landau equation}\label{FGLasymp}   		

 The purpose of this section is to apply our previous results on the Ginzburg-Landau boundary equation to the asymptotic analysis, as $\eps\downarrow0$, of solutions of the fractional Ginzburg-Landau equation.  We start with the analogue of Theorem~\ref{asymptneum} and Theorem~\ref{statdef} in the fractional setting when no exterior condition is imposed. We then prove Theorem \ref{mainthm}. We conclude this section with the particular case of minimizers of the Ginzburg-Landau 1/2-energy under a Dirichlet exterior condition.

\subsection{Asymptotics without exterior Dirichlet  condition} 
We start in this subsection with a general case where no exterior condition is imposed. We obtain here the most important convergence results.

 \begin{theorem}\label{thmwithout}
Let $\omega\subset\R^{n}$ be a bounded open set with Lipschitz boundary. Let $\varepsilon_k\downarrow0$ be an arbitrary sequence,  and let  
$\{v_k\}_{k\in\mathbb{N}}\subset \widehat H^{1/2}(\omega;\mathbb{R}^m)\cap L^\infty(\R^n)$ be such that for every $k\in\N$, $|v_k|\leq 1$, and $v_k$ weakly solves  
\begin{equation}\label{eqfracfinal} 
 (-\Delta)^{\frac{1}{2}}  v_k=\frac{1}{\varepsilon_k}(1-|v_k|^2)v_k \quad \text{in $\omega$}\,. 
\end{equation}
If  $\sup_k \mathcal{E}_{\varepsilon_k}(v_k,\omega)<\infty$, then there exist a (not relabeled) subsequence and $v_*\in  \widehat H^{1/2}(\omega;\mathbb{R}^m)$ a bounded weak 1/2-harmonic map into $\mathbb{S}^{m-1}$ in $\omega$  such that $v_k\rightharpoonup v_*$ weakly in $H^{1/2}(\omega)\cap L^2_{\rm loc}(\R^n)$ as $k\to\infty$. 
In addition, there exist a  
nonnegative Radon measure $\mu_{\rm sing}$ on $\omega$, a countably $\mathscr{H}^{n-1}$-rectifiable relatively closed set $\Sigma\subset \omega$ 
of locally finite $(n-1)$-dimensional Hausdorff measure in $\omega$, and a Borel function $\theta:\Sigma\to(0,\infty)$ such that   
\vskip5pt
\begin{itemize}[leftmargin=22pt]
\item[\rm (i)] $\displaystyle|\nabla v^\e_k|^2\mathscr{L}^{n+1}\LL\R^{n+1}_+ \mathop{\rightharpoonup}\limits^*  
|\nabla v^\e_*|^2\mathscr{L}^{n+1}\LL\R^{n+1}_+ +\mu_{\rm sing} $  
locally weakly* as Radon measures on $\R^{n+1}_+\cup\omega$;   
\vskip5pt
\item[\rm (ii)] $\displaystyle \frac{(1-|v_k|^2)^2}{\varepsilon_k}\to 0$ in $L^1_{\rm loc}(\omega)$; 
\vskip5pt
\item[\rm (iii)] $\mu_{\rm sing}= \theta \mathscr{H}^{n-1}\LL\Sigma$; 
\vskip5pt
\item[\rm (iv)] $v_*\in  C^\infty(\omega\setminus\Sigma)$ and $v_k\to v_*$ in $ C^{\ell}_{\rm loc}(\omega\setminus\Sigma)$ for every $\ell\in\N$; 
\vskip5pt
\item[\rm (v)] if $n\geq 2$, the limiting 1/2-harmonic map $v_*$ and the defect measure $\mu_{\rm sing}$ satisfy 
$$\left[\frac{\de}{\de t} \,\mathcal{E}\big(v_*\circ\phi_t,\omega\big)\right]_{t=0} =\frac{1}{2} \int_\Sigma {\rm div}_{\Sigma}X\,\de \mu_{\rm sing}$$
for all vector fields $X\in C^1(\R^{n};\R^{n})$ compactly supported in $\omega$, where $\{\phi_t\}_{t\in\R}$ denotes the flow on $\R^n$ generated by $X$; 
\vskip5pt 

\item[\rm (vi)] if $n=1$, the set $\Sigma$ is locally finite in $\omega$ and $v_*\in  C^\infty(\omega)$. 
\end{itemize}
\end{theorem}

\begin{proof}{\it Step 1.} From the assumptions $|v_k|\leq 1$ and $\sup_k \mathcal{E}_{\varepsilon_k}(v_k,\omega)<\infty$, we first deduce from 
Lemma~\ref{adminHchap} that the sequence $\{v_k\}$ is bounded in $L^2(\R^n,\mathfrak{m})$, where the measure $\mathfrak{m}$ is defined in \eqref{defmeasm}. Therefore, we can find a subsequence and $v_*\in L^2(\R^n,\mathfrak{m})$ such that $v_k\rightharpoonup v_*$ weakly in $L^2(\R^n,\mathfrak{m})$. In particular, $v_k\rightharpoonup v_*$ weakly in $L^2_{\rm loc}(\R^n)$. On the other hand, the uniform energy bound also shows that  $|v_k|^2\to 1$ in $L^2(\omega)$, and $\{v_k\}$ is bounded in $H^{1/2}(\omega)$. Hence  $v_k\rightharpoonup v_*$ weakly in $H^{1/2}(\omega)$, and from the compact embedding $H^{1/2}(\omega)\hookrightarrow L^2(\omega)$, it implies that $v_k\to v_*$ strongly in $L^2(\omega)$. In particular, $|v_*|=1$ a.e. in $\omega$. 

We now claim that $v_*\in \widehat H^{1/2}(\omega;\mathbb{R}^m)$, and more precisely that 
$$\mathcal{E}(v_*,\omega)\leq \liminf_{k\to\infty} \mathcal{E}(v_k,\omega) <\infty\,.$$
In view of the weak convergence of $\{v_k\}$ in $H^{1/2}(\omega)$, it remains to show that 
\begin{equation}\label{tim1439bis}
\iint_{\omega\times\omega^c}\frac{|v_*(x)-v_*(y)|^2}{|x-y|^{n+1}}\,\de x\de y\leq  \liminf_{k\to\infty} \iint_{\omega\times\omega^c}\frac{|v_k(x)-v_k(y)|^2}{|x-y|^{n+1}}\,\de x\de y\,.
\end{equation}
 We fix $R>0$ and $0<\delta<R/2$ such that $\omega\subset D_{R/2}$. Set 
$\omega_{\delta}:=\{x\in\R^n: {\rm dist}(x,\omega)<\delta\}$, and write 
\begin{multline}
\iint_{\omega\times(\omega_\delta)^c\cap D_R}\frac{|v_k(x)-v_k(y)|^2}{|x-y|^{n+1}}\,\de x\de y=\int_{\omega}|v_k(x)|^2w_1(x)\,\de x\\
+  \int_{(\omega_\delta)^c\cap D_R}|v_k(x)|^2w_2(x)\,\de x
-2\int_{(\omega_\delta)^c\cap D_R} v_k(x)\cdot f_k(x)\,\de x=:I_k+II_k+III_k\,,
\end{multline}
where 
$$w_1(x):=\int_{(\omega_\delta)^c\cap D_R} \frac{\de y}{|x-y|^{n+1}}\,,\quad w_2(x):=\int_{\omega} \frac{\de y}{|x-y|^{n+1}}\,,\quad f_k(x):=\int_{\omega} \frac{v_k(y)\,\de y}{|x-y|^{n+1}}\,.$$
By the strong $L^2(\omega)$-convergence of $v_k$ toward $v_*$, $f_k$ converges strongly in $L^2((\omega_\delta)^c\cap D_R)$  to the function $f_*(x):=\int_{\omega} \frac{v_*(y)\,\de y}{|x-y|^{n+1}}$. From the weak $L^2_{\rm loc}(\R^n)$-convergence of $v_k$, we  deduce that 
$$\int_{\omega}|v_*(x)|^2w_1(x)\,\de x \leq \liminf_{k\to\infty} I_k\,,\qquad 
\int_{(\omega_\delta)^c\cap D_R}|v_*(x)|^2w_2(x)\,\de x\leq \liminf_{k\to\infty} II_k\,,$$
and 
$$\lim_{k\to\infty} III_k =  -2\int_{(\omega_\delta)^c\cap D_R} v_*(x)\cdot f_*(x)\,\de x\,.$$
Therefore, 
$$\iint_{\omega\times(\omega_\delta)^c\cap D_R}\frac{|v_*(x)-v_*(y)|^2}{|x-y|^{n+1}}\,\de x\de y\leq \liminf_{k\to\infty} 
\iint_{\omega\times\omega^c}\frac{|v_k(x)-v_k(y)|^2}{|x-y|^{n+1}}\,\de x\de y\,,$$ 
and \eqref{tim1439bis} follows letting $R\to\infty$ and $\delta\to 0$. 
\vskip3pt

We end this first step showing that $v_n^\e\rightharpoonup v_*^\e$ weakly in $H^1_{\rm loc}(\R^{n+1}_+\cup\omega)$. Indeed, we start deducing from Lemma~\ref{context} that $v_n^\e\rightharpoonup v_*^\e$ weakly in $L^2_{\rm loc}(\overline{\R}^{n+1}_+)$.  On the other hand, the uniform energy bound together with Lemma \ref{hatH1/2toH1} and standard estimates on harmonic functions, shows that $\{v_k^\e\}$ is bounded in $H^1_{\rm loc}(\R^{n+1}_+\cup\omega)$, whence the announced weak convergence. 
\vskip5pt

\noindent{\it Step 2.} Let us now consider an increasing  sequence $\{\Omega_l\}_{l\in\N}$ of  bounded admissible open sets such that  $\overline{\partial^0\Om_l}\subset\omega$ for every $l\in\N$, $\cup_l\Om_l=\R^{n+1}_+$, and $\cup_l\partial^0\Om_l=\omega$. By \eqref{bdlinftyext}, 
Step 1, and the results in Section \ref{FractGL},  $v_k^\e\in H^1(\Om_l;\R^m)\cap L^\infty(\Om_l)$  satisfies $|v_k^\e|\leq 1$ and solves 
$$ \begin{cases}
\Delta v^\e_k= 0 & \text{in $\Omega_l$}\,,\\[8pt]
\displaystyle \frac{\partial v^\e_k}{\partial \nu}=\frac{1}{\varepsilon_k}(1-|v^\e_k|^2)v^\e_k & \text{on $\partial^0\Omega_l$}\,,  
\end{cases}
$$
for every $l\in\N$. In addition, we have proved in Step 1 that $\sup_k E_{\eps_k}(v_k^\e,\Om_l)<\infty$ for every $l\in\N$. Therefore, we can find a further subsequence such that the conclusions of Theorems~\ref{asymptneum} \& \ref{statdef} hold in every $\Om_l$, and $v_*^\e$ is the limiting $(\mathbb{S}^{m-1}, \partial^0\Om_l)$-boundary harmonic map in each $\Om_l$ by Step 1. This yields the announced conclusions on the defect measure $\mu_{\rm sing}$ and on the concentration set $\Sigma$ stated in (i), (ii), (iii), and (iv). The stationarity relation stated in (v) between $v_*$ and $\mu_{\rm sing}$ is in turn a direct consequence of 
Theorem  \ref{statdef} and Lemma~\ref{compstatfrac}. Then it only remains to prove that $v_*$ is a weak 1/2-harmonic map into $\mathbb{S}^{m-1}$ in $\omega$. By Proposition \ref{caracthalfharm}, it is enough to check that $\langle  (-\Delta)^{\frac{1}{2}}  v_*,\varphi\rangle_\omega=0$ for every $\varphi\in H^{1/2}_{00}(\omega;\R^m)\cap L^\infty(\omega)$ compactly supported in~$\omega$ satisfying  $v_*\cdot \varphi=0$ a.e. in~$\omega$.  Given such a test function $\varphi$, we consider an arbitrary extension $\Phi\in H^1(\R^{n+1}_+;\R^m)\cap L^\infty(\R^{n+1}_+)$ of $\varphi$ which is compactly supported in $\R^{n+1}_+\cup\omega$. Then ${\rm supp}\,\Phi\subset \Om_l\cup\partial^0\Om_l$ for $l$ large enough. Since $v_*^\e$ is a weak $(\mathbb{S}^{m-1},\partial^0\Om_l)$-boundary harmonic map in $\Om_l$, we infer from Lemma \ref{repnormderfraclap} that 
$$\langle  (-\Delta)^{\frac{1}{2}}  v_*,\varphi\rangle_\omega=\int_{\Om_l} \nabla v_*^\e \cdot\nabla\Phi\,\de x =0\,, $$
and the proof is complete. 
\end{proof}

\begin{remark}\label{lastrem}
We notice that Theorem~\ref{asymptneum} actually implies that $v_k^\e\to v_*^\e$ in $ C^{\ell}_{\rm loc}\big(\R^{n+1}_+\cup(\omega\setminus\Sigma)\big)$ for every $\ell\in\N$. In view 
of Lemma \ref{repnormderfraclap}, it shows that $(-\Delta)^{\frac1 2}v_k\to (-\Delta)^{\frac1 2}v_*$ in $ C^{\ell}_{\rm loc}(\omega\setminus\Sigma)$ for every $\ell\in\N$.
\end{remark}

\subsection{Asymptotics with Dirichlet exterior condition} 
 This section is devoted to  proof of  Theorem \ref{mainthm}.  To this aim, we consider for the rest of this subsection a smooth bounded open set   $\omega\subset \R^n$, and a  smooth exterior Dirichlet condition $g:\R^n\to \R^m$ satisfying  
 $|g|=1$ in $\R^n\setminus\omega$.  
 
 Given an arbitrary sequence $\eps_k\downarrow 0$, we consider $\{v_k\}_{k\in\N}\subset H^{1/2}_g(\omega;\R^m)\cap L^4(\omega)$ such that for each $k\in\N$, $v_k$ weakly solves  \eqref{eqfracfinal}, and $\sup_k \mathcal{E}_{\eps_k}(v_k,\omega)<\infty$. We recall that under such assumptions, we have proved in Section \ref{FractGL} that $v_k\in  C_{\rm loc}^{0,\alpha}(\R^n)\cap  C^\infty(\R^n\setminus\partial\omega)$ for every $0<\alpha<1/2$, and that $|v_k|\leq 1$ in $\R^n$. Therefore, we can apply 
 Theorem~\ref{thmwithout}   to find a (not relabeled) subsequence such that $v_k\rightharpoonup v_*$ weakly in $H^{1/2}(\omega)\cap L^2_{\rm loc}(\R^n)$ 
 for some map $v_*\in \widehat H^{1/2}(\omega;\R^m)$ which is a bounded weak 1/2-harmonic map into $\mathbb{S}^{m-1}$ in $\omega$, and such that all the conclusions of 
 Theorem~\ref{thmwithout} hold. Since $v_k$ is constantly equal to $g$ outside~$\omega$,   we also infer that $v_*\in H^{1/2}_g(\omega;\R^m)$.  
 
 It now remains to prove that 
 \begin{itemize}
\item[\it (a)]  $v_k-v_*\rightharpoonup 0$ weakly in $H^{1/2}_{00}(\omega)$;  
\vskip3pt
\item[\it (b)] $\mu_{\rm sing}$ is a finite measure on $\omega$, $\mathscr{H}^{n-1}(\Sigma)<\infty$, and statements (i) and (iii) in Theorem \ref{mainthm} hold. 
\end{itemize}
  
\begin{proof}[Proof of (a)] We first notice that $v_k-v_*\in H^{1/2}_{00}(\omega;\R^m)$ for every $k\in\N$ by \eqref{caracH1/2g}. Hence, 
$$[v_k-v_*]^2_{H^{1/2}(\R^n)}= \mathcal{E}(v_k-v_*,\omega)\leq 2\big(\mathcal{E}(v_k,\omega)+\mathcal{E}(v_*,\omega)\big)\,,$$ 
so that $\{v_k-v_*\}$ in bounded in   $H^{1/2}_{00}(\omega)$. Since $v_k-v_*\rightharpoonup 0$ weakly in $H^{1/2}(\omega)$, it remains to show that 
$$\iint_{\omega\times\omega^c} \frac{(v_k(x)-v_*(x))\cdot \varphi(x)}{|x-y|^{n+1}}\,\de x\de y \to 0$$
for every $\varphi\in H^{1/2}_{00}(\omega;\R^m)$. First notice that  $v_k\to v_*$ strongly in $L^2(\omega)$ by the compact embedding $H^{1/2}(\omega)\hookrightarrow L^2(\omega)$. Then, by  density of smooth maps compactly supported in $\omega$, it suffices 
to consider the case where $\varphi\in\mathscr{D}(\omega)$. For such a test function $\varphi$, the assertion is easily proved using the dominated convergence Theorem and the fact that $v_k\to v_*$ a.e. in $\omega$ (up to a further subsequence if necessary). 
\end{proof}

\begin{proof}[Proof of (b)]  From the uniform energy bound, we may find a further subsequence such that 
$${\rm e}(v_k,\omega)\,\mathscr{L}^n\LL\omega \mathop{\rightharpoonup}\limits^* {\rm e}(v_*,\omega)\,\mathscr{L}^n\LL\omega+\mu_{\rm def}$$
  weakly* as Radon measures on $\omega$ for some finite nonnegative measure $\mu_{\rm def}$. We thus have to prove that 
  $\mu_{\rm def}\equiv\mu_{\rm sing}$. It will then show that $\mu_{\rm sing}$ is finite,  and that $\mathscr{H}^{n-1}(\Sigma) <\infty$ (since  $\mathscr{H}^{n-1}(\Sigma) $ is controlled by the total variation of $\mu_{\rm sing}$, see \eqref{train}). 
  
  Let us now fix $\varphi\in\mathscr{D}(\omega)$ arbitrary. We notice that 
  \begin{align*}
  \int_\omega {\rm e}(v_k,\omega)\varphi\,\de x 
     = &\,  \langle  (-\Delta)^{\frac{1}{2}}  v_k,\varphi v_k\rangle_\omega  -\frac{\gamma_n}{2}\iint_{\omega\times\omega} \frac{(v_k(x)-v_k(y))\cdot v_k(y)(\varphi(x)-\varphi(y))}{|x-y|^{n+1}}\,\de x\de y\\
 & -\gamma_n\iint_{\omega\times\omega^c} \frac{(v_k(x)-g(y))\cdot g(y)\varphi(x)}{|x-y|^{n+1}}\,\de x\de y\\
=:&\,I_k-II_k-III_k\,.
  \end{align*}
We consider  a function $\Phi\in  C^\infty(\overline{\R}^{n+1}_+)$ compactly supported in $\R^{n+1}_+\cup\omega$  
such that $\Phi_{|\R^n}=\varphi$, and set $K:={\rm supp}(\Phi)$. We observe that $\Phi v_k^\e$ is a smooth function compactly supported by $K$, so that Lemma~\ref{repnormderfraclap}  yields 
$$\langle  (-\Delta)^{\frac{1}{2}}  v_k,\varphi v_k\rangle_\omega=\int_{\R^{k+1}_+}\nabla v_k^\e\cdot\nabla(\Phi v_k^\e)\,\de x =
\int_{\R^{n+1}_+}|\nabla v_k^\e|^2\Phi\,\de x+\int_{K}\nabla v_k^\e\cdot(v_k^\e \nabla\Phi)\,\de x\,.$$
From {\it (a)} we deduce that $v_k^\e \rightharpoonup v_*^\e$ weakly in $H^1_{\rm loc}(\overline{\R}^{n+1}_+)$, so that $\nabla v_n^\e\rightharpoonup\nabla v_*^\e$ weakly in $L^2(K)$ and  $v_k^\e\to v_*^\e$ strongly in $L^2(K)$. Together with item (i) in  Theorem~\ref{thmwithout},  it yields 
\begin{multline*}
\langle  (-\Delta)^{\frac{1}{2}}  v_k,\varphi v_k\rangle_\omega\mathop{\longrightarrow}\limits_{k\to\infty}  \int_{\R^{n+1}_+}|\nabla v_*^\e|^2\Phi\,\de x+\int_\omega\varphi\,\de \mu_{\rm sing}+\int_{K}\nabla v_*^\e\cdot(v_*^\e \nabla\Phi)\,\de x\\
 =\int_{\R^{n+1}_+}\nabla v_*^\e\cdot\nabla(\Phi v_*^\e)\,\de x+\int_\omega\varphi\,\de \mu_{\rm sing}\,.
\end{multline*}
By  Lemma \ref{repnormderfraclap} again, we have thus proved that 
\begin{equation}\label{tim1555}
\langle  (-\Delta)^{\frac{1}{2}}  v_k,\varphi v_k\rangle_\omega\mathop{\longrightarrow}\limits_{k\to\infty} \langle  (-\Delta)^{\frac{1}{2}}  v_*,\varphi v_*\rangle_\omega+\int_\omega\varphi\,\de \mu_{\rm sing}\,,
\end{equation}
where we have used the (elementary) fact that $\varphi v_*\in H^{1/2}_{00}(\omega;\R^m)$. 

On the other hand, we have  that $v_k\to v_*$ a.e. on $\R^n$, eventually after the extraction of a further subsequence. Using $\varphi\in\mathscr{D}(\omega)$ and the weak convergence of $v_k$ in $H^{1/2}(\omega)$, we  deduce  that 
\begin{equation}\label{tim1627}
II_k \to \frac{\gamma_n}{2}\iint_{\omega\times\omega} \frac{(v_*(x)-v_*(y))\cdot v_*(y)(\varphi(x)-\varphi(y))}{|x-y|^{n+1}}\,\de x\de y
\end{equation}
and by dominated convergence, 
\begin{equation}\label{tim1628}
III_k\to  \gamma_n\iint_{\omega\times\omega^c} \frac{(v_*(x)-g(y))\cdot g(y)\varphi(x)}{|x-y|^{n+1}}\,\de x\de y
\end{equation}
as $k\to\infty$. Gathering \eqref{tim1555}, \eqref{tim1627}, and \eqref{tim1628} leads to 
$$\int_\omega {\rm e}(v_k,\omega)\varphi\,\de x \mathop{\longrightarrow}\limits_{k\to\infty} \int_\omega {\rm e}(v_*,\omega)\varphi\,\de x
+\int_\omega\varphi\,\de \mu_{\rm sing}\,,$$
and thus $\mu_{\rm def}=\mu_{\rm sing}$ by the arbitrariness of $\varphi$. 

To derive item (iii), we just notice that
\begin{align*}
\int_\omega \frac{(1-|v_k|^2)}{\varepsilon_k}\,\varphi\,\de x & =  \int_\omega \frac{(1-|v_k|^2)^2}{\varepsilon_k}\,\varphi\,\de x
+ \frac{1}{\varepsilon_k}\int_\omega(1-|v_k|^2)v_k\cdot(\varphi v_k)\,\de x \\
& = \int_\omega \frac{(1-|v_k|^2)^2}{\varepsilon_k}\,\varphi\,\de x+ \langle  (-\Delta)^{\frac{1}{2}}  v_k,\varphi v_k\rangle_\omega \,,
\end{align*}
so that the announced convergence follows from item (ii), \eqref{tim1555}, and Remark \ref{eq1/2harm}. 
\end{proof}
  
\begin{remark} 
In view of Remarks \ref{eq1/2harm} \& \ref{lastrem}, we have  
$$\displaystyle \frac{1-|v_k(x)|^2}{\eps_k}\to \frac{\gamma_n}{2}\int_{\R^n}\frac{|v_*(x)-v_*(y)|^2}{|x-y|^{n+1}}\,\de y  \;\text{ in }  C^{\ell}_{\rm loc}(\omega\setminus\Sigma)$$
for every $\ell\in\N$. 
\end{remark}

\subsection{Asymptotics for Dirichlet minimizers}   
  
We finally consider solutions of the minimization problem~\eqref{GLminProb}, and we show that, in this case, no concentration occurs by minimality.    
  
\begin{theorem}\label{thmmin}
Let $\omega\subset\R^{n}$ be a smooth bounded open set, and let $g:\R^n\to \R^m$ be a smooth map satisfying $|g|=1$ in $\R^n\setminus\omega$, and such that 
\begin{equation}\label{assumpnotempty}
H_g^{1/2}(\omega;\mathbb{S}^{m-1}):=\left\{v\in H_g^{1/2}(\omega;\R^m): |v|=1 \text{ a.e. in $\omega$}\right\}\not =\emptyset\,.
\end{equation}
Let $\varepsilon_k\downarrow0$ be an arbitrary sequence,  and let  
$\{v_k\}_{k\in\mathbb{N}}\subset H^{1/2}_g(\omega;\mathbb{R}^m)\cap L^4(\omega)$ be such that for each $k\in\N$, 
\begin{equation}\label{argmin}
v_k\in {\rm argmin} \left\{ \mathcal{E}_{\eps_k}(v,\omega) :  v\in H_g^{1/2}(\omega;\R^m)\cap L^4(\omega)\right\} \,.
\end{equation}
Then there exist a (not relabeled) subsequence  and $v_*\in H^{1/2}_g(\omega;\R^m)$ a minimizing 1/2-harmonic map into $\mathbb{S}^{m-1}$ in $\omega$ such that $v_k-v_*\to 0$ strongly in $H^{1/2}_{00}(\omega)$. In addition, 
\vskip5pt
\begin{itemize}[leftmargin=22pt]
\item[\rm (i)] $\displaystyle  \mathcal{E}_{\eps_k}(v_k,\omega)\to \mathcal{E}(v_*,\omega)$;
\vskip5pt
\item[\rm (ii)] $\displaystyle \frac{1-|v_k(x)|^2}{\eps_k}\rightharpoonup \frac{\gamma_n}{2}\int_{\R^n}\frac{|v_*(x)-v_*(y)|^2}{|x-y|^{n+1}}\,\de y $ in $\mathscr{D}'(\omega)$; 
\vskip5pt
\item[\rm (iii)]  $v_n\to v_*$ in $ C^{\ell}_{\rm loc}(\omega\setminus{\rm sing}(v_*))$ for every $\ell\in\N$; 
\vskip5pt
\item[\rm (iv)] if $n=1$, then $\omega\cap{\rm sing}(v_*)=\emptyset$, while ${\rm dim}_{\mathscr{H}}\big({\rm sing}(v_*)\cap\omega\big)\leq n-2$ for $n\geq 3$, and ${\rm sing}(v_*)\cap\omega$ is discrete for $n=2$. 
\end{itemize}
\end{theorem}

\begin{proof}
By assumption \eqref{assumpnotempty} there exists $\widetilde g \in H_g^{1/2}(\omega;\R^m)$ such that $|\widetilde g|=1$ a.e. in~$\omega$. Then we infer from \eqref{argmin} that for all $k\in\N$, 
$$\mathcal{E}_{\eps_k}(v_k,\omega)\leq \mathcal{E}_{\eps_k}(\widetilde g,\omega) =\mathcal{E}(\widetilde g,\omega) \,.$$
Therefore $\sup_k \mathcal{E}_{\eps_k}(v_k,\omega)<\infty$, and we can extract a (not relabeled) subsequence such that the conclusions of Theorem~\ref{mainthm} do hold for a map 
$v_*\in \widehat H^{1/2}_g(\omega;\R^m)$ which is a bounded weak 1/2-harmonic map into $\mathbb{S}^{m-1}$ in $\omega$, and a finite measure $\mu_{\rm sing}$ on $\omega$. By the compact embedding $H^{1/2}(\omega)\hookrightarrow L^2(\omega)$, we may assume that $v_k\to v_*$ strongly in $L^2(\omega)$. Since $v_k=v_*=g$ outside $\omega$, we deduce that $v_k\to v_*$ a.e. in~$\R^n$, eventually up to a further subsequence. By \eqref{assumpnotempty}  and the lower semicontinuity of the 1/2-Dirichlet energy with respect to a.e. pointwise convergence, we have 
$$\mathcal{E}(v,\omega)\geq \liminf_{k\to\infty} \mathcal{E}_{\eps_k}(v_k,\omega) \geq \liminf_{k\to\infty} \mathcal{E}(v_k,\omega)\geq \mathcal{E}(v_*,\omega)$$
for all $v\in H_g^{1/2}(\omega;\mathbb{S}^{m-1})$. Hence $v_*$ is a minimizing 1/2-harmonic map into $\mathbb{S}^{m-1}$ in $\omega$, and item (iv) follows from Theorem~\ref{reghalfharmintro}. On the other hand, the minimality of $v_k$ yields 
$$ \mathcal{E}_{\eps_k}(v_k,\omega)\leq \mathcal{E}_{\eps_k}(v_*,\omega)=\mathcal{E}(v_*,\omega)\,.$$
Combining the two previous inequalities leads to item (i). In turn, it implies that $\mu_{\rm sing}=0$, and then yields item (ii). 
The conclusion in item (iii) is a straightforward consequence of Theorem~\ref{asymptneum} - item (iii) - together with the fact that $\mu_{\rm sing}=0$. It now only remains to prove the strong convergence of $v_k-v_*$ to $0$ in  $H^{1/2}_{00}(\omega)$. To this purpose, we have to prove that
\begin{equation}\label{last}
\lim_{k\to\infty}\big\langle  (-\Delta)^{\frac{1}{2}} v_k, v_*\big\rangle_{\omega}=\big\langle  (-\Delta)^{\frac{1}{2}} v_*, v_*\big\rangle_{\omega}=2\mathcal{E}(v_*,\omega)\,.
\end{equation} 
Indeed, if \eqref{last}  holds, then item (i) yields 
$$ [v_k-v_*]^2_{H^{1/2}(\mathbb{R}^n)}=\mathcal{E}(v_k-v_*,\omega)= \mathcal{E}(v_k,\omega)+\mathcal{E}(v_*,\omega)-\big\langle(-\Delta)^{\frac{1}{2}} v_k, v_*\big\rangle_{\omega}\mathop{\longrightarrow}\limits_{k\to\infty} 0\,.$$
To show \eqref{last} we first write
\begin{multline*}
 \big\langle  (-\Delta)^{\frac{1}{2}} v_k, v_*\big\rangle_{\omega}= \frac{\gamma_n}{2}\iint_{\omega\times\omega}  \frac{(v_k(x)-v_k(y))\cdot(v_*(x)-v_*(y))}{|x-y|^{n+1}}\,\de x\de y\\ 
+ \gamma_n \iint_{\omega\times \omega^c}  \frac{(v_k(x)-v_*(y))\cdot(v_*(x)-v_*(y))}{|x-y|^{n+1}}\,\de x\de y =:I_k+II_k\,.
\end{multline*}
Since $v_k\rightharpoonup v_*$ weakly in $H^{1/2}(\omega)$, we have 
$$I_k\mathop{\longrightarrow}\limits_{k\to\infty}   \frac{\gamma_n}{2}\iint_{\omega\times\omega}  \frac{|v_*(x)-v_*(y)|^2}{|x-y|^{n+1}}\,\de x\de y \,.$$
Next we set $V_k(x,y):=(v_k(x)-v_*(y))=(v_k(x)-v_k(y))$. Then $\{V_k\}$ is bounded in $L^2(\omega\times\omega^c, {\mathbf \Lambda})$ for the weighted measure
${\mathbf \Lambda}:=|x-y|^{-n-1}\mathscr{L}^n_x\otimes\mathscr{L}^n_y$.  Extracting a further subsequence if necessary, $V_k$ is thus  weakly converging in $L^2(\omega\times\omega^c, {\mathbf \Lambda})$ to some function $V$. Using the strong convergence in $L^2(\omega)$ of~$v_k$, we can argue as in the proof of Lemma \ref{approxlemma} in Appendix A to show that $V(x,y)=(v_*(x)-v_*(y))$. As a consequence, 
$$\lim_{k\to\infty}\gamma_n \iint_{\omega\times \omega^c}  \frac{(v_k(x)-v_*(y))\cdot(v_*(x)-v_*(y))}{|x-y|^{n+1}}\,\de x\de y =\gamma_n \iint_{\omega\times \omega^c}  \frac{|v_*(x)-v_*(y)|^2}{|x-y|^{n+1}}\,\de x\de y\,,$$
and \eqref{last}  is proved. 
\end{proof}

\vskip10pt

\section{Appendix A}   						        

The purpose of this first appendix is to give detailed proofs of the different auxiliary results stated in Section \ref{prelim}. 

\begin{proof}[Proof of Lemma~\ref{adminHchap}.]  Let $\omega\subset\R^n$ be a bounded open set, and $v\in \widehat H^{1/2}(\omega;\R^m)$. Let us fix 
$x_0\in\omega$ and $\rho>0$ such that $D_{2\rho}(x_0)\subset\omega$. Then, 
$$\iint_{ D_{2\rho}^c(x_0)\times D_\rho(x_0)} \frac{|v(x)-v(y)|^2}{|x-y|^{n+1}}\,\de x\de y\leq C\mathcal{E}\big(v,D_{2\rho}(x_0)\big)\,.$$
Next we estimate
\begin{align*}
\iint_{D_{2\rho}^c(x_0)\times D_\rho(x_0)}  \frac{|v(x)-v(y)|^2}{|x-y|^{n+1}}\,\de x\de y  
&\geq \iint_{D_{2\rho}^c(x_0)\times D_\rho(x_0)} \frac{\big||v(x)|-|v(y)|\big|^2}{(|x-x_0|+|y-x_0|)^{n+1}}\,\de x\de y \\
&\geq C_{\rho} \iint_{D_{2\rho}^c(x_0)\times D_\rho(x_0)} \frac{|v(x)|^2-2|v(y)|^2}{(|x-x_0|+1)^{n+1}}\,\de x\de y\,,
\end{align*}
which yields 
$$\int_{D^c_{2\rho}(x_0)}  \frac{|v(x)|^2}{(|x-x_0|+1)^{n+1}}\,\de x\leq C_{\rho}\left(\mathcal{E}\big(v,D_{2\rho}(x_0)\big)+\|v\|^2_{L^2(D_\rho(x_0))}\right)\,.$$
Hence, 
$$\int_{\R^n}  \frac{|v(x)|^2}{(|x-x_0|+1)^{n+1}}\,\de x\leq C_{\rho}\left(\mathcal{E}\big(v,D_{2\rho}(x_0)\big)+\|v\|^2_{L^2(D_{2\rho}(x_0))}\right)\,,$$
and the proof is complete.  
\end{proof}

\begin{proof}[Proof of Lemma \ref{hatH1/2toH1}, Part 1.]  We prove in this first part that $v^\e\in L^2_{\rm loc}(\overline{\R}^{n+1}_+)$. We fix  $R>0$ arbitrary, $x_0\in\omega$ and $\rho>0$ such that $D_{2\rho}(x_0)\subset\omega$. We claim that  
$v^\e\in L^2(B_R^+(x_0))$. 
Using Jensen's inequality 
 we estimate for $x=(x',x_{n+1})\in B_R^+(x_0)$, 
 \begin{align}
\nonumber  |v^\e(x)|^2  &\leq \gamma_n\int_{\R^n} \frac{x_{n+1}|v(z)|^2}{(|x'-z|^2+x^2_{n+1})^{\frac{n+1}{2}}}\,\de z\\[8pt]
\nonumber\displaystyle & 
\leq  \gamma_n\int_{D^c_{5R}(x_0)} \frac{x_{n+1}|v(z)|^2}{(|z-x_0|^2-2R|z-x_0|)^{\frac{n+1}{2}}}\,\de z
+\gamma_n\int_{D_{5R}(x_0)} \frac{x_{n+1}|v(z)|^2}{(|x'-z|^2+x^2_{n+1})^{\frac{N+1}{2}}}\,\de z\\
\label{estil2loc1}  &\leq C_R\int_{\R^n} \frac{|v(z)|^2}{(|z-x_0|+1)^{n+1}}\,\de z+\gamma_n\int_{D_{5R}(x_0)} \frac{x_{n+1}|v(z)|^2}{(|x'-z|^2+x^2_{n+1})^{\frac{n+1}{2}}}\,\de z\,.
 \end{align}
Then we estimate for $0<x_{n+1}<R$, 
\begin{align}
\nonumber \int_{D_{R}(x_0)}\bigg(\int_{D_{5R}(x_0)} & \frac{x_{n+1}|v(z)|^2}{(|x'-z|^2+x^2_{n+1})^{\frac{n+1}{2}}}\,\de z\bigg)\,\de x' \\
\nonumber &= \int_{D_{R}(x_0)}\bigg(\int_{D_{5R}(x')} \frac{x_{n+1}|v(x'-y+x_0)|^2}{(|y-x_0|^2+x^2_{n+1})^{\frac{n+1}{2}}}\,\de y\bigg)\,\de x' \\
\nonumber& \leq\int_{D_{6R}(x_0)} \bigg( \int_{D_{R}(x_0)} \frac{x_{n+1}|v(y+x_0-x')|^2}{(|y-x_0|^2+x^2_{n+1})^{\frac{n+1}{2}}}\,\de x'\bigg)\,\de y \\
\nonumber & \leq \left(\int_{D_{6R}(x_0)}\frac{x_{n+1}}{(|y-x_0|^2+x^2_{n+1})^{\frac{n+1}{2}}}\,\de y\right)\left(\int_{D_{7R}(x_0)}|v(z)|^2\,\de z\right)\\
\label{estil2loc2} & \leq C \int_{D_{7R}(x_0)}|v(z)|^2\,\de z\,.
\end{align}
Combining \eqref{estil2loc1} and \eqref{estil2loc2} we deduce from Lemma \ref{adminHchap} that 
\begin{equation}\label{conclestil2loc}
\int_{B_R^+(x_0)} |v^\e(x)|^2\,\de x\leq C_R \int_{\R^n} \frac{|v(z)|^2}{(|z-x_0|+1)^{n+1}}\,\de z 
\leq C_{R,\rho}\left(\mathcal{E}\big(v,D_{2\rho}(x_0)\big)+\|v\|^2_{L^2(D_{2\rho}(x_0))}\right)\,, 
\end{equation}
which  ends this first part.
\end{proof}

\begin{remark}[{\bf Proof of Lemma~\ref{context}}]
Notice that the first inequality in \eqref{conclestil2loc} shows the continuity of the linear operator $\mathfrak{P}_R$ defined in \eqref{Pfrak}. 
\end{remark}

To complete the proof of Lemma  \ref{hatH1/2toH1}, we shall need the following smooth approximation result.

\begin{lemma}\label{approxlemma} 
Let $\omega\subset\R^n$ be a bounded open set, and let $v\in  \widehat{H}^{1/2}(\omega;\R^m)$. There exists a sequence 
$\{v_k\}_{k\in\N}$ of smooth functions  such that $v_k\to v$ strongly in 
$ L^{2}_{\rm loc}(\mathbb{R}^n)$, and such that for any relatively compact  open subset $\omega'\subset\omega$ with Lipschitz boundary, 
\begin{itemize}[leftmargin=22pt]
\item[\rm (i)]  $v_k\in \widehat{H}^{1/2}(\omega';\R^m)$ for $k$ large enough;
\vskip3pt

\item[\rm (ii)] $\mathcal{E}(v_k-v,\omega')\to 0$ as $k\to \infty$;
\vskip3pt

\item[\rm (iii)] $v_k^\e\to v^\e$ in $L^2_{\rm loc}(\overline{\R}^{n+1}_+)$ as $k\to\infty$. 
\end{itemize}
\end{lemma}

\begin{proof}
Let us fix a sequence $\varepsilon_k\downarrow0$, and a nonnegative function $\varrho\in  C^\infty_c(\R^n)$ compactly supported in $D_1$ 
satisfying $\int_{D_1}\varrho\,dx=1$. Setting $\varrho_k(x):=\varepsilon_k^{-n}\varrho(x/\varepsilon_k)$, we define 
$$v_k(x):= \int_{D_{\varepsilon_k}}\varrho_k(z)v(x+z)\,dz\,.$$
Then $v_k\in  C^\infty(\R^n;\R^m)$, $v_k\to v$ strongly in $ L^{2}_{\rm loc}(\mathbb{R}^n)$. 
Extracting a subsequence if necessary, we may assume that $v_k\to v$ a.e. in $\R^n$. 

Let us now consider a relatively compact open subset $\omega'\subset\omega$ with Lipschitz boundary. Set 
$$\delta_0:=\min\left\{ {\rm dist}(\omega',\partial\omega), \sup_{x\in\omega'}{\rm dist}(x,\partial\omega')\right\}\,,$$ 
and define for $0<\delta<\delta_0$, 
$$\omega_\delta:=\{x\in\omega: {\rm dist}(x,\omega')<\delta\} \,,\quad \omega_{-\delta}:=\{x\in\omega':{\rm dist}(x,\partial\omega')>\delta\}\,.$$
For $0<\varepsilon_k<\delta<\delta_0$, we estimate through Jensen's inequality, 
\begin{align}
\nonumber\iint_{\omega'\times\omega'}\frac{|v_k(x)-v_k(y)|^2}{|x-y|^{n+1}}\;\de x\de y & \leq \iint_{\omega'\times\omega'}\left( \int_{D_{\varepsilon_k}}\varrho_k(z) \frac{|v(x+z)-v(y+z)|^2}{|x-y|^{n+1}}\,\de x\de y\right)\,dz\\
\nonumber & \leq \int_{D_{\varepsilon_k}}\varrho_k(z)\left(\iint_{(z+\omega')\times(z+\omega')}\frac{|v(x)-v(y)|^2}{|x-y|^{n+1}}\,\de x\de y\right)\,dz\\
\label{heur1445} &\leq \iint_{\omega_\delta\times\omega_\delta}\frac{|v(x)-v(y)|^2}{|x-y|^{n+1}}\,\de x\de y\,.
\end{align}
Hence $\{v_k\}$ is bounded in $H^{1/2}(\omega';\R^m)$, and since $v_k\to v$ in $L^2(\omega')$ we deduce that $v_k\rightharpoonup v$ 
weakly in  $H^{1/2}(\omega';\R^m)$. Now we infer from \eqref{heur1445} and Fatou's lemma that 
\begin{multline*}
\iint_{\omega'\times\omega'}\frac{|v(x)-v(y)|^2}{|x-y|^{n+1}}\;\de x\de y \leq 
\liminf_{k\to+\infty}\iint_{\omega'\times\omega'}\frac{|v_k(x)-v_k(y)|^2}{|x-y|^{n+1}}\;\de x\de y\\ 
\leq \limsup_{k\to+\infty}\iint_{\omega'\times\omega'}\frac{|v_k(x)-v_k(y)|^2}{|x-y|^{n+1}}\;\de x\de y
 \leq  \iint_{\omega_\delta\times\omega_\delta}\frac{|v(x)-v(y)|^2}{|x-y|^{n+1}}\,\de x\de y\,.
\end{multline*}
Since $[v]^2_{H^{1/2}(\omega_\delta)}\to [v]^2_{H^{1/2}(\omega')}$ as $\delta\downarrow0$ by monotone convergence, we conclude from this last inequality and the arbitrariness of $\delta$ that $[v_k]^2_{H^{1/2}(\omega')}\to [v]^2_{H^{1/2}(\omega')}$. Therefore $v_k\to v$ strongly in $H^{1/2}(\omega')$.  

Similarly to \eqref{heur1445}, we estimate
$$\iint_{\omega'\times(\omega')^c}\frac{|v_k(x)-v_k(y)|^2}{|x-y|^{n+1}}\;\de x\de y \leq \iint_{\omega_\delta\times\omega^c_{-\delta}}\frac{|v(x)-v(y)|^2}{|x-y|^{n+1}}\,\de x\de y\,.
$$
Arguing as above, we apply Fatou's lemma as $k\to\infty$, and then the monotone convergence theorem as $\delta\downarrow0$ to deduce that
\begin{equation}\label{time2034}
\iint_{\omega'\times(\omega')^c}\frac{|v_k(x)-v_k(y)|^2}{|x-y|^{n+1}}\;\de x\de y \mathop{\longrightarrow}\limits_{k\to\infty} \iint_{\omega'\times(\omega')^c}\frac{|v(x)-v(y)|^2}{|x-y|^{n+1}}\;\de x\de y\,. 
\end{equation}
We have thus proved that 
$$\mathcal{E}(v_k,\omega')\mathop{\longrightarrow}\limits_{k\to\infty} \mathcal{E}(v,\omega')\,. 
$$
In particular,  $v_k\in \widehat{H}^{1/2}(\omega';\R^m)$ for $k$ large enough. 

Setting $V_k(x,y):=(v_k(x)-v_k(y))$,  \eqref{time2034} implies that the sequence $\{V_k\}$ is bounded in $L^2(\omega'\times(\omega')^c, {\mathbf\Lambda})$ for the weighted measure
$ {\mathbf \Lambda}:=|x-y|^{-n-1}\mathscr{L}^n_x\otimes\mathscr{L}^n_y$.  
Hence we can extract a subsequence such that $V_k\rightharpoonup w$ weakly $L^2(\omega'\times(\omega')^c,{\mathbf \Lambda})$. 
On the other hand, from the convergence of $v_k$ to $v$ in $L^2_{\rm loc}(\mathbb{R}^n)$ we deduce that for all $F\in L^2(\omega'\times(\omega')^c,{\mathbf \Lambda})$ 
compactly  supported in $\omega'\times\big(\,\overline{\omega'}\,\big)^c$, 
\begin{align*}
\iint_{\omega'\times(\omega')^c} \frac{V(x,y)\cdot F(x,y)}{|x-y|^{n+1}}\,\de x\de y& = \lim_{k\to\infty} \iint_{\omega'\times(\omega')^c} \frac{V_k(x,y)\cdot F(x,y)}{|x-y|^{n+1}}\,\de x\de y\\
& = \iint_{\omega'\times(\omega')^c} \frac{(v(x)-v(y))\cdot F(x,y)}{|x-y|^{n+1}}\,\de x\de y\,.
\end{align*}
Functions with compact support in $\omega'\times\big(\,\overline{\omega'}\,\big)^c$ being dense in $L^2(\omega'\times(\omega')^c,{\mathbf \Lambda})$, we conclude that 
$$\iint_{\omega'\times(\omega')^c} \frac{V(x,y)\cdot F(x,y)}{|x-y|^{n+1}}\,\de x\de y =  \iint_{\omega'\times(\omega')^c} \frac{(v(x)-v(y))\cdot F(x,y)}{|x-y|^{n+1}}\,\de x\de y$$
for all $F\in L^2(\omega'\times(\omega')^c, {\mathbf \Lambda})$. The Riesz Representation Theorem then yields $V(x,y)=(v(x)-v(y))$. As a consequence, 
\begin{equation}\label{tim1911}
\iint_{\omega'\times(\omega')^c}\frac{(v_k(x)-v_k(y))\cdot(v(x)-v(y))}{|x-y|^{n+1}}\;\de x\de y 
\mathop{\longrightarrow}\limits_{k\to\infty} \iint_{\omega'\times(\omega')^c}\frac{|v(x)-v(y)|^2}{|x-y|^{n+1}}\,\de x\de y\,.
\end{equation}
Combing \eqref{tim1911} with the convergence of $v_n$ in $H^{1/2}(\omega';\R^m)$, we infer  that 
$$\big\langle  (-\Delta)^{\frac{1}{2}} v_k,v\big\rangle_{\omega'} \mathop{\longrightarrow}\limits_{k\to\infty} \big\langle  (-\Delta)^{\frac{1}{2}} v,v\big\rangle_{\omega'} =2\mathcal{E}(v,\omega')\,,$$
and we finally conclude 
$$\mathcal{E}(v_k-v,\omega')=\mathcal{E}(v,\omega')-\big\langle  (-\Delta)^{\frac{1}{2}} v_k,v\big\rangle_{\omega'}+\mathcal{E}(v_k,\omega') \mathop{\longrightarrow}\limits_{k\to\infty} 0\,,$$
which proves (ii).
\vskip3pt

It now remains to show (iii). We consider an arbitrary radius $R>0$, $x_0\in \omega'$ and $\rho>0$ such that $D_{2\rho}(x_0)\subset\omega' $. Noticing that $(v_k-v)^\e=v_k^\e-v^\e$, 
we deduce from \eqref{conclestil2loc} that
\begin{multline*}
\int_{B_R^+(x_0)}|v_k^\e-v^\e|^2\,\de x  \leq  C_{R,\rho}\left(\mathcal{E}\big(v_k-v,D_{2\rho}(x_0)\big)+\|v_k-v\|^2_{L^2(D_{2\rho}(x_0))}\right)\\
\leq  C_{R,\rho}\left(\mathcal{E}(v_k-v,\omega')+\|v_k-v\|^2_{L^2(D_{2\rho}(x_0))}\right)\to 0\,,
\end{multline*}
and the proof is complete. 
\end{proof}

\begin{proof}[Proof of Lemma \ref{hatH1/2toH1}, Part 2.] \noindent{\it Step 1.} Let us first consider  the case where 
$v\in  \widehat{H}^{1/2}(\omega;\R^m)\cap  C^\infty(\mathbb{R}^n)$. Then $v^\e$ is smooth in $\overline{\mathbb{R}}^{n+1}_+$. 
Let $x_0\in \omega$ and $\rho>0$ such that $D_{2\rho}(x_0)\subset \omega$. We consider a cut-off function $\chi\in  C^\infty(\mathbb{R}^{n+1};[0,1])$ satisfying $\chi=1$ in $B_\rho(x_0)$, and $\chi=0$ in $\mathbb{R}^{n+1}\setminus B_{2\rho}(x_0)$. 
We set $\Phi:=\chi^2 v^\e$ and $\varphi:=\Phi_{|\R^n}= \chi^2v$.  From the harmonicity of $v^\e$ in $\R^{n+1}_+$ we infer that 
$$\int_{\mathbb{R}^{n+1}_+}\nabla v^\e\cdot\nabla\Phi\,\de x =\int_{\mathbb{R}^n} \frac{\partial v^\e}{\partial\nu}\cdot\varphi\, \de x 
=\lim_{\delta\downarrow0}\int_{\mathbb{R}^n}\frac{\big(v(x)-v^\e(x,\delta)\big) \cdot\varphi(x)}{\delta}\,\de x\,.
$$
On the other hand, applying formula \eqref{poisson} to represent $v^\e$ yields
\begin{multline*}
\int_{\mathbb{R}^n}\frac{\big(v(x)-v^\e(x,\delta)\big)\cdot\varphi(x)}{\delta} \,\de x= 
\gamma_n\iint_{\mathbb{R}^n\times\mathbb{R}^n} \frac{(v(x)-v(y))\cdot \varphi(x)}{(|x-y|^2+\delta^2)^{\frac{n+1}{2}}} \,\de x\de y\\[5pt]
=  \frac{\gamma_n}{2}\iint_{D_{2\rho}(x_0)\times D_{2\rho}(x_0)} \frac{(v(x)-v(y))\cdot (\varphi(x)-\varphi(y))}{(|x-y|^2+\delta^2)^{\frac{n+1}{2}}} \,\de x\de y\\
+ \gamma_n\iint_{D_{2\rho}(x_0)\times D^c_{2\rho}(x_0)} \frac{(v(x)-v(y))\cdot (\varphi(x)-\varphi(y))}{(|x-y|^2+\delta^2)^{\frac{n+1}{2}}} \,\de x\de y\,.
\end{multline*}
Using the fact that $v\in \widehat H^{1/2}(D_{2\rho}(x_0);\R^m)$ and $\varphi\in H^{1/2}_{00}(D_{2\rho}(x_0);\R^m)$, 
we derive by dominated convergence that
$$\lim_{\delta\downarrow0}\int_{\mathbb{R}^n}\frac{\big(v(x)-v^\e(x,\delta)\big) \cdot\varphi(x)}{\delta}\,\de x= 
\big\langle (-\Delta)^{\frac{1}{2}} v,\varphi\big\rangle_{D_{2\rho}(x_0)} \,. $$
As in  \eqref{estinormH-1/2fraclap}, it then follows from Cauchy-Schwarz Inequality that   
\begin{equation}\label{noidea1}
\int_{\mathbb{R}^{n+1}_+}\nabla v^\e\cdot\nabla\Phi\,\de x \leq \sqrt{\mathcal E(v,D_{2\rho}(x_0))}\sqrt{\mathcal E(\varphi,D_{2\rho}(x_0))} 
\leq \frac{1}{2}\big(\mathcal E(v,D_{2\rho}(x_0))+\mathcal E(\varphi,D_{2\rho}(x_0))\big)\,.
\end{equation}
Now we estimate  
\begin{align}
\nonumber\int_{\mathbb{R}^{n+1}_+}\nabla v^\e\cdot\nabla\Phi\,\de x &\geq \int_{B_{2\rho}^+(x_0)}\chi^2|\nabla v^\e|^2\,\de x  
-2\int_{B_{2\rho}^+(x_0)} \chi|v^\e|\,|\nabla v^\e| \,|\nabla\chi|\,\de x\\
\nonumber &\geq \frac{1}{2} \int_{B_{2\rho}^+(x_0)}\chi^2|\nabla v^\e|^2\,\de x-2\int_{B_{2\rho}^+(x_0)} |v^\e|^2|\nabla\chi|^2\,\de x\\
\label{noidea2}&\geq \frac{1}{2} \int_{B^+_{\rho}(x_0)}|\nabla v^\e|^2\,\de x-C_\rho\|v^\e\|^2_{L^2(B^+_{2\rho}(x_0))}\,,
\end{align}
for a constant $C_\rho>0$ independent of $v$. Then \eqref{noidea1} and \eqref{noidea2}  yield
\begin{equation}\label{preestiH1loc}
\int_{B^+_{\rho}(x_0)}|\nabla v^\e|^2\,\de x\leq \mathcal E(v,D_{2\rho}(x_0))+\mathcal E(\varphi,D_{2\rho}(x_0))+C_\rho\|v^\e\|^2_{L^2(B^+_{2\rho}(x_0))}\,, 
\end{equation}
and it remains to estimate the second term in the right handside of this inequality. A straightforward computation  yields 
\begin{multline*}
\mathcal E(\varphi,D_{2\rho}(x_0))\leq  2 \mathcal E(v,D_{2\rho}(x_0)) 
+ 4\int_{D_{4\rho}(x_0)}\left(\int_{ D_{2\rho}(x_0)} \frac{|\chi^2(x)-\chi^2(y)|^2}{|x-y|^{n+1}} \,\de x
\right) |v(y)|^2\,\de y\\
+  4\int_{D^c_{4\rho}(x_0)}\left(\int_{D_{2\rho}(x_0)}\frac{1}{|x-y|^{n+1}} \,\de x\right) |v(y)|^2\,\de y\,.
\end{multline*}
Noticing that for $y\in D_{4\rho}(x_0)$, 
$$\int_{ D_{2\rho}(x_0)} \frac{|\chi^2(x)-\chi^2(y)|^2}{|x-y|^{n+1}} \,\de x \leq C_\rho \int_{D_{2\rho}(x_0)} \frac{1}{|x-y|^{n-1}}\,\de x\leq C_\rho\,,$$
and that for $y\in D^c_{4\rho}(x_0)$, 
$$\int_{D_{2\rho}(x_0)}\frac{1}{|x-y|^{n+1}} \,\de x \leq \frac{C_\rho}{(|y-x_0|+1)^{n+1}}\,,$$
we deduce that 
\begin{equation}\label{good}
\mathcal E\big(\varphi,D_{2\rho}(x_0)\big)\leq  2 \mathcal E(v,D_{2\rho}(x_0)) + C_{\rho} \int_{\R^n} \frac{|v(z)|^2}{(|z-x_0|+1)^{n+1}}\,\de z\,.
\end{equation}
Finally, gathering \eqref{preestiH1loc} with \eqref{conclestil2loc} and \eqref{good},   we infer that
$$\int_{B^+_{\rho}(x_0)}|\nabla v^\e|^2\,\de x\leq 3  \mathcal E(v,D_{2\rho}(x_0)) + C_{\rho} \int_{\R^n} \frac{|v(z)|^2}{(|z-x_0|+1)^{n+1}}\,\de z\,.$$
In view of  Lemma \ref{adminHchap}, we have thus proved that 
$$\int_{B^+_{\rho}(x_0)}|\nabla v^\e|^2\,\de x\leq  C_{\rho}  \big( \mathcal E\big(v,D_{2\rho}(x_0)\big) +\|v\|^2_{L^2(D_{2\rho}(x_0))}\big)\,.$$
\vskip5pt

\noindent{\it  Step 2.} Let us now consider an arbitrary $v\in  \widehat{H}^{1/2}(\omega;\R^m)$, and a sequence $\{v_k \}\subset  C^\infty(\R^n;\R^m)$ 
given by Lemma~\ref{approxlemma}.  Let $x_0\in \omega$ and $\rho>0$ such that $D_{3\rho}(x_0)\subset \omega$.  Then,   for $k$ large enough, 
$v_k\in  \widehat{H}^{1/2}(D_{2\rho}(x_0);\R^m)$. Appying Step 1 to $v_k$ we infer that for $k$ large enough, 
$$ \int_{B^+_{\rho}(x_0)}|\nabla v_k^\e|^2\,\de x\leq  C_{\rho}  \big( \mathcal E(v_k,D_{2\rho}(x_0)) +\|v_k\|^2_{L^2(D_{2\rho}(x_0))}\big)\,.$$
By Lemma~\ref{approxlemma}, the right hand side of this inequality is uniformly bounded with respect to~$k$. Since $v_k^\e\to v^\e$ in $L^2(B^+_{\rho}(x_0))$ 
(still by  Lemma~\ref{approxlemma}), we deduce that $v_k^\e\rightharpoonup v^\e$ weakly in $H^1(B^+_{\rho}(x_0))$. By lower semicontinuity 
and (ii) in  Lemma~\ref{approxlemma}, we obtain
$$\int_{B^+_{\rho}(x_0)}|\nabla v^\e|^2\,\de x\leq \liminf_{k\to\infty} \int_{B^+_{\rho}(x_0)}|\nabla v_k^\e|^2\,\de x
\leq C_{\rho}  \big( \mathcal E\big(v,D_{2\rho}(x_0)\big) +\|v\|^2_{L^2(D_{2\rho}(x_0))}\big)\,.$$
Moreover, in view of the arbitrariness of $x_0$, we conclude that 
\begin{equation}\label{time2201}
v^\e\in H^1(B^+_{\rho}(x);\R^m)\quad\text{for all $x\in\omega$ and $\rho>0$ such that $D_{3\rho}(x)\subset \omega$}\,. 
\end{equation}
\vskip3pt

Finally, the conclusion $v^\e\in H^1_{\rm loc}(\R^{n+1}_+\cup\omega;\R^m)$ 
follows from \eqref{time2201} together with a standard covering argument. 
\end{proof}

\begin{proof}[Proof of Lemma \ref{repnormderfraclap}.] By the density of compactly supported smooth funtions in \eqref{densitysmoothH1/200}, 
we may assume without loss of generality that  $\varphi \in \mathscr{D}(\omega;\R^m)$. 
Let $\Omega\subset\R^{n+1}_+$ be an admissible bounded open set such that ${\rm supp}\,\varphi \subset \partial^0\Om$,  and $\overline{\partial^0\Omega}\subset\omega$.
We then consider a smooth extension $\Phi$ of $\varphi$ to $\R^{n+1}_+$ which is compactly supported in $\Omega\cup\partial^0\Om$. 
Let $\{v_k \}\subset  C^\infty(\R^n;\R^m)$ be a sequence 
given by Lemma~\ref{approxlemma}. Then, $v_k\in  \widehat{H}^{1/2}(\partial^0\Om;\R^m)$ for $k$ large enough. Arguing as in the proof of Lemma~\ref{hatH1/2toH1} (Part 2, Step 1), we show that 
$$\left\langle \frac{\partial v_k^\e}{\partial\nu}, \varphi\right\rangle:= \int_{\mathbb{R}^{n+1}_+}\nabla v_k^\e\cdot\nabla\Phi\,\de x 
= \int_{\mathbb{R}^n} \frac{\partial v_k^\e}{\partial\nu}\cdot\varphi\, \de x=\big\langle  (-\Delta)^{\frac{1}{2}} v_k,\varphi\big\rangle_{\partial^0\Om}\,. 
$$
By the proof of Lemma \ref{hatH1/2toH1} (Part 2, Step 2), $v_k^\e\rightharpoonup v^\e$ weakly in $H^1(\Omega)$. Therefore, 
$$
\left\langle \frac{\partial v_k^\e}{\partial\nu}, \varphi\right\rangle= \int_{\Omega}\nabla v_k^\e\cdot\nabla\Phi\,\de x
\mathop{\longrightarrow}\limits_{k\to\infty}   \int_{\Omega}\nabla v^\e\cdot\nabla\Phi\,\de x = \left\langle \frac{\partial v^\e}{\partial\nu}, \varphi\right\rangle\,.
$$
On the other hand, by \eqref{estinormH-1/2fraclap} and Lemma~\ref{approxlemma} we have  
$$\left|\big\langle  (-\Delta)^{\frac{1}{2}} (v_k-v),\varphi\big\rangle_{\partial^0\Om} \right|\leq \|\varphi\|_{H^{1/2}_{00}(\partial^0\Om)}\sqrt{\mathcal{E}(v_k-v,\partial^0\Om)}\mathop{\longrightarrow}\limits_{k\to\infty} 0\,.$$ 
Consequently, 
$$\left\langle \frac{\partial v^\e}{\partial\nu}, \varphi\right\rangle=\lim_{k\to\infty}\left\langle \frac{\partial v_k^\e}{\partial\nu}, \varphi\right\rangle 
=   \lim_{k\to\infty}\big\langle  (-\Delta)^{\frac{1}{2}} v_k,\varphi\big\rangle_{\partial^0\Om} 
= \big\langle  (-\Delta)^{\frac{1}{2}} v,\varphi\big\rangle_{\partial^0\Om} = \big\langle  (-\Delta)^{\frac{1}{2}} v,\varphi\big\rangle_{\omega}\,,$$
and the lemma is proved. 
\end{proof}

\vskip15pt

\section{Appendix B}   									
 
 This purpose of this last appendix is to provide the elliptic regularity results we have been using during the proof of Lemma~\ref{convc1alph}. 

\begin{lemma}\label{supersol}
For $\varepsilon>0$, let $u_\eps\in H^1(B_1^+)$ be the unique (variational) solution of
$$\begin{cases}
-\Delta u_\eps = 1 & \text{in $B_1^+$}\,, \\
u_\eps=1 & \text{on $\partial^+ B_1$}\,,\\[8pt]
\displaystyle \varepsilon\frac{\partial u_\eps}{\partial\nu} + u_\eps =0 & \text{on $D_1$}\,.
\end{cases}
$$
Then $u_\eps\in  C^{0,\alpha}(B_1^+)\cap  C^\infty\big(\overline B_1^+\setminus\partial D_1\big)$ for some $0<\alpha<1$, $u_\eps> 0$ in $\overline B_1^+$, and 
$$u_\eps\leq C\varepsilon\quad\text{ on $D_{1/2}$}\,,$$
for some constant $C>0$ which only depends on the dimension.  
\end{lemma}

\begin{proof}
The H\"older continuity of $u_\eps$ (up to the boundary) follows from the general results  of {\sc Stampacchia} \cite{Stamp1,Stamp2} (see also \cite{Chicc}), while higher order regularity away from $\partial D_1$ follows from standard elliptic regularity theory. The fact that $u_\eps> 0$ in $\overline B_1^+$ is an easy consequence of the maximum principle and the Hopf boundary Lemma.  Indeed, assume that $x_0$ be a minimum point of $u_\eps$ such that $u_\eps(x_0)\leq 0$. By the maximum principle $x_0\in \partial B_1^+$, 
and thus $x_0\in D_1$ since $u_\eps=1$ on $\partial^+B_1$ and $u_\eps$ is continuous on $\overline B^+_1$. Since $u_\eps$ is smooth up to $D_1$ (away from $\partial D_1$), we can apply  the Hopf boundary Lemma to infer that $\partial_\nu u_\eps(x_0)<0$. Then  
$\varepsilon\partial_\nu u_\eps(x_0) + u_\eps(x_0)<0$ which contradicts the equation on $D_1$. 

Let us now consider the unique solution $w\in H^1(B_1^+)$ of 
$$\begin{cases}
-\Delta w = 1 & \text{in $B_1^+$}\,, \\
w=0 & \text{on $\partial B^+_1$}\,.
\end{cases}
$$
It is well known that $w\in  C^{0,1}(B_1^+)\cap  C^\infty\big(\overline B_1^+ \setminus\partial D_1\big)$ by convexity of $B_1^+$ (see {\it e.g.} \cite{Fr}). 
We set $v_\eps:=u_\eps-w$, so that $v_\eps\in  C^{0,\alpha}(B_1^+)\cap  C^\infty\left(\overline B_1^+ \setminus\partial D_1\right)$ solves 
$$\begin{cases}
\Delta v_\eps = 0 & \text{in $B_1^+$}\,, \\
v_\eps=1 & \text{on $\partial^+ B_1$}\,,\\[8pt]
\displaystyle \varepsilon\frac{\partial v_\eps}{\partial\nu} + v_\eps = -\varepsilon\frac{\partial w}{\partial \nu} & \text{on $D_1$}\,.
\end{cases}
$$
Setting 
$$\kappa:= \left\|\frac{\partial w}{\partial \nu}\right\|_{L^\infty(D_{1})}\,,$$
we observe that $\kappa$ only depends on the dimension. Next we define $\bar u_\eps\in H^1(B_1^+)$ as the unique (variational) solution of 
$$\begin{cases}
\Delta \bar u_\eps = 0 & \text{in $B_{1}^+$}\,, \\
\bar u_\eps=1 & \text{on $\partial^+ B_{1}$}\,,\\[8pt]
\displaystyle \varepsilon\frac{\partial \bar u_\eps}{\partial\nu} + \bar u_\eps = 0 & \text{on $D_{1}$}\,.  
\end{cases}
$$
As previously, $\bar u_\eps \in  C^{0,\beta}(B_1^+)\cap  C^\infty\big(\overline B_1^+\setminus\partial D_1\big)$ for some $0<\beta<1$, and by the maximum principle and the Hopf boundary Lemma, we have 
\begin{equation}\label{baru01}
0\leq \bar u_\eps \leq 1 \qquad \text{ in $\overline B^+_{1}$}\,.
\end{equation}

Observe that
\begin{equation}\label{majvbyubar}
v_\eps\leq \bar u_\eps +\kappa\varepsilon \qquad \text{in $\overline B^+_{1}$}\,.
\end{equation}
Indeed, consider the function $h:= \bar u_\eps +\kappa\varepsilon -v_\eps$. Then $h$ satisfies 
$$\begin{cases}
\Delta h = 0 & \text{in $B_{1}^+$}\,, \\
h\geq 0 & \text{on $\partial^+ B_{1}$}\,,\\[8pt]
\displaystyle \varepsilon\frac{\partial h}{\partial\nu} + h \geq 0 & \text{on $D_{1}$}\,,
\end{cases}
$$
so that $h\geq 0$ in   $\overline B^+_{1} $, still by the maximum principle and  the Hopf boundary Lemma. 

Finally, in view of \eqref{baru01} and according to \cite[Theorem 6.26]{GT},  
$$\|\nabla \bar u_\eps\|_{L^{\infty}(B^+_{1/2})}\leq C\,,$$
for a constant $C>0$ which only depends on the dimension. In particular $|\partial_\nu \bar u_\eps|\leq C$ on $D_{1/2}$, and thus 
\begin{equation}\label{controlubarbound}
\bar u_\eps \leq C\varepsilon  \quad\text{ on $D_{1/2}$}\,.
\end{equation}
Since $v_\eps=u_\eps$ on $D_1$, the conclusion now follows from \eqref{majvbyubar} and \eqref{controlubarbound}. 
\end{proof}

We conclude this appendix with a standard elliptic boundary estimate for a linear Neumann problem. As showed in \cite[proof of Lemma 2.2]{CSM}, such (classical) estimate can be easily obtained from usual boundary regularity by considering  the auxiliary function 
$v(x',x_{n+1}):=\int_0^{x_{n+1}}u(x',t)\,\de t$ which satisfies a Poisson equation in $B_1^+$ with an homogeneous Dirichlet boundary condition on $D_1$. 

\begin{lemma}\label{regNeum}
Given $0<\alpha<1$, let $f\in  C^{0,\alpha}(B_1^+)$, $g\in  C^{0,\alpha}(D_1)$, and  
$u\in  C^{2}(B_1^+)\cap  C^{1,\alpha}(B_1^+)$ solving  
$$\begin{cases} 
-\Delta u= f & \text{in $B_1^+$}\,,\\[8pt]
\displaystyle \frac{\partial u}{\partial \nu}= g & \text{on $D_1$}\,.
\end{cases}
$$ 
Then, 
$$\|u\|_{ C^{1,\alpha}(B_r^+)}\leq C_{r,\alpha}(\|f\|_{ C^{0,\alpha}(B_1^+)}+\|g\|_{ C^{0,\alpha}(D_1)}+\|u\|_{L^\infty(B_1^+)} )$$ 
for every $0<r<1$ and a constant $C_{r,\alpha}>0$ which only depends on $n$, $\alpha$, $r$. 
\end{lemma}


\vskip15pt

 \noindent{\it Aknowledgements.} 
 The authors wish to thank Francesca Da Lio and Tristan Rivi\`ere for their suggestions and comments during the preparation of the paper.  
  V.M. is supported by the {\it Agence Nationale de la Recherche} under Grant No. ANR 10-JCJC 0106; 
 Y. S. is supported by the {\it Agence Nationale de la Recherche} under Grant No. ANR 12-BS01-0013.

\vskip15pt



\begin{thebibliography}{68}
%
\bibitem{ABScr} {\sc G. Alberti, G. Bouchitt{\'e}, P. Seppecher} : Un r\'esultat de perturbations singuli\`eres 
avec la norme $H^{1/2}$, {\it C. R. Acad. Sci. Paris S\'er. I Math.} {\bf 319} (1994), 333--338. 
%
\vskip3pt
\bibitem{ABS} {\sc G. Alberti, G. Bouchitt{\'e}, P. Seppecher} : Phase transition with the line-tension effect,
{\it Arch. Rational Mech. Anal.} {\bf144} (1998), 1--46.
%
\vskip3pt
\bibitem{AFP} {\sc L. Ambrosio, N. Fusco, D. Pallara} : {\it Functions of Bounded Variation and Free Discontinuity Problems}, Oxford University Press, New York (2000).
%
\vskip3pt
\bibitem{AS} {\sc L. Ambrosio, H.-M. Soner} : A measure-theoretic approach to higher codimension mean curvature flows,
{\it Ann. Scuola Norm. Sup. Pisa Cl. Sci. (4)}  {\bf 25} (1997),  27--49. 
%
\vskip3pt
\bibitem{Bal} {\sc A. Baldes} : Harmonic mappings with a partially free boundary, {\it Manuscripta Math.}  {\bf 40} (1982), 255--275. 
%
\vskip3pt
\bibitem{BGS} {\sc V. Banica, M. d. M. Gonz\'alez, M. S\'aez} : Some constructions for the fractional Laplacian on noncompact manifolds, preprint \texttt{ 	arXiv:1212.3109}.
%
\vskip3pt
\bibitem{BMRS} {\sc L. Berlyand, P. Mironescu, V. Rybalko, E. Sandier} : Minimax critical points in Ginzburg-Landau problems with semi-stiff boundary conditions: Existence and bubbling, to appear in {\it Comm. Partial Differential Equations} (preprint available at  
\texttt{http://hal.archives-ouvertes.fr/hal-00747639}).
%
\vskip3pt
\bibitem{Bet} {\sc F. Bethuel} : On the singular set of stationary harmonic maps, {\it Manuscripta Math.} {\bf 78} (1993), 417--443. 
%
\vskip3pt
\bibitem{BBH} {\sc F. Bethuel, H. Brezis, F. H{\'e}lein} : {\it Ginzburg-{L}andau vortices}, Progress in Nonlinear Differential Equations and their Applications, Birkh\"auser Boston Inc., Boston MA (1994). 
%
\vskip3pt
\bibitem{BCL} {\sc H. Brezis, J.-M. Coron, E.-H. Lieb} : Harmonic maps with defects, {\it Comm. Math. Phys.} {\bf 107} (1986), 649--705. 
%
\vskip3pt
\bibitem{BN1} {\sc H. Brezis, L. Nirenberg} : Degree theory and BMO. I. Compact manifolds without boundaries, {\it Selecta
Math. (N.S.)} {\bf 1} (1995), 197--263. 
%
\vskip3pt
\bibitem{BN2} {\sc H. Brezis, L. Nirenberg} : Degree theory and BMO. II. Compact manifolds with boundaries, {\it Selecta
Math. (N.S.)} {\bf 2} (1996), 309--368.
%
\vskip3pt 
\bibitem{CC} {\sc X. Cabr\'e, E. Cinti} : Energy estimates and 1-{D} symmetry for nonlinear equations
              involving the half-{L}aplacian, {\it Discrete Contin. Dyn. Syst.}  {\bf 28} (2010), 1179--1206. 
              %
\vskip3pt               
\bibitem{CC2} {\sc X. Cabr\'e, E. Cinti} :  Sharp energy estimates for nonlinear fractional diffusion equations, to appear in {\it Calc. Var. Partial Differential Equations}.                   
%
\vskip3pt 
\bibitem{CS1} {\sc X. Cabr\'e, Y. Sire} : Nonlinear equations for fractional Laplacians I: Regularity, maximum principles, and Hamiltonian estimates, to appear in 
{\it Ann. Inst. H. Poincar\'e Anal. Non Lin\'eaire}. 
%
\vskip3pt 
\bibitem{CS2} {\sc X. Cabr\'e, Y. Sire} : Nonlinear equations for fractional Laplacians II: existence, uniqueness, and qualitative properties of solutions, to appear in {\it Trans. Amer. Math. Soc}. 
%
\vskip3pt
\bibitem{CSM} {\sc X. Cabr\'e, J. Sol\`a-Morales} : Layer solutions in a halfspace for boundary reactions, {\it Comm. Pure Appl. Math.} {\bf 58} (2005), 1678--1732.
%
\vskip3pt
\bibitem{CRS} {\sc L. Caffarelli, J.-M. Roquejoffre, O. Savin} : Nonlocal minimal surfaces, {\it Comm. Pure Appl. Math.} {\bf 63} (2010), 1111--1144.
%
\vskip3pt
\bibitem{CaffSil} {\sc L. Caffarelli, L. Silvestre} : An extension problem related to the fractional Laplacian, {\it Comm. Partial Differential Equations} {\bf 32} (2007), 1245--1260.
%
\vskip3pt
\bibitem{CS} {\sc Y.-M. Chen, M. Struwe} :  Existence and partial regularity results for the heat flow for harmonic maps, {\it Math.~Z.} {\bf 201} (1989), 83--103.
%
\vskip3pt
\bibitem{Chicc} {\sc M. Chicco, M. Venturino} : H\"older regularity for solutions of mixed boundary value
              problems containing boundary terms, {\it Boll. Unione Mat. Ital. Sez. B Artic. Ric. Mat. (8)} {\bf 9} (2006), 267--281. 
%
\vskip3pt
\bibitem{DaLRi} {\sc F. Da Lio, T. Rivi{\`e}re} : Three-term commutator estimates and the regularity of {$\frac12$}-harmonic maps into spheres, {\it Anal. PDE} 
 {\bf 4}  (2011), 149--190. 
%
\vskip3pt
\bibitem{DaLRi2} {\sc F. Da Lio, T. Rivi{\`e}re} : Sub-criticality of non-local Schr\"odinger systems with antisymmetric potential and applications to half harmonic maps, 
{\it Advances in Math.} {\bf 227} (2011), 1300--1348.
%
\vskip3pt
\bibitem{DaLRi3} {\sc F. Da Lio, T. Rivi{\`e}re} : Fractional harmonic maps and free boundaries problems, in preparation. 
%
\vskip3pt
\bibitem{DFV} {\sc S. Dipierro, A. Figalli, E. Valdinoci} : Strongly nonlocal dislocation dynamics in crystals, preprint \texttt{arXiv:1311.3549}.
%
\vskip3pt
\bibitem{Hitch} {\sc E. Di Nezza, G. Palatucci, E. Valdinoci} : Hitchhiker's guide to the fractional Sobolev spaces, {\it Bull. Sci. Math.} {\bf 136} (2012), 521--573. 
%
\vskip3pt
\bibitem{DPV} {\sc S. Dipierro, G. Palatucci, E. Valdinoci} : Dislocation dynamics in crystals: a macroscopic theory in a fractional Laplace setting, to appear in {\it Comm. Math. Phys.}
%
\vskip3pt
\bibitem{DG} {\sc F. Duzaar, J.-F. Grotowski} : A mixed boundary value problem for energy minimizing harmonic maps, {\it Math. Z.}  {\bf 221} (1996), 153--167. 
%
\vskip3pt
 \bibitem{DS} {\sc F. Duzaar, K. Steffen} : A partial regularity theorem for harmonic maps at a free boundary, {\it Asymptotic Anal.} {\bf 2} (1989), 299--343. 
 %
 \vskip3pt
\bibitem{DS2} {\sc F. Duzaar, K. Steffen} : An optimal estimate for the singular set of a harmonic map in the free boundary,
{\it J. Reine Angew. Math.}  {\bf 401} (1989), 157--187. 
 %
 \vskip3pt 
 \bibitem{Ev} {\sc L.-C. Evans} : Partial regularity for stationary harmonic maps into spheres, {\it Arch. Rat. Mech. Anal.} {\bf 116} (1991), 101--113.
  %
\vskip3pt
 \bibitem{EvGa} {\sc L.-C. Evans, R.-F. Gariepy} : {\it Measure theory and fine properties of functions}, Studies in Advanced Mathematics, CRC Press,  Boca Raton FL (1992). 
 %
\vskip3pt
 \bibitem{FS} {\sc A. Fraser, R. Schoen} : {\it Minimal surfaces and eigenvalue problems}, preprint \texttt{arXiv:1304.0851}. 
 %
\vskip3pt
\bibitem{Fr} {\sc S.-J. Fromm} : Potential space estimates for {G}reen potentials in convex domains,
{\it Proc. Amer. Math. Soc.} {\bf 119} (1993), 225--233. 
 %
\vskip3pt
\bibitem{GM} {\sc A. Garroni, S. M{\"u}ller} : A variational model for dislocations in the line tension
              limit, {\it Arch. Ration. Mech. Anal.}  {\bf 181} (2006), 535--578. 
 %
\vskip3pt
\bibitem{GT} {\sc D. Gilbarg, N.-S. Trudinger} :  {\it Elliptic Partial Differential Equations of Second Order}, Classics in Mathematics, Springer-Verlag, Berlin (2001). 
%
\vskip3pt
\bibitem{Gonz} {\sc M. d. M. Gonz\'alez} : Gamma convergence of an energy functional related to the fractional Laplacian, {\it Calc. Var. Partial Differential Equations} 
{\bf 36} (2009), 173--210. 
%
\vskip3pt
\bibitem{GonzMon} {\sc M. d. M. Gonz\'alez, R. Monneau} : Slow motion of particle systems as a limit of a reaction-diffusion equation with half-Laplacian in dimension one, 
{\it Discrete Contin. Dyn. Syst.} {\bf 32} (2012),1255--1286.
%
\vskip3pt
\bibitem{GSS} {\sc M. d. M. Gonz\'alez, M. S\'aez, Y. Sire} : Layer solutions for the fractional Laplacian on hyperbolic space: existence, uniqueness and qualitative properties, to appear in {\it Ann. Mat. Pura Appl.}
%
\vskip3pt
\bibitem{G} {\sc P. Grisvard} : {\it Elliptic problems in nonsmooth domains}, Monographs and Studies in Mathematics, Pitman, Boston MA (1985). 
%
\vskip3pt
\bibitem{GJ} {\sc R. Gulliver, J. Jost} : Harmonic maps which solve a free-boundary problem, {\it J. Reine Angew. Math.} 
 {\bf 381} (1987), 61--89.
%
\vskip3pt
\bibitem{Ham} {\sc R.-S. Hamilton} : {\it Harmonic maps of manifolds with boundary}, Lecture Notes in Mathematics, Springer-Verlag,
Berlin (1975).
%
\vskip3pt
\bibitem{HL} {\sc R. Hardt, F.-H. Lin} : Partially constrained boundary conditions with energy minimizing mappings, {\it Comm. Pure Appl. Math.}  {\bf 42} (1989), 309--334. 
%
\vskip3pt
\bibitem{Hel} {\sc F. H{\'e}lein} : R\'egularit\'e des applications faiblement harmoniques entre une surface et une vari\'et\'e riemannienne,
{\it C. R. Acad. Sci. Paris S\'er. I Math.}  {\bf 312} (1991), 591--596. 
%
\vskip3pt
\bibitem{HT} {\sc J.-E. Hutchinson, Y. Tonegawa} : Convergence of phase interfaces in the van der {W}aals-{C}ahn-{H}illiard theory,
{\it Calc. Var. Partial Differential Equations} {\bf 10} (2000), 49--84.
%
\vskip3pt
\bibitem{K1} {\sc M. Kurzke} :  A nonlocal singular perturbation problem with periodic well
              potential, {\it ESAIM Control Optim. Calc. Var.} {\bf 12} (2006), 52--63. 
%
\vskip3pt
\bibitem{K2} {\sc M. Kurzke} : Boundary vortices in thin magnetic films, {\it Calc. Var. Partial Differential Equations} {\bf 26} 
(2006), 1--28.
%
\vskip3pt
\bibitem{LL} {\sc E.-H. Lieb, M. Loss} : {\it Analysis}, Graduate Studies in Mathematics, American Mathematical Society, Providence RI (2001).
%
\vskip3pt
\bibitem{Lin} {\sc F.-H. Lin} : Gradient estimates and blow-up analysis for stationary harmonic maps, {\it Annals of Math.} {\bf 149} (1999), 785--829. 
%
\vskip3pt
\bibitem{LW1} {\sc F.-H. Lin, C.-Y. Wang} : Harmonic and quasi-harmonic spheres, {\it Comm. Anal. Geom.}  {\bf 7} (1999), 397--429. 
%
\vskip3pt
\bibitem{LW2} {\sc F.-H. Lin, C.-Y. Wang} : Harmonic and quasi-harmonic spheres. {II}, {\it Comm. Anal. Geom.} {\bf 10} (2002), 341--375.
%
\vskip3pt
\bibitem{LW3} {\sc F.-H. Lin, C.-Y. Wang} : Harmonic and quasi-harmonic spheres. {III}. {R}ectifiability of the parabolic defect measure and generalized varifold flows, {\it Ann. Inst. H. Poincar\'e Anal. Non Lin\'eaire}  {\bf 19} (2002), 209--259. 
%
\vskip3pt
\bibitem{MilSir} {\sc V. Millot, Y. Sire} : Asymptotics for a fractional Allen-Cahn equation  and stationary nonlocal minimal surfaces, in preparation. 
%
\vskip3pt
\bibitem{MR} {\sc A.-V. Milovanov, J.-J. Rasmussen} : Fractional generalization of the Ginzburg-Landau equation: an unconventional approach to critical phenomena in complex media, {\it Physics Letters A} {\bf 337}Ê(2005), 75--80. 
%
\vskip3pt
\bibitem{MirPis} {\sc P. Mironescu, A. Pisante} : A variational problem with lack of compactness for {$H^{1/2}(S^1;S^1)$} maps of prescribed degree, 
{\it J. Funct. Anal.} {\bf 217} (2004), 249--279.
%
\vskip3pt
\bibitem{Mod} {\sc L. Modica} : The gradient theory of phase transitions and the minimal interface criterion,
 {\it Arch. Rational Mech. Anal.}  {\bf 98} (1987), 123--142. 
%
\vskip3pt
\bibitem{Mos} {\sc R. Moser} : Intrinsic semiharmonic maps, {\it J. Geom. Anal.}  {\bf 21} (2011), 588--598.
%
\vskip3pt
\bibitem{PSV} {\sc G. Palatucci, O. Savin, E. Valdinoci} : Local and global minimizers for a variational energy involving a fractional norm, to appear in {\it Ann. Mat. Pura Appl.}
%
\vskip3pt
\bibitem{Pr} {\sc D. Preiss} : Geometry of measures in $\R^n$: distributions, rectifiability, and densities,  {\it Ann. of
Math.} {\bf 125} (1987), 537--643.
%
\vskip3pt
\bibitem{Riv} {\sc T. Rivi{\`e}re} : Everywhere discontinuous harmonic maps into spheres, {\it Acta Math.} {\bf 175} (1995), 197--226.
%
\vskip3pt
\bibitem{Riv2} {\sc T. Rivi{\`e}re} : Line vortices in the {${\rm U}(1)$}-{H}iggs model,
{\it ESAIM Contr\^ole Optim. Calc. Var.}  {\bf 1} (1995/96), 77--167. 
%
\vskip3pt
\bibitem{SandSer} {\sc E. Sandier, S. Serfaty} : {\it Vortices in the magnetic {G}inzburg-{L}andau model}, 
Progress in Nonlinear Differential Equations and their Applications, Birkh\"auser Boston Inc., Boston MA (2007). 
%
\vskip3pt
\bibitem{Sav} {\sc G. Savar{\'e}} : Regularity and perturbation results for mixed second order elliptic problems,
{\it Comm. Partial Differential Equations} {\bf 22} (1997), 869--899. 
%
\vskip3pt
\bibitem{SV1} {\sc O. Savin, E. Valdinoci} : $\Gamma$-convergence for nonlocal phase transitions,
{\it Ann. Inst. H. Poincar\'e Anal. Non Lin\'eaire} {\bf 29} (2012), 479 - 500. 
%
\vskip3pt
\bibitem{SV2} {\sc O. Savin, E. Valdinoci} : Density estimates for a variational model driven by the Gagliardo norm, to appear in {\it J. Math. Pures Appl}. 
%
\vskip3pt
\bibitem{Sch} {\sc C. Scheven} : Partial regularity for stationary harmonic maps at a free boundary, {\it Math. Z.}, {\bf 253} (2006), 135--157. 
%
\vskip3pt
\bibitem{SU} {\sc R. Schoen, K. Uhlenbeck} : A regularity theory for harmonic maps, {\it J. Diff. Geom.} {\bf 17} (1982), 307--335.  
%
\vskip3pt
\bibitem{SerV} {\sc R. Servadei, E. Valdinoci} : Weak and viscosity solutions of the fractional Laplace equation, {\it Publ. Math.} {\bf 58} (2014), 133--154.
%
\vskip3pt              
\bibitem{Sham} {\sc E. Shamir} : Regularization of mixed second-order elliptic problems, {\it Isr. J. of Math.} {\bf 6} (1968), 150--168.
%
\vskip3pt      
\bibitem{Sim2} {\sc L. Simon} : {\it Lectures on geometric measure theory}, Proceedings of the Centre for Mathematical Analysis, Australian National University, Centre for Mathematical Analysis, Canberra  (1983).
%
\vskip3pt      
\bibitem{Sim} {\sc L. Simon} : {\it Theorems on regularity and singularity of energy minimizing maps}, Lectures in Mathematics ETH Z\"urich,
Birkh\"auser Verlag, Basel (1996).
%
\vskip3pt     
\bibitem{SV} {\sc Y. Sire, E. Valdinoci} : Fractional {L}aplacian phase transitions and boundary
              reactions: a geometric inequality and a symmetry result, {\it J. Funct. Anal.} {\bf 256} (2009), 1842--1864. 
%
\vskip3pt
\bibitem{Stamp1} {\sc G. Stampacchia} : Problemi al contorno ellitici, con dati discontinui, dotati di
              soluzionie h\"olderiane, {\it Ann. Mat. Pura Appl. (4)} {\bf 51} (1960), 1--37. 
%
\vskip3pt
\bibitem{Stamp2} {\sc G. Stampacchia} : Le probl\`eme de {D}irichlet pour les \'equations elliptiques
              du second ordre \`a coefficients discontinus, {\it Ann. Inst. Fourier (Grenoble)}  {\bf 15} (1965), 189--258. 
%
\vskip3pt
\bibitem{ST} {\sc P.R. Stinga, J.L. Torrea} : Extension problem and Harnack's inequality for some fractional operators, {\it Comm. Partial Differential Equations}
{\bf 35} (2010), 2092--2122
%
\vskip3pt
\bibitem{struw} {\sc M. Struwe} : On a free boundary problem for minimal surfaces, {\it Invent. Math.}  {\bf 75} (1984), 547--560. 
%
\vskip3pt
\bibitem{TZ} {\sc V.-E. Tarasov, G.-M. Zaslavsky} : Fractional Ginzburg-Landau equation for fractal media, 
{\it Physica~A: Statistical Mechanics and its Applications} {\bf 354} (2005), 249--261.
%
\vskip3pt
\bibitem{W} {\sc C.-Y. Wang} : Limits of solutions to the generalized {G}inzburg-{L}andau
              functional, {\it Comm. Partial Differential Equations} {\bf 27} (2002), 877--906. 
%
\vskip3pt
\bibitem{WZ} {\sc H. Weitzner, G.-M. Zaslavsky} : Some applications of fractional equations, 
{\it Commun. Nonlinear Sci. Numer. Simul.} {\bf 8} (2003), 273--281. 
%
\vskip3pt
\bibitem{Zi} {\sc W.-P. Ziemer} : {\it Weakly differentiable functions}, Graduate Texts in Mathematics, Springer-Verlag, New York (1989). 
\end{thebibliography}
\end{document}